\documentclass{amsart}

\usepackage{amssymb,amsmath,graphicx}

\usepackage{xcolor}

\usepackage[shortlabels]{enumitem}

\usepackage[utf8]{inputenc}
\usepackage[T1]{fontenc}

\newtoks\prt

\numberwithin{equation}{section}

\newtheorem{thm}{Theorem}[section]

\newtheorem{lemma}[thm]{Lemma}

\newtheorem{prop}[thm]{Proposition}

\newtheorem{cor}[thm]{Corollary}

\newtheorem{theoreml}{Theorem}[section]

\newcommand{\thistheoremname}{}
\newtheorem*{thm*}{\thistheoremname}
\newenvironment{thmx*}[1]
  {\renewcommand{\thistheoremname}{#1}%
   \begin{thm*}}
  {\end{thm*}}

\theoremstyle{definition}

\newtheorem{remark}[thm]{Remark}

\newtheorem{definition}[thm]{Definition}

\def\eqn#1$$#2$${\begin{equation}\label#1#2\end{equation}}

\def\A{\mathcal A}

\def\C{\mathcal C}

\def\F{\mathcal F}

\def\T{\mathcal T}

\def\e{e^\ast}

\def\s{u^\ast}

\def\t{v^\ast}

\def\M{\mathcal M}

\def\P{\mathcal P}

\def\ep{\varepsilon}

\def\en{\mathbb N}

\def\el{\mathbb L}

\def\er{\mathbb R}

\def \id {\operatorname{id}}

\def \ext {\operatorname{ext}}

\def\supp{\operatorname{supp}}

\def\span{\operatorname{span}}

\def \reg {\partial _{\kern1pt\text{reg}}}

\newcommand{\norm}[1]{\left\|#1\right\|}

\newcommand{\abs}[1]{\left| #1  \right|}

\newcommand{\olb}[1]{\overline{\boldsymbol{#1}}}

\newcommand{\rstr}{\restriction}
\newcommand{\concat}{{^\smallfrown}}
\newcommand{\emptyseq}{\epsilon}

\newcommand{\TT}{\mathfrak{T}}

\newcommand{\eps}{\varepsilon}

\newcommand{\eodef}{\hfill$\dashv$}

\title{Complementability of separable spaces $\mathcal{C}(K)$ in Banach spaces}
\author[J.~Rondo\v s]{Jakub Rondo\v s}
\address{Department of Mathematics, Faculty of Electrical Engineering, Czech Technical University in Prague, Technicka 2,
166 27 Prague 6, Czech Republic.}
\email{jakub.rondos@gmail.com}

\author[D.~Sobota]{Damian Sobota}
\address{Kurt G\"odel Research Center, Department of Mathematics, Vienna University, Kolingasse 14--16, 1090 Vienna, Austria.}
\email{damian.sobota@univie.ac.at}

\thanks{J. Rondo\v{s} was supported by the Austrian Science Fund (FWF), grant I~5918-N, and by the Czech Science Foundation (GA\v{C}R), grant GA23-04776S. D. Sobota was supported by the Austrian Science Fund (FWF),  grant ESP~108-N}

\subjclass[2020]{Primary: 46E15, 46B03, 46B04. Secondary: 47B38, 54A25, 54C35.}

\keywords{$C(K)$-space, Banach space of continuous functions, isometric embedding, isomorphic embedding, complemented subspace, tree}

\begin{document}
\begin{abstract}
For a metric compact space $L$ and a Banach space $E$, we provide a characterization of the complementability of the Banach space $\mathcal{C}(L)$ of continuous functions on $L$ inside $E$ in terms of the existence of a certain tree in the product $E \times E^*$, based on new descriptions of the Banach spaces $\mathcal{C}([1, \omega^{\alpha}])$ for countable ordinal numbers $\alpha$ and $\mathcal{C}(2^{\omega})$. Applying this general result in the case where $E=\mathcal{C}(K)$ for some compact space $K$, 
we further obtain a characterization of the existence of a positively $1$-complemented positively isometric copy of $\mathcal{C}(L)$ inside $\mathcal{C}(K)$ in terms of the topology of $K$ and the space of probability Radon measures on $K$. In the process, we also prove a variant of the classical Holszty\'{n}ski theorem for isometric embeddings onto complemented subspaces.
\end{abstract}

\maketitle

\section{Introduction}

For a compact space $L$, by $\C(L)$ we denote as usual the Banach space of all continuous real-valued functions on $L$ endowed with the supremum norm. We recall a standard fact stating that $\C(L)$ is separable if and only if $L$ is metrizable. The first main purpose of this paper is, given a metrizable compact space $L$ and a Banach space $E$, to provide a characterization of when the space $\C(L)$ is complemented (isometrically or isomorphically) in $E$. This we achieve in several steps. 

First, we develop in Section \ref{sec:new_descriptions} a useful isometric description of spaces $\C(L)$, where $L$ is countable or $L$ is the Cantor space $2^{\omega}$, in terms of certain spaces of real-valued functions on special trees on the first infinite ordinal number $\omega$, see Definitions \ref{def:tree} and \ref{def:tree_space} and Theorems \ref{A_lambda is C(omega^alpha)} and \ref{A_Lambda is C(2^omega)}. Similar ideas, but not identical, have appeared in the works of Bourgain \cite{Bourgain1979}, Dilworth \textit{et al.} \cite{Dilworth-et-al}, Rosenthal \cite{RosenthalC(K)}, etc.---the main difference in our approach when compared to the works of the above authors is that their spaces were defined as the completions of certain normed spaces consisting of elements supported by finite trees, while our spaces are described in a totally explicit form with an aid of appropriately chosen Markushevich bases, which is crucial for our purposes. 

Then, in Section \ref{sec:ctcf}, for a countable successor ordinal $\alpha$ (resp. $\infty$) and $m\in\mathbb{N}$, we introduce a notion of an \emph{(isometric) projectional tree of rank $\alpha$ and order $m$} in a given Banach space $E$, being a suitably constructed tree-like structure in the product $E \times E^*$, see Definition \ref{def:ctcf} for details. Using this notion, we prove our first main result, which characterizes the existence of a complemented copy of the separable space $\C(L)$ in $E$ for $L$ countable or $L=2^\omega$.

\begin{theoreml}\label{thm:mainA}
Let $E$ be a Banach space.

\vspace{4pt}
\noindent (1) For a countable ordinal $\alpha$ and $m\in\en$, the following assertions are equivalent: 
    \begin{enumerate}[(i)]
        \item $\C([1, \omega^\alpha m])$ is isometric to a complemented subspace of $E$,
        \item $E$ admits an isometric projectional tree of rank $\alpha+1$ and order $m$.
    \end{enumerate}

\vspace{4pt}
\noindent (2) The following assertions are equivalent:
    \begin{enumerate}[(i)]
        \item $\C(2^\omega)$ is isometric to a complemented subspace of $E$,
        \item $E$ admits an isometric projectional tree of rank $\infty$ and order $1$.
    \end{enumerate}
\end{theoreml}



The isomorphic version of Theorem \ref{thm:mainA} for \textit{all} metric compact spaces $L$, obtained with an aid of the classical classification results of Bessaga and Pe\l czy\'nski \cite{BessagaPelcynski_classification} and Miljutin \cite[Theorem 5.6]{Pelczynski_book_1968} and involving the Cantor--Bendixson height $ht(L)$ of $L$ (see Section \ref{sec:prelim}), reads as follows.

\begin{theoreml}\label{thm:mainB}
Let $E$ be a Banach space and $L$ be a metric compact space. Then, the following assertions are equivalent:

\begin{enumerate}[(i)]
    \item $\C(L)$ is isomorphic to a complemented subspace of $E$,
    \item $E$ admits a projectional tree of rank $ht(L)$ and order $1$.
\end{enumerate}
\end{theoreml}

For the proofs of the above two main results, see Theorems \ref{main characterization-zerodim} and \ref{main characterization}. Note that Theorems \ref{thm:mainA} and \ref{thm:mainB} should be particularly useful when dealing with the complementability of separable spaces $\C(L)$ in various classes of nonseparable Banach spaces $E$, since, in general, a construction of an appropriate projectional tree in the product $E\times E^*$ seems to be a much easier task than to construct a complemented copy of a separable space $\C(L)$ directly in $E$, as projectional trees are built by natural recursive processes (see also, e.g., Remarks \ref{rem:tree sufficient 2} and \ref{rem:tree sufficient 1}). Besides, in the realm of \textit{separable} Banach spaces the complementabiliy of spaces $\C(L)$ for metric $L$ is already quite well understood. First, each copy of $\C([1, \omega])\simeq c_0$ in a separable Banach space $E$ is complemented (Sobczyk \cite{Sobczyk_c0}) and, notably, $c_0$ is the only (up to isomorphism) separable space with such a property (Zippin \cite{Zippin1977_separable}). 
Concerning other separable spaces $\C(L)$, the following weaker fact holds in general: if a separable Banach space $E$ contains a copy of a separable space $\C(L)$, then $E$ contains a complemented (possibly different) copy of $\C(L)$ (Pe\l czy\'nski \cite{Pelczynski_separable}). 
On the other hand, the case of non-separable Banach spaces $E$ is less clear.

Let us also mention that Theorem \ref{thm:mainA} might be viewed as a natural extension of the characterization of the complementability of $c_0$ in Banach spaces obtained by Schlumprecht \cite{schlumprecht_phd_thesis} (see also \cite[Theorem 1.1.2]{cembranos2006banach}). Moreover, as explained at the beginning of Section \ref{sec:ctcf}, one of the main motivations for introducing and studying projectional trees in Banach spaces was the standard method of constructing complemented isometric copies of the Banach spaces $c_0$ and $\C([1,\omega])$ in spaces $\C(K)$ for $K$ containing nontrivial convergent sequences. As is commonly known, this is closely related to the notion of a Grothendieck space: recall that a Banach space $E$ is \emph{Grothendieck} if every weak$^*$ convergent sequence in the dual space $E^*$ is weakly convergent. For a general discussion and examples of Grothendieck Banach spaces, see  \cite{Diestel_Grothendieck} or \cite{Gonzalez_Kania}. The class of compact spaces $K$ whose Banach spaces $\C(K)$ are Grothendieck include all extremally disconnected compact spaces (Grothendieck \cite{Grothendieck}), or, more generally, all F-spaces (Seever \cite{Seever}); on the other hand, the space $c_0$ and all separable spaces $\C(K)$ are not Grothendieck. By the results of Cembranos \cite{Cembranos_complemented_c0} and Schachermayer \cite{Schachermayer_Grothendieck_C(K)}, for a compact space $K$, the space $\C(K)$ is Grothendieck if and only if $\C(K)$ does not contain complemented (isometric) copies of $c_0$ (cf. \cite[Chapter 19]{KKLPS}). In the proof of the ``only if'' part one constructs a sequence $(f_n,\mu_n)_{n\in\omega}$ of pairs in the product $\C(K)\times\C(K)^*$ which behaves in a somewhat biorthogonal way, similarly to our projectional trees. Thus, Theorems \ref{thm:mainA} and \ref{thm:mainB} may be in a certain sense considered as generalizations of Cebranos and Schachermayer's result and so, \textit{a fortiori}, the property of the nonexistence of projectional trees in Banach spaces generalizes the Grothendieck property.  

\medskip

The second main goal of our paper is to demonstrate the utility of the above notions and results in the case when $E$ is again some, generally nonseparable, space of continuous functions $\C(K)$. The study of complemented subspaces of $\C(K)$ spaces is currently an active field of research, with many important steps forward done in the recent years, see e.g.  \cite{Acuaviva_ComplementedC(K)Candido}, \cite{CORREA_compact_lines_sobczyk},\ \cite{JohnsonKaniaSchchtman2016complementedcountablesuport}, 
\cite{Kakol_Sobota_Zdomskyy_c0_in_C(K_times_L)}, \cite{koszmider2005decompositions}, \cite{koszmider2021banach}, \cite{koszmider2011complementation}, \cite{plebanek_rondos_sobota_products}, \cite{Plebanek_Alarcon_CSP}, and references therein. For classical results in the area, we refer the reader to papers \cite{benyamini1978extension}, \cite{Lindenstrauus_ell_infty}, \cite{peck1976lattice}, \cite{Pelczynski_projections}, \cite{Pelczynski_book_1968}, or \cite{RosenthalC(K)}. 
Using the tools developed in Sections \ref{sec:new_descriptions} and \ref{sec:ctcf}, we obtain in Section \ref{sec:apps_to_ck} two novel results on the complementability of (separable) spaces $\C(L)$ inside arbitrary spaces $\C(K)$. 
Our research might be also viewed as a natural continuation of our work \cite{Rondos-Sobota_copies_of_separable_C(L)} where, building on classical works of such authors as Gordon \cite{Gordon3}, Holszty\'{n}ski \cite{Holsztynski1966}, Pe\l czy\'nski and Semadeni \cite{Pelczynski_Semadeni_scattered}, or Rosenthal \cite{RosenthalC(K)}, we investigated when a separable space $\C(L)$ is embeddable (isometrically or isomorphically) into an arbitrary space $\C(K)$, in particular from the point of view of the topologies of $K$ and $L$. 
In the complemented case the situation has appeared to more complicated, and our results need to rely not only on the topologies of $K$ and $L$, but also on the properties of the spaces of Radon measures $\M(K)$ and $\M(L)$ (we recall that a purely topological characterization of the existence of a complemented copy of a separable space $\C(L)$ inside some space $\C(K)$ is unknown even for the simplest case of the space $\C(L)=\C([1, \omega])\simeq c_0$). 

Let us mention that a classical approach to locating a complemented copy of a given space $\C(L)$ inside another space $\C(K)$ is to find a continuous surjection of $K$ onto $L$ admitting a so-called \emph{regular averaging operator}, which automatically induces a positive norm-$1$ projection of $\C(K)$ onto $\C(L)$ (see \cite{Argyros-Arvanitakis_averaging_operators}, \cite{Pelczynski_book_1968}, \cite{RosenthalC(K)}). Our next main result shows that, in a certain sense, this approach is close to being equivalent to the existence of such a copy of $\C(L)$, with the only difference being that the continuous surjection in question need not be defined on the entire space $K$, but possibly only on some closed subset of $K$. 

\begin{theoreml}\label{thm:mainC}
Let $K$ be a compact space and $L$ be a compact metric space. Then, the following assertions are equivalent:
\begin{enumerate}[(1),itemsep=1mm]
    \item There exist a closed linear subspace $E$ of $\C(K)$, a surjective positive isometry $T\colon\C(L) \rightarrow E$, and a positive norm-$1$ projection $P\colon\C(K) \rightarrow E$.
    \item There exist a closed subset $F$ of $K$, a continuous surjection $\rho\colon F \rightarrow L$, and a continuous mapping $\phi\colon L \rightarrow \P(F)$ such that $\supp(\phi(x))\subseteq\rho^{-1}(x)$ for each $x\in L$.
\end{enumerate}
\end{theoreml}

For the proof of the extended version of Theorem \ref{thm:mainC}, see Section \ref{sec:holsztynski_separable}. Let us here note that Theorem \ref{thm:mainC} generalizes the following folklore fact which also concerns the complementability of separable spaces $\C(L)$ inside general spaces $\C(K)$, cf. Corollary \ref{cor:generalized_folklore_fact}.

\begin{thmx*}{Folklore Fact}\label{folklore_fact}\ \\
\indent (1) If a compact space $K$ contains a homeomorphic copy of a compact metric space $L$, then the space $\C(K)$ contains a complemented isometric copy of the space $\C(L)$.

(2) If $M$ is an uncountable compact metric space, then the space $\C(M)$ contains a complemented isometric copy of the space $\C(L)$ for any compact metric space $L$.
%
\end{thmx*}
\begin{proof}[Sketch of the proof]
For (1), note that by the Borsuk--Dugundji theorem (cf. \cite[Theorem 6.6]{Pelczynski_book_1968}) every compact metric space is \emph{Dugundji}, therefore, by \cite[Corollary 2.4]{Pelczynski_book_1968}, $\C(L)$ can be embedded isometrically as a complemented subspace of $\C(K)$.

To see (2), notice that by (1) the space $\C(2^\omega)$ isometrically embeds onto a complemented subspace of $\C(M)$. Then, since every compact metric space is \emph{Miljutin} (see \cite[Theorem 5.6]{Pelczynski_book_1968}) and by \cite[Corollary 2.3]{Pelczynski_book_1968}, $\C(L)$ embeds isometrically onto a $1$-complemented subspace of $\C(2^\omega)$ and hence onto a complemented subspace of $\C(M)$.
%
\end{proof}


One of the key ingredients of the proof of Theorem \ref{thm:mainC} is our last main result, Theorem \ref{thm:mainD}, which is a variant of the classical theorem of Holszty\'{n}ski (\cite{Holsztynski1966}) for the case of complemented subspaces, and which is of an independent interest. Recall that Holszty\'nski's result says that if $T$ is an isometry of a space $\C(L)$ (not necessarily separable) into another space $\C(K)$, then there are a closed subset $F$ of $K$, a continuous surjection $\rho\colon F\to L$, and a continuous mapping $\sigma\colon F\to\{-1,1\}$, such that $T(f)(y)=\sigma(y)f(\rho(y))$ for every $f\in\C(L)$ and $y\in F$. We provide the following stronger conclusion for the complemented case, which, among other things, also gives a simple formula for the norm-$1$ projection of $\C(K)$ onto $\C(L)$, see Section \ref{sec:holsztynski}. Note that the result provides a bridge connecting Holszty\'nski's theorem with above-mentioned Pe\l czy\'nski's theory of averaging operators, see \cite[Section 4]{Pelczynski_book_1968}.

\begin{theoreml}\label{thm:mainD}
Let $K$ and $L$ be compact spaces. Let $T\colon\C(L) \rightarrow\C(K)$ be an isometry onto a complemented subspace $E$ of $\C(K)$ and $P\colon\C(K) \rightarrow E$ be a norm-$1$ projection onto $E$. Then, there exist a closed subset $F$ of $K$, a continuous surjection $\rho\colon F \rightarrow L$, and continuous mappings $\sigma\colon F \rightarrow\{-1,1\}$ and $\phi\colon L \rightarrow S_{\M(F)}$ such that:
\begin{enumerate}[(i)]
    \item for each $x \in L$ it holds
    \[\supp(\phi(x))\subseteq\rho^{-1}(x),\]
    \item for each $f \in \C(L)$ and $y \in F$ it holds
    \[T(f)(y)=\sigma(y) f(\rho(y)),\]
    \item for each $g \in \C(K)$ and $y \in F$ it holds
    \[P(g)(y)=\sigma(y) \big\langle \phi(\rho(y)), g \big\rangle.\]
\end{enumerate}
\end{theoreml}

It is quite surprising that even though numerous variants and generalizations of the Holszty\'{n}ski theorem exist in the literature (see, e.g., \cite{Araujo-Font_isometries}, \cite{Botelho_Jamison_complex_Cambern-Holsztynski}, \cite{Chu-Michael_isometries}, \cite{Wolfr-Gebeily_isometries}, \cite{Galego-Villamizar_continuous_maps}, \cite{Galindo-Palacios_isometries}), up to our best knowledge none of them seems to deal with the quite natural assumption that the space $\C(L)$ is isometrically embedded onto a complemented subspace in the space $\C(K)$. This might possibly be caused by the fact that our proof requires, as a key ingredient, the usage of the \emph{maximum principle} for affine functions of the first Borel class, which has been discovered only recently, see Lemma \ref{lem:bidual of C(K) and Borel functions on K} and paper \cite{dosp}. The only result found by us, which is somewhat related to Theorem \ref{thm:mainD}, 
is \cite[Theorem 1]{FriedmanRusso1982} characterizing norm-$1$ projections on Banach spaces of the form $\C_0(X)$ using certain families of measures that resemble our mapping $\phi$. 

\medskip

To finish the introduction, let us briefly recap the plan of the paper. In the next preliminary section we recall some standard notation as well as introduce basic notions used throughout the paper. In Section \ref{sec:new_descriptions} we provide isometric descriptions of the spaces $\C([1,\omega^\alpha])$ and $\C(2^\omega)$ in terms of spaces of real-valued functions on trees. Section \ref{sec:ctcf} contains characterizations of the existence of complemented isometric or isomorphic copies of separable spaces $\C(L)$ in arbitrary Banach spaces $E$ in terms the presence of certain trees in $E \times E^*$. In Section \ref{sec:apps_to_ck} we generalize Holszty\'nski's classical result to the case of isometric embeddings onto complemented subspaces and, building on this result as well as on the results from previous sections, we give a characterization of when a separable space $C(L)$ is positively isometric to a complemented subspace of $\C(K)$ via a positive norm-$1$ projection.

\section{Preliminaries\label{sec:prelim}}

In this section we provide basic notation and terminology used throughout the paper.

\subsection{Real numbers}

For a sequence $(a_k)_{k=0}^\infty$ of real numbers we use the convention that $\sum_{k=1}^0a_k=0$. Further, as usual, for a real number $a\in\er$ we set $a^+=\max\{a,0\}$ and $a^-=\max\{-a,0\}$. We will need the following simple lemma.

\begin{lemma}\label{lem:sup_max_0}
For any sequence $(a_n)_{n=0}^\infty$ of real numbers we have
\[\sup_{n\ge0}\bigg(\sum_{k=0}^n a_k\bigg)^+=\left(a_0+\sup_{n\ge1}\bigg(\sum_{k=1}^n a_k\bigg)^+\right)^+.\]
\end{lemma}
\begin{proof}
    Let $(a_n)_{n=0}^\infty$ be a sequence of real numbers. We have
    \[\sup_{n\ge0}\bigg(\sum_{k=0}^n a_k\bigg)^+=\sup_{n\ge0}\max\bigg\{\sum_{k=0}^n a_k,\ 0\bigg\}=\sup_{n\ge0}\max\bigg\{a_0+\sum_{k=1}^n a_k,\ 0\bigg\}=\]
    \[=\max\bigg\{a_0+\sup_{n\ge0}\sum_{k=1}^n a_k,\ 0\bigg\}=\max\left\{a_0+\max\bigg\{0,\ \sup_{n\ge1}\sum_{k=1}^n a_k\bigg\},\ 0\right\}=\]
    \[=\max\left\{a_0+\sup_{n\ge1}\max\bigg\{0,\ \sum_{k=1}^n a_k\bigg\},\ 0\right\}=\max\left\{a_0+\sup_{n\ge1}\bigg(\sum_{k=1}^n a_k\bigg)^+,\ 0\right\}=\]
    \[=\left(a_0+\sup_{n\ge1}\bigg(\sum_{k=1}^n a_k\bigg)^+\right)^+.\]
\end{proof}

\subsection{Sets and ordinal numbers}

The cardinality of a set $X$ is denoted by $\abs{X}$. By $\en$ we denote the set of all \emph{positive} natural numbers, i.e. $\en=\{1,2,3,\ldots\}$. By $\omega$ we denote the first infinite ordinal number; as usual, we identify $\omega$ with the set of \emph{all} natural numbers, i.e. $\omega=\{0,1,2,\ldots\}=\{0\}\cup\en$. By $\omega_1$ we denote the first uncountable ordinal number. 

For a set $X$ by $\id_X\colon X\to X$ we denote the identity mapping on $X$, and for its subset $A$ by $\chi_A\colon X\to\{0,1\}$ we denote the characteristic function of $A$.

\subsection{Trees on $\omega$} 

We will frequently exploit trees on countable sets, for which we will use mostly standard notation and terminology, following e.g. Kechris \cite{kechris}.

Let $A$ be a set. By $A^{<\omega}$ we denote the set of all \emph{finite} sequences of elements of $A$, i.e., $A^{<\omega}=\bigcup_{n\in\omega}A^n$.  The only element of $A^0$, i.e., the \emph{empty} sequence, is denoted by $\emptyseq$. For $A=\omega$ we also write $\el=\omega^{<\omega}$. 

Let $s,t\in A^{<\omega}$ and $u\in A^\omega$. The symbol $l(s)$ denotes the \emph{length} of $s$, i.e. the unique number $n\in\omega$ such that $s\in A^n$; note that $s=\big\langle s(0),s(1),\ldots,s(l(s)-1)\big\rangle$. We write $s\preceq t$ if $s(k)=t(k)$ for any $k=0,\ldots,l(s)-1$, and $s\prec t$ if $s\preceq t$ and $l(s)<l(t)$. We also write $s \perp t$ if $s$ and $t$ are \emph{incomparable}, i.e.,  if none of the relations $s \preceq t$ and $t \preceq s$ holds. Similarly, $s\prec u$ if $s(k)=u(k)$ for any $k=0,\ldots,l(s)-1$. We set $s\concat t=\big\langle s(0),\ldots,s(l(s)-1),t(0),\ldots,t(l(t)-1)\big\rangle$. 
For $a\in A$, we simply write $s\concat a$ for $s\concat\langle a\rangle$ and $a\concat s$ for $\langle a\rangle\concat s$. If $l(s)>0$, then we set $sh(s)=\big\langle s(1),\ldots,s(l(s)-1)\big\rangle$ (\emph{shift} of $s$), i.e. $sh(s)$ removes the first element $s(0)$ of $s$, and $pred(s)=\big\langle s(0),\ldots,s(l(s)-2)\big\rangle$ (\emph{predecessor} of $s$), i.e. $pred(s)$ removes the last element $s(l(s)-1)$ of $s$; in particular, for any $a\in A$ we have $sh(\langle a\rangle)=\emptyseq=pred(\langle a\rangle)$.

\medskip

Let $\Lambda$ be a subset of $\el$. We say that $\Lambda$ is \emph{hereditary} if, for any $s,t\in\el$, whenever $t\in\Lambda$ and $s\preceq t$, then $s\in\Lambda$. An element $s\in\Lambda$ is called a \emph{leaf} of (hereditary) $\Lambda$ if there is no $t\in\Lambda$ such that $s\prec t$. The symbol $NL(\Lambda)$ stands for the set of all elements of $\Lambda$ which are \emph{not} leaves; note that $NL(\el)=\el$.

Let $\Omega\subseteq\Lambda$. For every $s\in\Omega$ we denote
\[\Omega_s=\{t\in\Omega\colon\ s\preceq t\}\]
and
\[\Omega^s=\{t\in\omega^{<\omega}\colon\ s\concat t\in\Omega\}.\]
Note that $\Omega^{\langle n\rangle}=sh[\Omega_{\langle n\rangle}]$ for every $n\in\omega$. 
Also, the \emph{length} of $\Omega$ is defined as
\[l(\Omega)=\sup \{l(s)\colon\ s \in \Omega\}.\]

Assume that $\Lambda$ is hereditary and $\emptyseq\in\Lambda$. If $\Lambda$ does not have \emph{infinite branches}, i.e. there is no sequence $(s_n)_{n\in\omega}$ in $\Lambda$ such that $s_n\prec s_{n+1}$ for every $n\in\omega$, then for an element $s\in\Lambda$ its \emph{rank} $r(s)$ is defined recursively as follows: 
\[
r(s)=\begin{cases}
    0,&\text{ if }s\text{ is a leaf in }\Lambda,\\
    r(s)=\sup_{s \prec t} (r(t)+1),&\text{ otherwise}.
\end{cases}
\]

\begin{definition}\label{def:tree}
Let $\Lambda\subseteq\el$. Let $\alpha$ be a countable ordinal. We say that $\Lambda$ is a \emph{tree of rank} $\alpha+1$ if:
\begin{itemize}
    \item $\Lambda$ is hereditary and $\emptyseq\in\Lambda$,
    \item $\Lambda$ does not have infinite branches,
    \item $r(\emptyseq)=\alpha$,
    \item for each $s \in \Lambda$, either $s$ is a leaf, or $s\concat n \in \Lambda$ for all $n \in \omega$,
    \item for each $s \in NL(\Lambda)$, the sequence of ordinals $(r(s\concat n))_{n \in \omega}$ is either constant or strictly increasing.
\end{itemize}

If $\Lambda$ is \emph{full}, i.e. $\Lambda=\omega^{<\omega}=\el$, then we say that $\Lambda$ is a \emph{tree of rank} $\infty$; we use the convention that $\alpha<\infty$ for each ordinal $\alpha$.

Finally, in general, by a \emph{tree} we mean a tree of rank $\beta$, where $\beta$ is either a countable successor ordinal or $\infty$. The set of all trees of rank $\beta$ will be denoted by $\TT_\beta$ and the set of all trees by $\TT$.\eodef
\end{definition}

Note that there is only one tree $\Upsilon$ of rank $1$, namely $\Upsilon=\{\emptyseq\}$, and so $\TT_1=\{\Upsilon\}$.

\begin{definition}
    Let $\Lambda\in\TT$. A finite subset $F\subseteq\Lambda$ is called a \emph{trunk} of $\Lambda$ if $F$ is hereditary and $\emptyseq\in F$. The set of all trunks of $\Lambda$ is denoted by $\mathfrak{Tr}(\Lambda)$.\eodef
\end{definition}

For every $\Lambda\in\TT$, the pair $(\mathfrak{Tr}(\Lambda),\subseteq)$ is a directed set and it holds $\Lambda=\bigcup \mathfrak{Tr}(\Lambda)$. We also have $\mathfrak{Tr}(\Upsilon)=\{\Upsilon\}$.

\subsection{Topological spaces}

All spaces considered in this paper are assumed to be \emph{Hausdorff}.

For ordinal numbers $\alpha\le\beta$ by $[\alpha,\beta]$ we denote the space of all ordinal numbers $\gamma$ such that $\alpha\le\gamma\le\beta$, endowed with the order topology.

We will use the notion of the Cantor--Bendixson derivatives of topological spaces. Let $X$ be a topological space. Set $X^{(0)}=X$ and  let $X^{(1)}$ be the set of all accumulation points of $X$ or, equivalently, $X^{(1)}=X\setminus\{x\in X\colon x\text{ is isolated in }X\}$. Further, for an ordinal number $\alpha>1$, let $X^{(\alpha)}=(X^{(\beta)})^{(1)}$ if $\alpha=\beta+1$, and $X^{(\alpha)}=\bigcap_{\beta<\alpha} X^{(\beta)}$ if $\alpha$ is a limit ordinal. For each ordinal $\alpha$ the set $X^{(\alpha)}$ is called the \emph{Cantor--Bendixson derivative of $X$ of order $\alpha$}. 

For $X$ there exists an ordinal $\alpha$ such that for each $\beta>\alpha$ it holds $X^{(\beta)}=X^{(\alpha)}$. For such an ordinal $\alpha$, we put $X^{(\infty)}=X^{(\alpha)}$; the set $X^{(\infty)}$ is called the \emph{perfect kernel} of $X$. $X$ is said to be \emph{scattered} if $X^{(\infty)}=\emptyset$---in this case the minimal ordinal $\alpha$ such that $X^{(\alpha)}=\emptyset$ is called the \emph{height} of $X$ and is denoted by $ht(X)$. If $X$ is not scattered, then let $ht(X)=\infty$; we again use the convention that $\alpha<\infty$ for each ordinal $\alpha$. It is easy to see that if $X$ is a scattered compact space, then $ht(X)$ is a successor ordinal and $X^{(ht(X)-1)}$ is a finite set. Note  that a compact metric space is scattered if and only if it is countable. We also notice that, for each ordinal number $\alpha$ and the ordinal interval $[1,\omega^{\alpha}]$, we have $([1, \omega^{\alpha}])^{(\alpha)}=\{\omega^{\alpha}\}$, which can be easily proved by transfinite induction, and thus $ht([1, \omega^{\alpha}])=\alpha+1$. Recall also the classical Mazurkiewicz--Sierpi\'{n}ski classification of countable compact spaces: each countable compact space $K$ is homeomorphic to the ordinal interval $[1, \omega^{\alpha}m]$, where $\alpha=ht(K)-1$ and $m=\abs{K^{(\alpha)}}$ (see \cite[Theorem 8.6.10]{semadeni}).

\medskip

By $2^\omega$ we denote the Cantor space (endowed with the product topology). For a sequence $s \in 2^{<\omega}$ let $[s]$ stand for the clopen subset of $2^{\omega}$ defined as
\[[s]=\{u \in 2^{\omega}\colon s \prec u\}.\]
Note that the family $\{[s]\colon s\in 2^{<\omega}\}$ is a base of the topology of $2^\omega$. Also, we endow $2^\omega$ with the (canonical) metric $d_{2^\omega}$, defined for every $x,y\in 2^\omega$ as follows:
\[d_{2^\omega}(x,y)=\begin{cases}2^{-k},&\text{ if }x\neq y\text{ and }k=\min\big\{n\in\omega\colon\ x(n)\neq y(n)\big\},\\0,&\text{ otherwise}.\end{cases}.\]

Below we define an auxiliary embedding $R$ of the full tree $\el=\omega^{<\omega}$ into the Cantor space $2^\omega$, which will enable us to define topologies on trees.

\begin{definition}    
The function $Q \colon \omega^{<\omega} \rightarrow 2^{<{\omega}}$ is defined as follows: let $Q(\emptyseq)=\emptyseq$, and for a nonempty sequence $s=\langle s_1, \ldots, s_n \rangle \in \omega^{<\omega}$ let
\smallskip
\[Q(s)=\langle\ \overbrace{1, \ldots, 1}^{s_1\text{ times}},\ 0,\ \overbrace{1, \ldots, 1}^{s_2\text{ times}},\ 0,\ \ldots,\ 0,\ \overbrace{1, \ldots, 1}^{s_n\text{ times}},\ 0\ \rangle.\]

The function $R \colon \omega^{<\omega} \rightarrow 2^{\omega}$ is defined for every $s\in\omega^{<\omega}$ and $n\in\omega$ as follows: 
\[R(s)(n)=\begin{cases}
    Q(s)(n)&,\text{ if }n<l(Q(s)),\\
    1,&,\text{ if }n\ge l(Q(s)).
\end{cases}\]\eodef
\end{definition}

It is immediate that for every incomparable $s,t\in\el\setminus\{\emptyseq\}$ we have $[Q(s)]\cap[Q(t)]=\emptyset$. Note also that the range of $R$ is a dense subset of $2^{\omega}$. Moreover, as $R$ is injective, we can use it to transport the topology of the image $R[\el]$, inherited from $2^\omega$, back onto $\el$, making $R$ a homeomorphism.

The topology on $\el$ has the following natural property:

\begin{lemma}\label{lemma:tree_topology}
   Let $(t_k)_{k\in\omega}$ be a sequence in $\el$ and let $s\in\el$. We have $\lim_{k\to\infty}t_k=s$ if and only if there are $k_0\in\omega$ and a sequence $(n_k)_{k\ge k_0}$ in $\omega$ such that $\lim_{k\to\infty}n_k=+\infty$ and for every $k\ge k_0$ we have $s\concat n_k\preceq t_k$. In particular, if $s\concat k\preceq t_k$ for every $k\in\omega$, then $\lim_{k\to\infty}t_k=s$.
\end{lemma}

For a countable successor ordinal $\alpha$, we endow each tree $\Lambda\in\TT_\alpha$ with the topology inherited from $\el$. An easy inductive argument, exploiting the Mazurkiewicz--Sierpi\'{n}ski theorem, shows that $\Lambda$ and $R[\Lambda]$ are actually homeomorphic to $[1,\omega^{\alpha-1}]$. 
A canonical homeomorphism of these spaces which posseses some additional properties will be constructed in Theorem \ref{A_lambda is C(omega^alpha)} below. 

\subsection{Banach spaces and Banach lattices}

All mappings (\emph{operators}) between vector spaces considered in this paper are assumed to be \emph{linear}. Let $(E,\norm{\cdot}_E)$ and $(F,\norm{\cdot}_F)$ be normed spaces. We write $E\simeq F$ (resp. $E\cong F$) if $E$ and $F$ are isomorphic (resp. isometrically isomorphic). 
Note that if we say that an operator $T\colon E\to F$ is an isometry, then we do not automatically assume that $T$ is surjective. If $E,F$ are Banach spaces and $T\colon E\to F$ is an isomorphism (resp. isometry) onto the (closed) subspace $T[E]$ of $F$, then we call $T$ an \emph{isomorphic embedding} (resp. \emph{isometric embedding}) of $E$ into $F$.

$S_E$ and $B_E$ denote the unit sphere and the unit ball of $E$, respectively. By $E^*$ we denote the dual space of $E$. The weak* topology of $E^*$ is denoted by $w^*$. A collection $(e_i,e_i^*)_{i\in I}$ of elements of the product $E\times E^*$ is called \emph{biorthogonal} if $e_i^*(e_i)=1$ and $e_i^*(e_j)=0$ for $i\neq j\in I$.

If $E$ is a normed lattice, its cone of \emph{positive} elements is denoted by $E_+$. The dual space $E^*$ is of course a Banach lattice, with the positive cone defined as follows: for every $\e\in E^*$, $\e \in E^*_+$ if and only if $\e(e) \geq 0$ for each $e \in E_+$. For a closed linear subspace $G$ of $E$ we set $G_{+}=G \cap E_+$. 
If $F$ is also a normed lattice, then an operator $T\colon G\to F$ is \emph{positive} if $T[G_+]\subseteq F_+$. For an element $e\in E$, we denote by $e^+$ and $e^-$ its \emph{positive} and \emph{negative} part, respectively, i.e. $e^+=e\vee 0$ and $e^-=(-e)\vee0$, and set $|e|=e^+\vee e^-$. 

The sum space $E\oplus F$ is by default endowed with the supremum norm, i.e. $\norm{e+f}_{E\oplus F}=\max\{\norm{e}_E,\norm{f}_F\}$ for any $e\in E, f\in F$. If $E,F$ are normed lattices, then $E\oplus F$ is also a normed lattice when endowed with the coordinate-wise order. 

If $E$ is a Banach space, then a closed linear subspace $G$ of $E$ is \emph{complemented} in $E$ if there is a closed linear subspace $H$ of $E$ such that $E=G+H$ and $G\cap H=\{0\}$, equivalently, if there exists a \emph{projection} from $E$ onto $G$, i.e. a bounded operator $P \colon E \rightarrow G$ such that $P\rstr{G}$ is the identity $\id_G$ on $G$; if $\norm{P}=1$, we say that $G$ is $1$-\emph{complemented} in $E$. If $E$ is a Banach lattice and $P$ is a positive operator, then we say that $G$ is \emph{positively complemented}. 

Let $\Lambda$ be a tree, i.e. $\Lambda\in\TT$. If $(e_s)_{s \in \Lambda}$ is a system of elements of $E$, then the sum $\sum_{s \in \Lambda} e_s$ stands for the unique element $e \in E$ which satisfies 
\[\inf_{F \in \mathfrak{Tr}(\Lambda)}\ \sup_{\substack{F' \in \mathfrak{Tr}(\Lambda)\\F\subseteq F'}}\ \norm{e-\sum_{s \in F'} e_s}_E=0,\]
if it exists.

We will usually omit the index in the norm and simply write $\|\cdot\|$ instead of $\|\cdot\|_E$.

\subsection{Spaces $\C(K)$, $\M(K)$, and $\P(K)$}

Let $K$ be a compact space. By $\C(K)$ we denote the Banach space of all real-valued functions on $K$ endowed with the supremum norm, and by $\C(K,[0,1])$ the subset of $\C(K)$ consisting of all functions $f$ such that $0\le f(x)\le 1$ for every $x\in K$. A function $f\in\C(K)$ is \emph{positive} if $f(x)\ge0$ for any $x\in K$; this cone of positive elements induces a natural Banach-lattice structure on $\C(K)$ such that for each $f,g\in\C(K)$ we have $f\le g$ if and only if $f(x)\le g(x)$ for all $x\in K$, as well as for each $f\in\C(K)$ we have $|f|(x)=|f(x)|$ for all $x\in K$.

If $U$ is a clopen subset of $K$, then we will often tacitly identify the space $\C(U)$ with the (closed linear) subspace $\{g\in\C(K)\colon\ g\rstr(K\setminus U)\equiv 0\}$ of $\C(K)$; note that $\C(U)$ is positively $1$-complemented in $\C(K)$.


\medskip

By $\M(K)$ we denote the Banach space of all regular signed Borel measures on $K$ endowed with the variation norm. Recall that by the Riesz--Markov--Kakutani representation theorem $\M(K)$ is isometrically isomorphic to the dual space $\C(K)^*$, with the action of $\M(K)$ on $\C(K)$ given for $\mu\in \M(K)$ and $f\in\C(K)$ by the formula $\mu(f)=\langle\mu,f\rangle=\int_K fd\mu$. A measure $\mu\in \M(K)$ is \emph{positive} if $\mu(B)\ge0$ for any Borel $B\subseteq K$. $\P(K)$ denotes the subset of $\M(K)$ consisting of all \emph{probability} measures, i.e. positive measures $\mu\in \M(K)$ such that $\mu(K)=1$. By default, the set $\P(K)$ is endowed with the weak$^*$ topology $w^*$ inherited from $\M(K)$. For a measure $\mu\in \M(K)$ by $|\mu|(\cdot)$ we denote its \emph{variation}. 

As usual, for a point $x\in K$, by $\delta_x$ we denote the \emph{one-point} measure (the \emph{Dirac delta}) concentrated at $x$. We further recall the standard fact that the set $\ext(B_{\M(K)})$ of all extreme points of the ball $B_{\M(K)}$ equals to $\{\pm \delta_x\colon x \in K \}$ (see e.g. \cite[Section 4.3]{lmns}), and thus, by the Krein--Milman theorem, we have $B_{\M(K)}=\overline{\mbox{co}}^{w^*}\{\pm \delta_x\colon x \in K \}$.

\medskip

For a function $f\in\C(K)$ we define its \emph{support} as $\supp(f)=\overline{\{x\in K\colon f(x)\neq 0\}}$. Similarly, the \emph{support} $\supp(\mu)$ of a measure $\mu\in \M(K)$ is the minimal closed subset $F\subseteq K$ such that $|\mu|(K\setminus F)=0$. If $F$ is a closed subset of $K$, then we naturally identify the space $\M(F)$ with the complemented subspace of $\M(K)$ consisting of all those measures $\mu$ for which $\supp(\mu)\subseteq F$; in a similar vein the set $\P(F)$ is identified with the closed subset of $\P(K)$. 
A function $f\in\C(K)$ (resp. a measure $\mu\in \M(K)$) is \emph{supported} by an open subset $U$ of $K$ if $\supp(f)\subseteq U$ (resp. $\supp(\mu)\subseteq U$). 

\medskip

As usual, $c_0$ denotes the Banach space $\{x\in\er^\omega\colon \lim_{n\to\infty}x(n)=0\}$ endowed with the supremum norm.

\section{Isometric descriptions of the spaces $\C([1, \omega^{\alpha}])$ and $\C(2^{\omega})$\label{sec:new_descriptions}}

In this section we provide isometric descriptions of the spaces $\C([1, \omega^{\alpha}])$, for countable $\alpha$, and $\C(2^{\omega})$ in terms of spaces of real-valued functions on trees $\Lambda\in\TT$. Although those descriptions may seem in the beginning a bit technical and intricate, they will be extremely useful in the sequel. Let us note that similar results, though definitely less concrete, have been obtained e.g. by Bourgain \cite[Section 2]{Bourgain1979} (see also \cite[page 35]{RosenthalC(K)}) and Dilworth \textit{et al.} \cite[Section 3.4]{Dilworth-et-al}.

We start with the following definition of a norm on the spaces $\er^\Lambda$, where $\Lambda\in\TT$.

\begin{definition}
Fix a tree $\Lambda\in\TT$ and let $\olb{a}=(a_s)_{s\in\Lambda}\in\er^\Lambda$. We define the \emph{$\Lambda$-norm} $\norm{\olb{a}}^\Lambda$ of $\olb{a}$ as
\[ \norm{\olb{a}}^\Lambda=\sup_{t\in\Lambda}\bigg|\sum_{s \preceq t} a_s\bigg|.\]\eodef
\end{definition}

\begin{definition}
Fix a tree $\Lambda\in\TT$ and let $\Omega\subseteq\Lambda$ be an arbitrary subset. For an element $\olb{a}=(a_s)_{s\in\Lambda}\in\er^\Lambda$, we define its \emph{restriction} $\olb{a}\rstr\Omega$ to $\Omega$ as the element of $\er^\Lambda$ given for every $s\in\Lambda$ by the formula:
\[
\big(\olb{a}\rstr\Omega\big)(s)=\begin{cases}
    a_s,&\text{ if }s\in\Omega,\\
    0,&\text{ is }s\in\Lambda\setminus\Omega.
\end{cases}
\]
We say that an element $\olb{a}\in\er^\Lambda$ is \emph{supported} by $\Omega$ if $\olb{a}\rstr\Omega=\olb{a}$, and that $\olb{a}$ is \emph{finitely supported} if there exists a finite set $F\subseteq\Lambda$ such that $\olb{a}$ is supported by $F$.\eodef
\end{definition}

\begin{definition}\label{def:tree_space}
Fix a tree $\Lambda\in\TT$. We define the set $\A_{\Lambda}\subseteq\er^{\Lambda}$ as the family of all systems $\olb{a}=(a_s)_{s \in \Lambda}\in\er^{\Lambda}$ such that:
\begin{enumerate}[({A}.1)]
    \item $\norm{\olb{a}}^\Lambda<\infty$,
    \item  for each $s \in NL(\Lambda)$ we have
    \[\lim_{n \rightarrow \infty} \norm{\olb{a}\rstr\Lambda_{s\concat n}}^{\Lambda}=0,\]
    \item we have 
    \[\lim_{n\to\infty}\norm{\olb{a}\rstr\{t\in\Lambda\colon l(t) \geq n\}}^{\Lambda}=0.\]
\end{enumerate}\eodef
\end{definition}

\begin{definition}
    For a tree $\Lambda\in\TT$, by $\mathcal{FS}(\A_\Lambda)$ we denote the subset of $\A_\Lambda$ consisting of all finitely supported elements in $\A_\Lambda$.\eodef
\end{definition}

Note that for the only tree $\Upsilon$ of rank $1$ we trivially have $\A_{\Upsilon}=\er^{\Upsilon}\cong\er$.

\begin{remark}\label{remark:norms_lim}
Fix $\Lambda\in\TT$. Let $\olb{a}=(a_s)_{s\in\Lambda}\in\er^\Lambda$. Then, for every $s\in\Lambda$ we have
\[\norm{\olb{a}\rstr\Lambda_s}^{\Lambda}=\norm{(a_{s\concat t})_{t\in\Lambda^s}}^{\Lambda^s}.\]
Consequently, if $\olb{a}$ satisfies condition (A.2), then for each $s \in NL(\Lambda)$ we have
\[\lim_{n\to\infty}\norm{(a_{(s\concat n)\concat t})_{t\in\Lambda^{s\concat n}}}^{\Lambda^{s\concat n}}=\lim_{n \rightarrow \infty} \norm{\olb{a}\rstr\Lambda_{s\concat n}}^{\Lambda}=0\]
as well as that for any $s\in\Lambda$ the system $(a_{s\concat t})_{t\in\Lambda^s}\in\er^{\Lambda^s}$ satisfies condition (A.2), too.
\end{remark}

The proof of the next lemma is routine, therefore we omit it.

\begin{lemma}
\label{A_lambda is a normed space}
Let $\Lambda\in\TT$. The space $\A_{\Lambda}$ with the coordinate-wise addition and multiplication by scalar is a vector space. The function $\norm{\cdot}^{\Lambda}$ is a norm on $\A_\Lambda$.
\end{lemma}

From now on, for each $\Lambda\in\TT$, we will treat $\norm{\cdot}^{\Lambda}$ as the default norm on $\A_\Lambda$. 

The next lemma implies that the space $\mathcal{FS}(\A_\Lambda)$ is norm-dense in $\A_\Lambda$.

\begin{lemma}
\label{lem:fin-approx}
Fix a tree $\Lambda\in\TT$ and let $\olb{a}=(a_s)_{s \in \Lambda}\in \A_{\Lambda}$. Then, 
\[\olb{a}=\lim_{F \in\mathfrak{Tr}(\Lambda)}\olb{a}\rstr F,\]
where the limit is computed with respect to the $\Lambda$-norm of $\A_\Lambda$, that is, for every $\eps>0$ there is $F\in\mathfrak{Tr}(\Lambda)$ such that for every $F'\in\mathfrak{Tr}(\Lambda)$ containing $F$ we have $\norm{\olb{a}-\olb{a}\rstr F'}^\Lambda<\eps$.
\end{lemma}

\begin{proof}
%
Fix $\eps>0$. 
By condition (A.3), we find $n \in \en$ such that
\[\norm{\olb{a}\rstr\{t \in \Lambda\colon l(t) > n\}}^\Lambda<\eps/4.\]
Set
\[\tilde{\Lambda}=\{s \in \Lambda\colon 0 \leq l(s) \leq n\},\]
so we have
\[\norm{\olb{a}\rstr(\Lambda\setminus\tilde{\Lambda})}^\Lambda<\eps/4.\]
We  proceed by induction on $k=1,\ldots, n$ to construct auxiliary subsets $F_k$ and $\Omega_k$ of $\Lambda$. 
Let 
\[F_1=\Big\{\langle m\rangle\in\tilde{\Lambda}\colon\  m\in \omega\text{ is such that } \norm{\olb{a}\rstr\Lambda_{\langle m\rangle}}^\Lambda \geq {\ep}/{(4n)}\Big\}\]
and 
\[\Omega_1=\bigcup\Big\{\tilde{\Lambda}_{\langle m\rangle}\colon\ m\in\omega,{\langle m\rangle}\not\in F_1\Big\}.\]
Assuming that $1 \leq k <n$ and the sets $F_k$ and $\Omega_k$ have been defined, let
\[F_{k+1}=\Big\{t\concat m\in\tilde{\Lambda}\colon t \in F_k\text{ and }m \in \omega\text{ are such that } \norm{\olb{a}\rstr\Lambda_{t\concat m}}^\Lambda \geq {\ep}/{(4n)}\Big\}\]
and 
\[\Omega_{k+1}=\bigcup\Big\{\tilde{\Lambda}_{t\concat m}\colon\ t\in F_{k}, m\in\omega, t\concat m\not\in F_{k+1}\Big\}.\]

After the above inductive process is finished, put
\[F=\{\emptyseq\} \cup \bigcup_{k=1}^n F_k\quad\text{and}\quad\Omega=\bigcup_{k=1}^n \Omega_k.\]
Then, by condition (A.2), $F$ is a trunk of $\Lambda$. Also, it follows from the construction that $\tilde{\Lambda}=F \cup \Omega$ and $F\cap\Omega=\emptyset$ as well as that for each $1\leq k\neq l\leq n$ we have $F_k\cap F_l=\emptyset$ and $\Omega_k\cap\Omega_l=\emptyset$. Finally, by the hereditariness of $\Lambda$ and $\tilde{\Lambda}$, for each $1 \leq k \leq n$ we have $\norm{\olb{a}\rstr\Omega_k} < {\ep}/{(4n)}$. Thus, by the triangle inequality for the $\Lambda$-norm, 
we get
\[\norm{\olb{a}-\olb{a}\rstr F}^{\Lambda}=\norm{\olb{a}\rstr(\Lambda\setminus F)}^{\Lambda}\leq\norm{\olb{a}\rstr(\Lambda \setminus \tilde{\Lambda})}^{\Lambda}+\norm{\olb{a}\rstr\Omega}^{\Lambda}\le\] 
\[\leq\norm{\olb{a}\rstr(\Lambda \setminus \tilde{\Lambda})}^{\Lambda}+\sum_{k=1}^n \norm{\olb{a}\rstr\Omega_k}^{\Lambda}<\ep/4+n\cdot({\ep}/{4n})=\ep/2.\]

To finish the proof, take any $F'\in\mathfrak{Tr}(\Lambda)$ such that $F\subseteq F'$ and note that, by the very fact that $F'$ is a trunk and a similar computation as above, we have
\[\norm{\olb{a}\rstr(F'\setminus F)}^\Lambda\le\norm{\olb{a}\rstr(F'\setminus\tilde{\Lambda})}^\Lambda+\norm{\olb{a}\rstr(F'\cap\Omega)}^\Lambda<\eps/4+\eps/4=\eps/2,\]
hence
\[\norm{\olb{a}-\olb{a}\rstr F'}^{\Lambda}=\norm{\olb{a}-\olb{a}\rstr F-\olb{a}\rstr(F'\setminus F)}^{\Lambda}\le\]
\[\le\norm{\olb{a}-\olb{a}\rstr F}^\Lambda+\norm{\olb{a}\rstr(F'\setminus F)}^{\Lambda}<\eps/2+\eps/2=\eps.\]
\end{proof}

We also define a vector-lattice structure on spaces $\A_\Lambda$.

\begin{definition}
Let $\Lambda\in\TT$. For $\olb{a}=(a_s)_{s \in \Lambda},\olb{b}=(b_s)_{s \in \Lambda}\in\A_\Lambda$, we write $\olb{a}\le\olb{b}$ if for each $t\in\Lambda$ we have
\[\sum_{s\preceq t}a_s\le\sum_{s\preceq t}b_s.\]
\eodef
\end{definition}


In addition to Lemma \ref{A_lambda is a normed space}, we have the following result.

\begin{lemma}
\label{A_lambda is a normed lattice}
Let $\Lambda\in\TT$. The space $\A_{\Lambda}$ with the above-defined partial order $\le$ is a normed lattice.
\end{lemma}

\begin{proof}
We need to show that, given $\olb{a}=(a_s)_{s \in \Lambda}, \olb{b}=(b_s)_{s \in \Lambda} \in \A_{\Lambda}$, their supremum $\olb{a} \vee \olb{b}$ exists in $\A_{\Lambda}$, and the implication $\abs{\olb{a}} \leq \abs{\olb{b}} \Rightarrow \norm{\olb{a}}^\Lambda \leq \norm{\olb{b}}^\Lambda$ holds true, where, recall, we have 
\[\abs{\olb{a}}=\olb{a}^+ \vee \olb{a}^{-}=(\olb{a} \vee 0) \vee (-\olb{a} \vee 0).\] 

We first claim that the element $\olb{c}=(c_s)_{s \in \Lambda}$ defined by setting $c_{\emptyseq}= \max\{a_{\emptyseq},b_{\emptyseq}\}$ and, for $s \neq \emptyseq$, 
\[c_s=\max\bigg\{\sum_{t \preceq s} a_t\ ,\ \sum_{t \preceq s} b_t\bigg\}-\max\bigg\{\sum_{t \prec s} a_t\ ,\ \sum_{t \prec s} b_t \bigg\},\]
is the supremum of the pair $\{\olb{a}, \olb{b}\}$ in $\A_{\Lambda}$. First, we need to show that $\olb{c}$ belongs to the space $\A_{\Lambda}$. Note that for every $t\in\Lambda$ and $u\preceq t$ we have
\[\tag{$*$}\sum_{u\preceq s\preceq t}c_s=\max\bigg\{\sum_{s \preceq t} a_s\ ,\ \sum_{s \preceq t} b_s\bigg\}-\max\bigg\{\sum_{s \prec u} a_s\ ,\ \sum_{s \prec u} b_s\bigg\}.\]
Taking $u=\emptyseq$, it follows that 
\[\norm{\olb{c}}^\Lambda=\sup_{t \in \Lambda} \abs{\sum_{s \preceq t} c_s}=\sup_{t \in \Lambda} \abs{ \max\bigg\{\sum_{s \preceq t} a_s\ ,\ \sum_{s \preceq t} b_s\bigg\}} \leq\]
\[\leq \max\Bigg\{\sup_{t \in \Lambda} \bigg|\sum_{s \preceq t} a_s\bigg|\ ,\ \sup_{t \in \Lambda} \bigg|\sum_{s \preceq t} b_s\bigg|\Bigg\}=\max \Big\{\norm{\olb{a}}^\Lambda\ ,\ \norm{\olb{b}}^\Lambda\Big\} <\infty,\]
so condition (A.1) holds for $\olb{c}$. 

We now prove that $\olb{c}$ satisfies condition (A.2). Assume that it does not---by ($*$) it follows that there are $s\in NL(\Lambda)$, $\eps>0$, a strictly increasing sequence $(n_k)_{k\in\omega}$ in $\omega$, and a sequence $(t_k)_{k\in\omega}$ in $\Lambda$ such that for every $k\in\omega$ we have $s\concat n_k\preceq t_k$ and 
\[\abs{\max\bigg\{\sum_{u \preceq t_k} a_u\ ,\ \sum_{u \preceq t_k} b_u\bigg\}-\max\bigg\{\sum_{u \prec s\concat n_k} a_u\ ,\ \sum_{u \prec s\concat n_k} b_u\bigg\}}\ge\eps.\]
By the inequality 
\[\abs{\max\{a, b\}-\max\{c, d\}} \leq \max\{\abs{a-c}, \abs{b-d}\},\]
valid for all real numbers $a, b, c, d$, for each $k\in\omega$ we have
\[\eps \leq \max \Bigg\{\abs{\sum_{s\concat n_k \preceq u \preceq t_k} a_u}\ ,\ \abs{\sum_{s\concat n_k \preceq u \preceq t_k} b_u} \Bigg\},\]
which contradicts the fact that both $\olb{a}$ and $\olb{b}$ satisfy condition (A.2).

Condition (A.3) can be checked in a similar fashion as condition (A.2). This implies that $\olb{c}\in \A_\Lambda$.

\medskip

Now, we check that $\olb{c}$ is an upper bound for $\olb{a}$ and $\olb{b}$. By the definition of the order, for each $t \in \Lambda$, we have 
\[\sum_{s\preceq t}c_s=\max\bigg\{\sum_{s \preceq t} a_s\ ,\ \sum_{s \preceq t} b_s\bigg\} \geq \sum_{s\preceq t}a_s,\]
i.e., $\olb{c} \geq \olb{a}$. We similarly show that $\olb{c} \geq \olb{b}$. 

Next, assume that an element $\olb{d} \in \A_{\Lambda}$ satisfies $\olb{d} \geq \olb{a}$ and $\olb{d} \geq \olb{b}$. As above, this is the same as saying that for each $t \in \Lambda$ we have
\[\sum_{s \preceq t} d_s \geq \sum_{s \preceq t} a_s\quad\text{and}\quad\sum_{s \preceq t} d_s \geq \sum_{s \preceq t} b_s.\]
Hence, for each $t\in\Lambda$ it holds
\[\sum_{s \preceq t} d_s \geq \max\bigg\{\sum_{s \preceq t} a_s\ ,\ \sum_{s \preceq t} b_s\bigg\}=\sum_{s \preceq t} c_s,\]
which proves that $\olb{c} \leq \olb{d}$, and so $\olb{c}$ is the supremum of the set $\{\olb{a}, \olb{b}\}$ in $\A_{\Lambda}$, as desired.

\medskip

Finally, we show that $\abs{\olb{a}} \leq \abs{\olb{b}}$ implies that $\norm{\olb{a}}^\Lambda \leq \norm{\olb{b}}^\Lambda$. First, notice that for the positive part $\olb{a}^+=\olb{a}\vee 0=(a_{s,+})_{s \in \Lambda}$ and the negative part $\olb{a}^-=(-\olb{a})\vee 0=(a_{s,-})_{s \in \Lambda}$ of $\olb{a}$ we have $a_{\emptyseq,+}=a_{\emptyseq}^+$, $a_{\emptyseq,-}=a_{\emptyseq}^-$, and, for $s \in \Lambda \setminus\{\emptyseq\}$, 
\[a_{s, +}=\bigg(\sum_{t \preceq s} a_t \bigg)^{+}-\bigg(\sum_{t \prec s} a_t \bigg)^{+}\quad\text{and}\quad a_{s, -}=\bigg(\sum_{t \preceq s} a_t \bigg)^{-}-\bigg(\sum_{t \prec s} a_t \bigg)^{-}.\]
It follows that for the element $\abs{\olb{a}}=(\widehat{a}_{s})_{s \in \Lambda}=\olb{a}^+ \vee \olb{a}^-$ it holds $\widehat{a}_{\emptyseq}=\max\{a_{\emptyseq}^+,a_{\emptyseq}^-\}=\abs{a_{\emptyseq}}$ and, for $s \neq \emptyseq$, 
\[
 \widehat{a}_{s}=\max\bigg\{\sum_{t \preceq s} a_{t, +}\ ,\ \sum_{t \preceq s} a_{t, -}\bigg\}-\max\bigg\{\sum_{t \prec s} a_{t, +}\ ,\ \sum_{t \prec s} a_{t, -}\bigg\}=\]
\[=\max\Bigg\{\bigg(\sum_{t \preceq s} a_t \bigg)^{+}\ ,\ \bigg(\sum_{t \preceq s} a_t \bigg)^{-}\Bigg\}-\max\Bigg\{\bigg(\sum_{t \prec s} a_t \bigg)^{+}\ ,\ \bigg(\sum_{t \prec s} a_t \bigg)^{-}\Bigg\}=\]  
\[=\abs{\sum_{t \preceq s} a_t}-\abs{\sum_{t \prec s} a_t}.\]
So, the assumption $\abs{\olb{a}} \leq \abs{\olb{b}}$ yields that for each $s \in \Lambda$ we have 
\[\abs{\sum_{t \preceq s} a_t}=\sum_{t \preceq s} \widehat{a}_{s} \leq \sum_{t \preceq s} \widehat{b}_{s}=\abs{\sum_{t \preceq s} b_t},\]
which immediately implies $\norm{\olb{a}}^\Lambda \leq \norm{\olb{b}}^\Lambda$, and so the proof of the lemma is finished.
\end{proof}

The following simple lemma, which presents basic yet important properties of the $^+$-operation in spaces $\A_\Lambda$, can be proved in a similar spirit.

\begin{lemma}\label{lem:norm of the positive part}
    Let $\Lambda\in\TT$ and $\olb{a}=(a_s)_{s \in \Lambda}$. 
    \begin{enumerate}[(1)]
    \item $\olb{a}\in(\A_\Lambda)_+$ if and only if for each $t\in\Lambda$ we have
    \[\sum_{s\preceq t}a_s\ge 0.\]
    \item It holds
    \[\norm{\olb{a}^+}^\Lambda=\sup_{t\in\Lambda}\bigg(\sum_{s \preceq t} a_s\bigg)^{+}\]
    and
    \[\norm{\olb{a}}^\Lambda=\max\big\{\norm{\olb{a}^+}^\Lambda,\norm{\olb{a}^-}^\Lambda\big\}.\]
    \item For each subset $\Omega\subseteq\Lambda$ we have
    \[\olb{a}^+\rstr\Omega=(\olb{a}\rstr\Omega)^+.\]
    \end{enumerate}
\end{lemma}




Before proceeding further, we recall the notion from \cite[Definition 3.9]{Rondos-Sobota_copies_of_separable_C(L)}. 

\begin{definition}
Let $E, F$ be normed lattices. An operator $T\colon E \rightarrow F$ \emph{preserves norms of the positive parts} (or, in short, is \emph{PNPP}), if for each $e \in E$ we have $\norm{e^+}=\norm{T(e)^+}$.\eodef
\end{definition}

Of course, if an operator $T\colon E\to F$ between two normed lattices is PNPP, then due to linearity it also preserves norms of the negative parts, that is, for each $e \in E$ we have $\norm{e^-}=\norm{T(e)^-}$. Moreover, $T$ is then a positive operator, since an element $e \in E$ is positive if and only if its negative part $e^-$ is zero. Note that we do not require that the image of $T$ is a sublattice of $F$. However, we have the following simple but useful fact.

\begin{lemma}\label{lem:pnpp_inverse_positive}
If $E, F$ are normed lattices and $T \colon E \rightarrow F$ is a PNPP injective operator onto a subspace $Z$ of $F$, then $T^{-1} \colon Z \rightarrow E$ is a positive operator.    
\end{lemma}

\begin{proof}
Let $z \in Z_+$. We find $e \in E$ such that $T(e)=z$. Then, we have
\[\norm{T^{-1}(z)^{-}}=\norm{e^-}=\norm{T(e)^{-}}=\norm{z^-}=0.\]
Thus, $T^{-1}(z)^{-}=0$, that is, $T^{-1}(z)$ is a positive element of $E$.
\end{proof}

Further, for certain classes of normed lattices, a surjective PNNP operator is automatically a Banach-lattice isometry.

\begin{lemma}
\label{lem:ordered normed spaces and PNPP}
Assume that $X$ is a normed lattice and $Y$ is a Banach lattice such that 
for each $x \in X$ and $y \in Y$ we have 
\[\norm{x}=\max\big\{{\norm{x^+}, \norm{x^-}}\big\}\quad\text{and}\quad\norm{y}=\max\big\{{\norm{y^+}, \norm{y^-}}\big\}.\]
Let $T\colon X\to Y$ be a PNPP operator. Then:
\begin{enumerate}[(i)]
	\item $T$ is an isometry into $Y$,
	\item if $T$ is a surjection onto $Y$, then $X$ is also a Banach lattice and $T$ is a lattice isometry.
\end{enumerate}
\end{lemma}

\begin{proof}
For each $x \in X$ we have
\[\norm{x}=\max\big\{\norm{x^+}, \norm{x^-}\big\}=\max\big\{\norm{(Tx)^+}, \norm{(Tx)^-}\big\}=\norm{Tx},\]
hence $T$ is an isometry into $Y$. If $T$ is onto $Y$, then this, of course, also implies that $X$ is a Banach space and so a Banach lattice. Further, by the PNPP property and Lemma \ref{lem:pnpp_inverse_positive}, 
$T$ and $T^{-1}$ are positive operators, hence $T$ preserves the lattice operations.
\end{proof}

Note that the equalities required in the assumptions of the lemma hold in all so-called \emph{abstract M-spaces} (or \emph{AM-spaces}), see \cite{lacey}. As we will see in the next subsection, the lemma will be particularly useful for the case of spaces $X=\A_{\Lambda}$ and $Y=\C(K)$ for some trees $\Lambda$ and compact spaces $K$. 

\medskip

Finally, let us also introduce some auxiliary notation.

\begin{definition}
For $\Lambda\in\TT$, $m \in \en$, and $1\le i\le m$, we denote:
\begin{itemize}
    \item by $\A_\Lambda^m$ the $m$-th power of the space $\A_\Lambda$ endowed with the maximum norm, i.e. $\A_\Lambda^m=\big(\bigoplus_{i=1}^m \A_\Lambda\big)_\infty$,
    \item by $\Sigma_m(\Lambda)$ the topological disjoint union of $m$ copies of $\Lambda$, i.e. $\Sigma_m(\Lambda)=\Lambda\times\{1,\ldots,m\}$,
    \item and by $\pi^m_i(\Lambda)$ the $i$-th copy of $\Lambda$ in $\Sigma_m(\Lambda)$, i.e. $\pi^m_i(\Lambda)=\Lambda\times\{i\}$.
\end{itemize}
\eodef
\end{definition}

Note that by virtue of Lemma \ref{A_lambda is a normed lattice} the above space $\A_\Lambda^m$ is a normed lattice when endowed with the coordinate-wise order.

\begin{definition}\label{def:chi_delta}
Let $\Lambda\in\TT$ and $m \in \en$. For $(s, i) \in \Sigma_m(\Lambda)$, by $\tilde{\chi}_{s, i}$ we denote the element
\[(a_{t, j})_{(t, j) \in \Sigma_m(\Lambda)} \in \A_{\Lambda}^m\]
such that $a_{s, i}=1$ and $a_{t, j}=0$ for $(t, j) \neq (s, i)\in\Sigma_m(\Lambda)$, and by $\tilde{\delta}_{s, i}\colon \A_{\Lambda}^m \rightarrow \er$ the function 
\[ \A_{\Lambda}^m\ni(a_{t, j})_{(t, j) \in \Sigma_m(\Lambda)}\longmapsto\tilde{\delta}_{s, i}\big((a_{t, j})_{(t, j) \in \Sigma_m(\Lambda)}\big)=\sum_{t \preceq s}a_{t, i}.\]
If $m=1$, then for every $s\in\Sigma_m(\Lambda)=\Lambda$ we will simply write $\tilde{\delta}_s$ instead of $\tilde{\delta}_{s,1}$ and $\tilde{\chi}_s$ instead of $\tilde{\chi}_{s,1}$.\eodef
\end{definition}

\noindent Notice that, for $(s, i), (t, j) \in \Sigma_m(\Lambda)$, we have
\[\tilde{\delta}_{s, i}\big(\tilde{\chi}_{t, j}\big)=
\begin{cases}
  1, & \text{ if } t \preceq s \text{ and } i=j, \\
  0, & \text{ otherwise}. \\
\end{cases}\]

The notation introduced by Definition \ref{def:chi_delta} is motivated by the fact that, when in the next subsection we isometrically identify the spaces $\A_{\Lambda}^m$ with separable spaces $\C(K)$, then the above-defined objects $\tilde{\chi}_{s,i}$ and $\tilde{\delta}_{s,i}$ will naturally correspond to the characteristic functions of certain clopen subsets of $K$  and to certain Dirac measures on $K$, respectively.

\subsection{Isometry between the spaces $\A_\Lambda$ with $\Lambda\in\TT_{\alpha+1}$ and $\C([1,\omega^\alpha])$}

We will prove that for a countable ordinal number $\alpha$ and any tree $\Lambda\in\TT_{\alpha+1}$ the spaces $\A_\Lambda$ and $\C([1,\omega^\alpha])$ are isometrically isomorphic.

We start with the following two lemmas which show that in the case of trees of countable rank condition (A.2) actually implies both conditions (A.1) and (A.3).

\begin{lemma}\label{lem:alambda_norm_finite}
    Let $\Lambda\in\TT_{\alpha+1}$ for some countable ordinal $\alpha$. If $\olb{a}\in \er^\Lambda$ satisfies condition (A.2), then $\olb{a}$ satisfies condition (A.1).
\end{lemma}
\begin{proof}
    The proof goes by transfinite induction on $\alpha$. The case $\alpha=0$ is clear.

    Let $\alpha>0$ be a countable ordinal and assume that the conclusion holds for all $\beta<\alpha$ and all tress of rank $\beta+1$. Fix $\Lambda\in\TT_{\alpha+1}$ and let $\olb{a}=(a_s)_{s\in\Lambda}\in\er^\Lambda$ satisfy condition (A.2). By the inductive assumption and Remark \ref{remark:norms_lim} (both used below at the very last equality), we have
    \[\norm{\olb{a}}^\Lambda=\sup_{t\in\Lambda}\bigg|\sum_{s \preceq t} a_s\bigg|=\max\bigg\{\abs{a_\emptyseq},\ \sup_{\substack{t\in\Lambda\\l(t)\ge1}}\bigg|\sum_{s \preceq t} a_s\bigg|\bigg\}\le\]
    \[\le\max\Bigg\{\abs{a_\emptyseq},\ \abs{a_\emptyseq}+\sup_{\substack{t\in\Lambda\\l(t)\ge1}}\Bigg|\sum_{\substack{s \preceq t\\l(s)\ge1}} a_s\Bigg|\Bigg\}=\abs{a_\emptyseq}+\sup_{\substack{t\in\Lambda\\l(t)\ge1}}\Bigg|\sum_{\substack{s \preceq t\\l(s)\ge1}} a_s\Bigg|=\]
    \[=\abs{a_\emptyseq}+\sup_{n\in\omega}\sup_{\substack{t\in\Lambda_{\langle n\rangle}}}\Bigg|\sum_{\substack{s \preceq t\\l(s)\ge1}} a_s\Bigg|=\abs{a_\emptyseq}+\sup_{n\in\omega}\sup_{t\in\Lambda^{\langle n\rangle}}\Bigg|\sum_{s\preceq t}a_{n\concat s}\Bigg|=\]
    \[=\abs{a_\emptyseq}+\sup_{n\in\omega}\norm{(a_{n\concat t})_{t\in\Lambda^{\langle n\rangle}}}^{\Lambda^{\langle n\rangle}}=\abs{a_\emptyseq}+\sup_{n\in\omega}\norm{\olb{a}\rstr\Lambda_{\langle n\rangle}}^\Lambda<\infty.\]
\end{proof}

\begin{lemma}\label{lem:alambda_norm_length}
    Let $\Lambda\in\TT_{\alpha+1}$ for some countable ordinal $\alpha$. If $\olb{a}\in \er^\Lambda$ satisfies condition (A.2), then $\olb{a}$ satisfies condition (A.3).
\end{lemma}
\begin{proof}
    As previously, the proof goes by transfinite induction on $\alpha$. The case $\alpha=0$ is again clear. 
    
    Let $\alpha>0$ be a countable ordinal and assume that the conclusion holds for all $\beta<\alpha$ and all trees of rank $\beta+1$. Fix $\Lambda\in\TT_{\alpha+1}$ and let $\olb{a}=(a_s)_{s\in\Lambda}\in\er^\Lambda$ satisfy condition (A.2). Assume, for the sake of contradiction, that there is $\eps>0$ such that for every $k\in\en$ there is $n_k\ge k+1$ for which we have
    \[\norm{\olb{a}\rstr\big\{t\in\Lambda\colon\ l(t)\ge n_k\big\}}^\Lambda=\sup_{\substack{t\in\Lambda\\l(t)\ge n_k}}\Bigg|\sum_{\substack{s\preceq t\\l(s)\ge n_k}}a_s\Bigg|>\eps.\]
    For each $k\in\en$ there is $t_k\in\Lambda$ with $l(t_k)\ge n_k$ and such that
    \[\tag{$*$}\Bigg|\sum_{\substack{s\preceq t_k\\l(s)\ge n_k}}a_s\Bigg|>\eps.\]
    As each tree $\Lambda^{\langle l\rangle}$ has rank $<\alpha+1$, by the inductive assumption for every $l\in\en$ there is $m_l\in\en$ such that for every $m\ge m_l$ we have
    \[\norm{\olb{a}\rstr\big\{t\in\Lambda_{\langle l\rangle}\colon\ l(t)\ge m+1\big\}}^\Lambda<\eps,\]
    hence, without loss of generality, we may assume that $t_k(0)\neq t_{k'}(0)$ for each $k\neq k'\in\en$.

    By ($*$), for every $k\in\en$ we have
    \[\norm{\olb{a}\rstr\Lambda_{\langle t_k(0)\rangle}}^\Lambda=\sup_{\substack{t\in\Lambda\\t(0)=t_k(0)}}\Bigg|\sum_{\substack{s\preceq t\\l(s)\ge1}}a_{s}\Bigg|\ge\Bigg|\sum_{\substack{s\preceq t_k\\l(s)\ge1}}a_{s}\Bigg|=
    \Bigg|\sum_{\substack{s\preceq t_k\\1\le l(s)<n_k}}a_{s}+\sum_{\substack{s\preceq t_k\\l(s)\ge n_k}}a_{s}\Bigg|\ge\]\[\ge\Bigg|\sum_{\substack{s\preceq t_k\\l(s)\ge n_k}}a_{s}\Bigg|-\Bigg|\sum_{\substack{s\preceq t_k\\1\le l(s)<n_k}}a_{s}\Bigg|
    >\eps-\Bigg|\sum_{\substack{s\preceq t_k\\1\le l(s)<n_k}}a_{s}\Bigg|.\]
    From the above inequality, since $\lim_{k\to\infty}\norm{\olb{a}\rstr\Lambda_{\langle t_k(0)\rangle}}^\Lambda=0$ by condition (A.2), we get
    \[\liminf_{k\to\infty}\Bigg|\sum_{\substack{s\preceq t_k\\1\le l(s)<n_k}}a_{s}\Bigg|\ge\eps.\]
    But then, we have
    \[0=\lim_{k\to\infty}\norm{\olb{a}\rstr\Lambda_{\langle t_k(0)\rangle}}^\Lambda\ge\liminf_{k\to\infty}\Bigg|\sum_{\substack{s\preceq t_k\\1\le l(s)<n_k}}a_{s}\Bigg|\ge\eps>0,\]
    which is clearly impossible.
\end{proof}

\begin{cor}\label{cor:condition_a2}
    Let $\Lambda\in\TT_{\alpha+1}$ for some countable ordinal $\alpha$. If $\olb{a}\in \er^\Lambda$ satisfies condition (A.2), then $\olb{a}\in \A_\Lambda$.
\end{cor}

We are in the position to prove the aforementioned isometric characterization of the Banach spaces $\C([1,\omega^\alpha])$ for countable ordinals $\alpha$ in terms of the tree spaces $\A_\Lambda$, where $\Lambda\in\TT_{\alpha+1}$.

\begin{thm}
\label{A_lambda is C(omega^alpha)}
Let $\alpha$ be a countable ordinal, $m \in \en$, and $\Lambda \in\TT_{\alpha+1}$. Then, $\A_{\Lambda}^m$  is a Banach lattice which is lattice isometric to $\C([1, \omega^{\alpha} m])$. 

In fact, there exist a surjective lattice isometry $T \colon \A_\Lambda^m \rightarrow \C([1, \omega^{\alpha} m])$ and a homeomorphism $\rho\colon\Sigma_m(\Lambda) \rightarrow [1, \omega^{\alpha}m]$ such that the equality
\[\tilde{\delta}_{s, i}\big(\tilde{\chi}_{t, j}\big)=\big\langle \delta_{\rho_{s, i}}, T\big(\tilde{\chi}_{t, j}\big) \big\rangle\]
holds for every $(s, i), (t, j) \in \Sigma_m(\Lambda)$.
\end{thm}

\begin{proof}
Clearly, it is enough to consider only the case when $m=1$.

By Lemma \ref{A_lambda is a normed lattice}, $\A_\Lambda$ is a normed lattice. To prove the first part of the theorem, it is enough by Lemma \ref{lem:ordered normed spaces and PNPP} to show that there exists a surjective PNPP operator $T \colon \A_{\Lambda}\to \C([1,\omega^\alpha])$.

We proceed by transfinite induction on $\alpha$. The case $\alpha=0$ is clear. Thus, let $\alpha>0$ be a countable ordinal and assume that the conclusion holds for all $\beta<\alpha$. Let $\Lambda\in\TT_{\alpha+1}$. 

Assume first that $\alpha=\beta+1$ for some countable ordinal $\beta$. For each $n \in \omega$ the tree $\Lambda^{\langle n\rangle}$ has rank $\beta+1$, so we use the inductive assumption to find a surjective lattice isometry
$T_n\colon \A_{\Lambda^{\langle n\rangle}}\to \C([1, \omega^{\beta}])$ 
and a homeomorphism $\rho_n \colon \Lambda^{\langle n\rangle} \rightarrow [1, \omega^{\beta}]$ such that 
\[\tilde{\delta}_{s}(\tilde{\chi}_{t})=\big\langle \delta_{\rho_{n}(s)}, T_n(\tilde{\chi}_{t}) \big\rangle\]
for every $s,t \in \Lambda$. We moreover fix for each $n \in \omega$ the canonical surjective lattice isometry 
\[S_n\colon\C([1, \omega^{\beta}])\to\C\big(\big[\omega^{\beta}n+1,\ \omega^{\beta}(n+1)\big]\big),\]
that is, $S_n$ is induced by the inverse mapping $\xi_n^{-1}$ to the natural homeomorphism $\xi_n\colon[1, \omega^{\beta}]\to\big[\omega^{\beta}n+1,\ \omega^{\beta}(n+1)\big]$. Notice that, for each $n \in \omega$ and $s, t \in \Lambda^{\langle n\rangle}$, we have 
\[\tilde{\delta}_s(\tilde{\chi}_t)=\big\langle\delta_{\rho_n(s)}, T_n(\tilde{\chi}_t) \big\rangle=\big\langle \delta_{(\xi_n \circ \rho_n)(s)}, (S_n \circ T_n)(\tilde{\chi}_t) \big\rangle.\]

Now, we define the desired lattice isometry $T\colon \A_{\Lambda}\to\C([1, \omega^{\alpha}])$ as follows. Let $\olb{a}=(a_s)_{s \in \Lambda} \in \A_{\Lambda}$. Using the inductive assumption, for each $n\in\omega$ we find a function $g_n \in \C\big(\big[\omega^{\beta}n+1, \omega^{\beta}(n+1)\big]\big)$
such that
\[g_n=S_n\big(T_n\big((a_{n\concat s})_{s \in \Lambda^{\langle n\rangle}}\big)\big),\]
so in particular we have $\norm{g_n}=\norm{(a_{n\concat s})_{s \in \Lambda^{\langle n\rangle}}}^{\Lambda^{\langle n\rangle}}$. Then, we define 
\[f=a_{\emptyseq} \chi_{[1, \omega^{\alpha}]}+\sum_{n=0}^{\infty} g_n.\]
Clearly, $f$ is a well-defined function from $[1,\omega^{\alpha}]$ into $\er$. Since for each $m \in \omega$ we have 
\[\norm{\sum_{n=m+1}^{\infty} g_n} = \sup_{n \geq m+1} \norm{g_n}=\sup_{n\ge m+1}\norm{(a_{n\concat s})_{s \in \Lambda^{\langle n\rangle}}}^{\Lambda^{\langle n\rangle}}=\sup_{n \geq m+1}\norm{\olb{a}\rstr\Lambda_{\langle n\rangle}}^{\Lambda},\]
by condition (A.2) we get
\[\lim_{m\to\infty}\norm{\sum_{n=m+1}^{\infty} g_n}=0,\]
and so it follows that $f$ is the uniform limit of functions of the form
\[a_{\emptyseq} \chi_{[1, \omega^{\alpha}]}+\sum_{n=0}^{m} g_n,\]
where $m\in\omega$. Consequently, $f$ is a well-defined continuous function on $[1, \omega^{\alpha}]$. We then set $T(\olb{a})=f$. Note that by the computation as above we have
\[\norm{T(\olb{a})}=\norm{f}\le\abs{a_\emptyseq}+\sup_{n\in\omega}\norm{\olb{a}\rstr\Lambda_{\langle n\rangle}}^{\Lambda}\le 2\abs{a_\emptyseq}+\norm{\olb{a}}^\Lambda\le3\norm{\olb{a}}^\Lambda.\]
It follows that $T$ is a well-defined bounded operator mapping $\A_\Lambda$ into $\C([1,\omega^\alpha])$. We need to show that $T$ is surjective and PNPP.

\medskip

We first check that $T$ is PNPP. Fix $\olb{a}=(a_s)_{s\in\Lambda}\in \A_\Lambda$ and let $(g_n)_{n\in\omega}$ be the sequence of functions as in the definition of $T(\olb{a})$. We show that $\norm{T(\olb{a})^+}=\norm{\olb{a}^+}^\Lambda$. To this end, by \cite[Lemma 3.6]{Rondos-Sobota_copies_of_separable_C(L)} (1), the inductive assumption (2), Lemma \ref{lem:norm of the positive part} (3), and Lemma \ref{lem:sup_max_0} (4), we have 
\[\norm{T(\olb{a})^+}\overset{(1)}{=}\max\left\{a_\emptyseq+\sup_{n \in \omega}\norm{g_n^+},\ 0\right\} 
=\max\left\{a_\emptyseq+\sup_{n \in \omega}\norm{S_n^{-1}(g_n)^+},\ 0\right\}=\]
\[\overset{(2)}{=}\max\left\{a_\emptyseq+\sup_{n \in \omega}\norm{T_n^{-1}\big(S_n^{-1}(g_n)\big)^+}^{\Lambda^{\langle n\rangle}},\ 0\right\}=\]
\[=\max\Bigg\{a_\emptyseq+\sup_{n \in \omega}\norm{\big(a_{n\concat s}\big)^+_{s \in \Lambda^{\langle n\rangle}}}^{\Lambda^{\langle n\rangle}},\ 0\Bigg\}=\]
\[\overset{(3)}{=}\max\Bigg\{a_\emptyseq+\sup_{n \in \omega}\sup_{t\in \Lambda^{\langle n\rangle}}\bigg(\sum_{s \preceq t} a_{n\concat s}\bigg)^+,\ 0\Bigg\}=\]
\[=\max\Bigg\{a_{\emptyseq}+\sup_{n \in \omega}\sup_{t\in \Lambda_{\langle n\rangle}}\bigg(\sum_{\langle n\rangle \preceq s \preceq t} a_s\bigg)^+,\ 0\Bigg\}=\]
\[=\left(a_{\emptyseq}+\sup_{n \in \omega}\sup_{t\in \Lambda_{\langle n\rangle}}\bigg(\sum_{\langle n\rangle \preceq s \preceq t} a_s\bigg)^+\right)^+=\]
\[\overset{(4)}{=}\sup_{t\in\Lambda}\bigg(\sum_{s \preceq t} a_s\bigg)^+\overset{(3)}{=}\norm{\olb{a}^+}^{\Lambda}.\]
It follows that $T$ is indeed PNPP.

\medskip

We now show that $T$ is surjective. Given a function $f \in \C([1, \omega^{\alpha}])$, we write $f$ in the unique way in the form 
\[f=a \chi_{[1, \omega^{\alpha}]}+\sum_{n=0}^{\infty} g_n,\]
where $a \in \er$, and the functions $g_n \in \C\big(\big[\omega^{\beta}n+1,\ \omega^{\beta}(n+1)\big]\big)$ are such that $(\norm{g_n})_{n \in \omega} \in c_0$, see \cite[Lemma 3.7]{Rondos-Sobota_copies_of_separable_C(L)}. For each $n\in\omega$, let 
\[\big(b^n_s\big)_{s \in \Lambda^{\langle n\rangle}}=T_n^{-1}\big(S_n^{-1}(g_n)\big),\]
so $\big(b^n_s\big)_{s \in \Lambda^{\langle n\rangle}}\in \A_{\Lambda^{\langle n\rangle}}$ and $\norm{g_n}=\norm{\big(b^n_s\big)_{s \in \Lambda^{\langle n\rangle}}}^{\Lambda^{\langle n\rangle}}$. Now, let $\olb{a}=(a_s)_{s \in \Lambda}$ be the element of $\er^\Lambda$ defined for every $s\in\Lambda$ by the formula:
\[a_s=\begin{cases}
    a,&\text{ if }s=\emptyseq,\\
    b^n_{sh(s)},&\text{ if }\langle n\rangle\preceq s\text{ for some }n\in\omega.
\end{cases}
\]

We check that $\olb{a}\in \A_\Lambda$. 
By Corollary \ref{cor:condition_a2}, it is enough to prove that $\olb{a}$ satisfies condition  (A.2), that is, that for each $s \in NL(\Lambda)$ we have
\[\lim_{n \rightarrow \infty} \norm{(a_t)_{t\in\Lambda}\rstr\Lambda_{s^\concat n}}^{\Lambda}=0.\]
It is however enough to check this equality for the case $s=\emptyseq$, since the other cases follow from the inductive assumption, as for any $k\in\omega$ and nonleaf $s\in\Lambda_{\langle k\rangle}$ we have
\[\lim_{n \rightarrow \infty} \norm{(a_t)_{t\in\Lambda}\rstr\Lambda_{s\concat n}}^{\Lambda}=\lim_{n \rightarrow \infty} \norm{\big(b_t^k\big)_{t\in\Lambda^{\langle k\rangle}}\rstr\big(\Lambda^{\langle k\rangle}\big)_{sh(s)\concat n}}^{\Lambda^{\langle k\rangle}}=0.\]
To this end, we have
\begin{equation}
\nonumber
\begin{aligned}
\lim_{n \rightarrow \infty} \norm{(a_t)_{t\in\Lambda}\rstr\Lambda_{\langle n\rangle}}^{\Lambda}=\lim_{n \rightarrow \infty} \norm{\big(b_t^n\big)_{t\in\Lambda^{\langle n\rangle}}}^{\Lambda^{\langle n\rangle}}=\lim_{n \rightarrow \infty} \norm{g_n}=0,
\end{aligned}
\end{equation}
as $(\norm{g_n})_{n \in \omega} \in c_0$, so condition (A.2) indeed holds. Consequently, $\olb{a}\in \A_\Lambda$.

By the construction of $T(\olb{a})$ and the uniqueness of the form of $f$ (see \cite[Lemma 3.7]{Rondos-Sobota_copies_of_separable_C(L)}), we have $T(\olb{a})=f$, and so $T$ is a surjection.

\medskip

Finally, we define the mapping $\rho \colon \Lambda \rightarrow [1, \omega^{\alpha}]$ by setting
$\rho(\emptyseq)=\omega^{\alpha}$, and $\rho(s)=\rho_n(sh(s))$ if $s \in \Lambda_{\langle n\rangle}$. Then, it is simple to prove using Lemma \ref{lemma:tree_topology} that $\rho$ is a homeomorphism. Also, the required equality
\[\tilde{\delta}_{s}(\tilde{\chi}_{t})=\big\langle \delta_{\rho({s})}, T(\tilde{\chi}_{t}) \big\rangle\]
for $s,t\in\Lambda$ can be checked by an easy argument appealing to the inductive assumption and having in mind that $\tilde{\delta}_s(\tilde{\chi}_t)$ equals to $1$ if $t\preceq s$, and $0$ otherwise.

\bigskip

Assume now that $\alpha$ is a limit ordinal. We proceed similarly as before. For each $n \in \omega$ let $\beta_n=r(\langle n\rangle)$, and note that $(\beta_n)_{n\in\omega}$ is a (strictly increasing) sequence of ordinals converging to $\alpha$. 
For each $n\in\omega$, as the tree $\Lambda^{\langle n\rangle}$ has rank $\beta_n+1$, by the inductive assumption we can find a surjective lattice isometry $T_n\colon \A_{\Lambda^{\langle n\rangle}}\to\C([1, \omega^{\beta_n}])$ and a homeomorphism $\rho_n \colon \Lambda^{\langle n\rangle}\to[1, \omega^{\beta_n}]$ such that 
\[\tilde{\delta}_{s}(\tilde{\chi}_{t})=\big\langle \delta_{\rho_{n}(s)}, T_n(\tilde{\chi}_{t}) \big\rangle\]
for every $s,t \in \Lambda^{\langle n\rangle}$. Also, for each $n\ge 1$, by the Mazurkiewicz--Sierpi\'{n}ski theorem, there exists a homeomorphism
\[\xi_n \colon [1,\omega^{\beta_n}] \rightarrow \big[\omega^{\beta_{n-1}}+1, \omega^{\beta_n}\big];\]
let
\[S_n\colon\C([1, \omega^{\beta_n}])\rightarrow \C\big(\big[\omega^{\beta_{n-1}}+1, \omega^{\beta_n}\big]\big) \]
be the surjective lattice isometry induced by $\xi_n^{-1}$.
Also, let
\[\xi_0 \colon [1, \omega^{\beta_0}]\rightarrow [1, \omega^{\beta_0}]\]
and
\[S_0\colon\C([1, \omega^{\beta_0}]) \rightarrow \C([1, \omega^{\beta_0}])\]
be the identity mappings. Notice that, for each $n \in \omega$ and $s, t \in \Lambda^{\langle n\rangle}$, we have 
\[\tilde{\delta}_s(\tilde{\chi}_t)=\big\langle\delta_{\rho_n(s)}, T_n(\tilde{\chi}_t) \big\rangle=\big\langle \delta_{(\xi_n \circ \rho_n)(s)}, (S_n \circ T_n)(\tilde{\chi}_t) \big\rangle.\]

Let $\olb{a}=(a_s)_{s \in \Lambda} \in \A_{\Lambda}$. Using the inductive assumption,  we find functions $g_0\in\C([1,\omega^{\beta_0}])$ and $g_n \in \C\big(\big[\omega^{\beta_{n-1}}+1, \omega^{\beta_n}\big]\big)$ for $n\ge 1$, such that
\[g_n=S_n\big(T_n\big((a_{n\concat s})_{s \in \Lambda^{\langle n\rangle}}\big)\big),\]
so in particular we have $\norm{g_n}=\norm{(a_{n\concat s})_{s \in \Lambda^{\langle n\rangle}}}^{\Lambda^{\langle n\rangle}}$. Then, we define 
\[T(\olb{a})=a_{\emptyseq} \chi_{[1, \omega^{\alpha}]}+\sum_{n=0}^{\infty} g_n.\]
The reasoning that the mapping $T\colon\C([1, \omega^{\alpha}]) \rightarrow \A_{\Lambda}$
is the required surjective PNPP operator and $\rho \colon \Lambda \rightarrow [1, \omega^{\alpha}]$ is the required homeomorphism is now basically the same as in the successor case.
%
\end{proof}

\begin{cor}
    For any countable ordinal $\alpha$ and any pair of trees $\Lambda,\Lambda'\in\TT_{\alpha+1}$, the Banach spaces $\A_\Lambda$ and $\A_{\Lambda'}$ are isometrically isomorphic.
\end{cor}

Theorem \ref{A_lambda is C(omega^alpha)} quickly yields the (isometrically) isomorphic description of spaces $\C(K)$ for \textit{all} countable compact spaces $K$.

\begin{cor}
    For any countable compact space $K$ there exists a tree $\Lambda$ of rank $ht(K)$ such that $\C(K)$ is  isomorphic to the space $\A_\Lambda$ and isometrically isomorphic to the space $\A_\Lambda^m$, where $m=\abs{K^{(ht(K)-1)}}$.
\end{cor}
\begin{proof}
    By the Mazurkiewicz--Sierpi\'{n}ski theorem, $K$ is homeomorphic to the space $[1,\omega^\alpha m]$, where $\alpha=ht(K)-1$ and $m=\abs{K^{(\alpha)}}$. Thus, $K$ is homeomorphic to the disjoint union of $m$ copies of the space $[1,\omega^\alpha]$, and so the space $\C(K)$ is isometrically isomorphic to the sum space $\big(\bigoplus_{i=1}^m\C([1,\omega^\alpha])\big)_\infty$. By Theorem \ref{A_lambda is C(omega^alpha)} we get the required isometry from $\C(K)$ onto $\A_\Lambda^m$. For the second isomorphism, note that from the Bessaga--Pe\l czy\'{n}ski isomorphic classification of the spaces of the form $\C([1,\gamma])$ for countable ordinals $\gamma$ (\cite[Theorem 1]{BessagaPelcynski_classification}) it follows that $\C([1,\omega^\alpha m])$ is isomorphic to $\C([1,\omega^\alpha])$ and again appeal to Theorem \ref{A_lambda is C(omega^alpha)}.
\end{proof}

\subsection{Isometry between the spaces $\A_{\el}$ and $\C(2^\omega)$\label{sec:el_cantor}}

We will now demonstrate that the spaces $\A_{\el}$ and $\C(2^\omega)$ are isometrically isomorphic. 

Recall that the linear subspace $\mathcal{FS}(\A_\el)$ of $\A_\el$, consisting of all finitely supported elements, is norm-dense in $\A_\el$ (Lemma \ref{lem:fin-approx}). 

\begin{lemma}
\label{fin-norm} The mapping $\widehat{T}\colon\mathcal{FS}(\A_\el)\to\C(2^\omega)$, given for every $(a_s)_{s\in\el}\in\mathcal{FS}(\A_\el)$ by the formula
\[\widehat{T}\big((a_s)_{s\in\el}\big)=\sum_{s \in\el} a_s\chi_{[Q(s)]},\]
is a PNPP isometry of $\mathcal{FS}(\A_\el)$ into $\C(2^\omega)$.
\end{lemma}

\begin{proof}
It is clear that $\widehat{T}$ is a linear mapping, so we first show that it is 
%
a PNPP operator. Note that if an element of $\mathcal{FS}(\A_\el)$ is supported by a set $F$, then it is supported by any trunk $F'\in\mathfrak{Tr}(\el)$ containing $F$---we will thus proceed by induction on length of trunks. The case $n=0$ is trivial, as then the only trunk of length $0$ is $\{\emptyseq\}$. Thus, let us assume that $n\ge1$ and that the desired equality holds for all elements of $\mathcal{FS}(\A_\el)$ supported by trunks of $\el$ of length $n-1$. 

Let $\olb{a}\in\mathcal{FS}(\A_\el)$ be supported by a trunk $F$ of length $n$. Then, using the argument as for \cite[Lemma 3.6]{Rondos-Sobota_copies_of_separable_C(L)} (1) and the inductive assumption (2), 
we have

\[\norm{\bigg(\sum_{s \in F} a_s\chi_{[Q(s)]}\bigg)^+}=\norm{\Bigg(a_\emptyseq+\sum_{\substack{k\in\omega\\\langle k\rangle\in F}}\sum_{s \in F_{\langle k\rangle}} a_s\chi_{[Q(s)]}\Bigg)^+}=\]
\[\overset{(1)}{=}\max\Bigg\{a_{\emptyseq}+\sup_{\substack{k \in \omega\\\langle k\rangle\in F}}\norm{\bigg(\sum_{s \in F_{\langle k\rangle}} a_s\chi_{[Q(s)]}\bigg)^+},\ 0\Bigg\}=\]
\[=\max\Bigg\{a_{\emptyseq}+\sup_{\substack{k \in \omega\\\langle k\rangle\in F}}\norm{\bigg(\sum_{t \in F^{\langle k\rangle}} a_{k\concat t}\chi_{[Q(k\concat t)]}\bigg)^+},\ 0\Bigg\}=\]
\[=\max\Bigg\{a_{\emptyseq}+\sup_{\substack{k \in \omega\\\langle k\rangle\in F}}\norm{\bigg(\sum_{t \in F^{\langle k\rangle}} a_{k\concat t}\chi_{[Q(t)]}\bigg)^+},\ 0\Bigg\}=\]
\[\overset{(2)}{=}\max\Bigg\{a_{\emptyseq}+\sup_{\substack{k \in \omega\\\langle k\rangle\in F}}\norm{\olb{b}_k^+}^\el,\ 0\Bigg\},\]
where each $\olb{b}_k=(b_{k,t})_{t\in\el}$ is an element of $\mathcal{FS}(\A_\el)$ supported by the trunk $F^{\langle k\rangle}$ of size $n-1$ and defined as follows:
\[b_{k,t}=\begin{cases}
    a_{k\concat t}&,\text{ if }t\in F^{\langle k\rangle},\\
    0&,\text{ otherwise,}
\end{cases}
\]
for $t\in\el$. By Lemma \ref{lem:norm of the positive part}, for every $k\in\omega$ such that $\langle k\rangle\in F$ we have
\[\norm{\olb{b}_k^+}^\el=\sup_{t\in\el}\bigg(\sum_{s\preceq t}b_{k,s}\bigg)^+=\sup_{t\in F^{\langle k\rangle}}\bigg(\sum_{s\preceq t}a_{k\concat s}\bigg)^+=\sup_{t\in F_{\langle k\rangle}}\bigg(\sum_{\langle k\rangle\preceq s\preceq t}a_{s}\bigg)^+,\]
hence, by Lemma \ref{lem:sup_max_0} (3) and Lemma \ref{lem:norm of the positive part} (4), we get
\[\norm{\bigg(\sum_{s \in F} a_s\chi_{[Q(s)]}\bigg)^+}=
\max\Bigg\{a_{\emptyseq}+\sup_{\substack{k \in \omega\\\langle k\rangle\in F}}\sup_{t\in F_{\langle k\rangle}}\bigg(\sum_{\substack{\langle k\rangle \preceq s \preceq t}} a_s\bigg)^+,\ 0\Bigg\}=\]
\[=\max\Bigg\{a_{\emptyseq}+\sup_{\substack{t\in F\\l(t)\ge1}}\Bigg(\sum_{\substack{s \preceq t\\l(s)\ge1}} a_s\Bigg)^+,\ 0\Bigg\}\overset{(3)}{=}\sup_{t\in F}\bigg(\sum_{\substack{s \preceq t}}a_s\bigg)^+\overset{(4)}{=}\norm{\olb{a}^+}^\el,
\]
and so the required equality holds also for trunks of length $n$. Consequently, $\widehat{T}$ is a PNPP operator.
%

Finally, Lemma \ref{lem:ordered normed spaces and PNPP} implies that $\widehat{T}$ is an isometry. The proof is thus finished.
\end{proof}

\begin{remark}
    Note that from the definition of the operator $\widehat{T}$ in Lemma \ref{fin-norm} it follows that for every $s\in\el$ we have $\widehat{T}(\tilde{\chi}_s)=\chi_{[Q(s)]}$, that is, the element $\tilde{\chi}_s$ corresponds to the characteristic function of the clopen set $[Q(s)]$.
\end{remark}

We are in the position to describe a lattice isometry from the space $\A_{\el}$ onto $\C(2^\omega)$.

\begin{thm}
\label{A_Lambda is C(2^omega)}
The space $\A_{\el}$ is a Banach lattice. The mapping $T\colon \A_{\el}\to\C(2^\omega)$, given for every $(a_s)_{s \in \el}\in \A_\el$ by the formula
\[T\big((a_s)_{s \in \el}\big)=\sum_{s \in \el} a_s \chi_{[Q(s)]},\]
is a lattice isometry onto $\C(2^\omega)$. 

Moreover, the homeomorphic embedding $R \colon {\el} \rightarrow 2^{\omega}$ satisfies for every $s, t \in \el$ the equality 
\[\tilde{\delta}_{s}(\tilde{\chi}_{t})=\big\langle \delta_{R(s)}, T(\tilde{\chi}_{t}) \big\rangle.\] 
\end{thm}

\begin{proof}
We first check that for each $(a_s)_{s \in\el} \in \A_{\el}$ the sum 
\[\sum_{s \in \el} a_s \chi_{[Q(s)]}\]
is a well-defined element of $\C(2^{\omega})$. Let $\olb{a}=(a_s)_{s\in\el}\in \A_\el$. Let $\eps>0$. By Lemma \ref{lem:fin-approx}, the net $(\olb{a}\rstr F)_{F \in \mathfrak{Tr}(\el)}$ converges to $\olb{a}$ in the $\el$-norm, so there is $F\in\mathfrak{Tr}(\el)$ such that for any $F'\in\mathfrak{Tr}(\el)$ with $F\subseteq F'$ we have
\[\norm{\olb{a}\rstr F'-\olb{a}\rstr F}^\el<\eps/2.\]
Since by Lemma \ref{fin-norm} for every $F'\in\mathfrak{Tr}(\el)$ with $F\subseteq F'$ we have
\[\norm{\sum_{s \in F'\setminus F} a_s\chi_{[Q(s)]}}=\norm{\olb{a}\rstr(F'\setminus F)}^\el=\norm{\olb{a}\rstr F'-\olb{a}\rstr F}^\el,\]
it follows that for any $F'\in\mathfrak{Tr}(\el)$ with $F\subseteq F'$ it holds
\[\norm{\sum_{s \in F'\setminus F} a_s\chi_{[Q(s)]}}<\eps/2.\]
Consequently, for any $F_1,F_2\in\mathfrak{Tr}(\el)$ with $F\subseteq F_1,F_2$ we get
\[\norm{\sum_{s \in F_1} a_s\chi_{[Q(s)]}-\sum_{s \in F_2} a_s\chi_{[Q(s)]}}\le\]
\[\le\norm{\sum_{s \in F_1} a_s\chi_{[Q(s)]}-\sum_{s \in F} a_s\chi_{[Q(s)]}}+\norm{\sum_{s \in F_2} a_s\chi_{[Q(s)]}-\sum_{s \in F} a_s\chi_{[Q(s)]}}=\]
\[=\norm{\sum_{s \in F_1\setminus F} a_s\chi_{[Q(s)]}}+\norm{\sum_{s \in F_2\setminus F} a_s\chi_{[Q(s)]}}<\eps/2+\eps/2=\eps.\]
As $\eps$ was arbitrary, we obtain that the net
\[\Bigg(\sum_{s \in F} a_s\chi_{[Q(s)]}\Bigg)_{F \in \mathfrak{Tr}(\el)}\]
is Cauchy in the supremum norm and so it converges to a function in $\C(2^{\omega})$, that is, in other words, the sum 
$\sum_{s\in\el}a_s\chi_{[Q(s)]}$
exists in $\C(2^{\omega})$. It follows that the mapping $T$ is well-defined. It is also clearly that $T$ is linear.  

\medskip

By Lemmas \ref{lem:fin-approx} and \ref{fin-norm} and the continuity of the lattice operations, for every $\olb{a}=(a_s)_{s\in\el}\in \A_\el$ we have
\[\norm{\olb{a}^+}^\el=\lim_{F \in\mathfrak{Tr}(\el)} \norm{\olb{a}^+\rstr F}^\el=\lim_{F\in\mathfrak{Tr}(\el)}\norm{(\olb{a}\rstr F)^+}^\el=\]
\[=\lim_{F \in \mathfrak{Tr}(\el)} \norm{\bigg(\sum_{s \in F} a_s\chi_{[Q(s)]}\bigg)^{+}}=\norm{\bigg(\sum_{s \in \el} a_s\chi_{[Q(s)]}\bigg)^{+}}=\norm{T(\olb{a})^+},\]
and hence $T$ is a PNPP operator.

\medskip

Now, in view of Lemmas 
\ref{A_lambda is a normed lattice} and \ref{lem:ordered normed spaces and PNPP}, in order to establish the first part of the conclusion it suffices to check that $T$ is surjective. Thus, we fix $f\in\C(2^\omega)$. To justify that $f\in T[\A_\el]$, we prove that the system $\olb{a}=(a_s)_{s \in \el}\in\er^\el$, defined for every $s\in\el$ by
\[a_s=\begin{cases}
f\big(R(\emptyseq)\big)&,\text{ if }s=\emptyseq,\\
f\big(R(s)\big)-f\big(R(pred(s))\big)&,\text{ otherwise},
\end{cases}\]
belongs to the space $\A_{\el}$, and that in this case we have $T(\olb{a})=f$.

Firstly, for any $t=\langle t_1,\ldots,t_n\rangle \in \el$ of length $n>0$ we have 
\[
\abs{\sum_{s \preceq t} a_s}=\]
\[=\bigg|f\big(R(\emptyseq)\big)+\Big(f\big(R(\langle t_1\rangle)\big)-f\big(R(\emptyseq)\big)\Big)+\Big(f\big(R(\langle t_1, t_2\rangle)\big)-f\big(R(\langle t_1\rangle)\big)\Big)+\ldots\]
\[\ldots+
\Big(f\big(R(t)\big)-f\big(R(\langle t_1, \ldots, t_{n-1}\rangle)\big)\Big) \bigg|=\abs{f\big(R(t)\big)}\le\norm{f}.\]
Hence, $\norm{\olb{a}}^{\el} \leq \norm{f}<\infty$, and so condition (A.1) is satisfied.

For conditions (A.2) and (A.3), let $d=d_{2^\omega}$ be the canonical metric on $2^\omega$ and fix $\eps>0$. By the uniform continuity of $f$ there exists $n_0 \in\en$ such that $\abs{f(x)-f(y)}<\ep$ whenever $x, y \in 2^{\omega}$ satisfy $d(x, y) < 1/{n_0}$. Let $s\in\el$. Notice that for each $n \geq n_0$, if $s\concat n \preceq t$, then $d\big(R(t), R(s\concat n)\big)<1/n\le1/n_0$, and thus, by a similar computation as above, we have 
\[\norm{\olb{a}\rstr\el_{s\concat n}}^{\el}=\sup_{\substack{t\in\el\\s\concat n\preceq t}}\abs{\sum_{s\concat n \preceq u \preceq t} a_u}=\sup_{\substack{t\in\el\\s\concat n\preceq t}}\abs{f\big(R(t)\big)-f\big(R(s\concat n)\big)}\le\ep,\]
hence condition (A.2) is satisfied as well.

Let $n\ge n_0$. For any $t_1, t_2 \in \el$ such that $l(t_1) \geq n$ and $t_1 \preceq t_2$ we have $d\big(R(t_1), R(t_2)\big) < 1/n$, similarly as above. 
Hence, 
\[\norm{\olb{a}\rstr\{s\in\el\colon l(s) \geq n\}}^{\el}=\sup \Bigg\{\abs{\sum_{t_1 \preceq s \preceq t_2} a_s}\colon\ t_1, t_2 \in \el,\ l(t_1) = n,\ t_1\preceq t_2\Bigg\}=\]
\[=\sup \bigg\{\abs{f\big(R(t_2)\big)-f\big(R(t_1)\big)}\colon\ t_1, t_2 \in \el,\ l(t_1) = n,\ t_1\preceq t_2\bigg\}\le \ep,\]
and thus also condition (A.3) is satisfied. 

It follows that $\olb{a}\in \A_\el$, and so
\[T(\olb{a})=f\big(R(\emptyseq)\big) \chi_{[Q(\emptyseq)]}+\sum_{\substack{s \in \el\\s \neq \emptyseq}}\Big(f\big(R(s)\big)-f\big(R(pred(s))\big)\Big) \chi_{[Q(s)]}\]
is a continuous function on $2^{\omega}$. Moreover, it follows that this function coincides with $f$ on the range of $R$. Indeed, it is clear that $\big(T(\olb{a})\big)\big(R(\emptyseq)\big)=f\big(R(\emptyseq)\big)$, so fix a sequence $s=\langle s_1,\ldots,s_n\rangle \in \el$ of length $n>0$ and observe that, for any $t\in\el$, we have $R(s) \in [Q(t)]$ if and only if $t \preceq s$. Thus, it holds
\[\big(T(\olb{a})\big)\big(R(s)\big)=f\big(R(\emptyseq)\big)+\Big(f\big(R(\langle s_1\rangle)\big)-f\big(R(\emptyseq)\big)\Big)+\]
\[+\Big(f\big(R(\langle s_1, s_2\rangle)\big)-f\big(R(\langle s_1\rangle)\big)\Big)+\ldots\]
\[\ldots+
\Big(f\big(R(s)\big)-f\big(R(\langle s_1, \ldots, s_{n-1}\rangle)\big)\Big)=f\big(R(s)\big).\]
Consequently, by the continuity of $f$ and $T(\olb{a})$ and the density of the range of $R$ in $2^{\omega}$, we get the equality $T(\olb{a})=f$, which shows that $T$ is surjective, as required.

\medskip

Finally, for the second part of the conclusion, we check that for each $s, t \in \el$ we have 
\[\tilde{\delta}_{s}(\tilde{\chi}_{t})=\big\langle \delta_{R(s)}, T(\tilde{\chi}_{t})\big\rangle.\] 
Let $s,t\in\el$. Recall that $\tilde{\delta}_s(\tilde{\chi}_t)$ equals to $1$ if $t\preceq s$, and $0$ otherwise. Since it holds 
\[\big\langle \delta_{R(s)}, T(\tilde{\chi}_{t})\big\rangle=T(\tilde{\chi}_{t})(R(s))=\chi_{[Q(t)]}(R(s)),\]
we similarly have that $\big\langle \delta_{R(s)}, T(\tilde{\chi}_{t}) \big\rangle$ equals $1$ if $t \preceq s$, and $0$ otherwise, as required.
\end{proof}

\subsection{A Markushevich basis and the dual of a space $\A_\Lambda$}

By the results presented in the previous subsections, the space dual to a given space $\A_\Lambda$ is isometrically isomorphic to the space $\M(L)$ for some metrizable compact space $L$. In addition to this observation, we note the following useful facts.

\begin{prop}
\label{prop:density}
Let $\Lambda \in \TT$ and $m \in \en$.
\begin{enumerate}[(i)]
    \item For each $(s,i)\in\Sigma_m(\Lambda)$ we have:
        \[\norm{\tilde{\chi}_{s, i}}^\Lambda=1,\quad\tilde{\delta}_{s, i} \in (\A_{\Lambda}^m)^*,\quad\norm{\tilde{\delta}_{s, i}}_{(\A_{\Lambda}^m)^*}=1.\]
    \item The following equalities hold:    
    \[\overline{\span} \big\{ \tilde{\chi}_{t, j}\colon\ (t, j) \in \Sigma_m(\Lambda)\big\}=\A_{\Lambda}^m,\]
    \[\overline{\mbox{\textup{co}}}^{w^*}\!\big\{\! \pm \tilde{\delta}_{t, j}\colon\ (t, j) \in \Sigma_m(\Lambda)\}=B_{(\A_{\Lambda}^m)^*}.\]
    \item The mapping
    \[\Sigma_m(\Lambda)\ni (s,i) \longmapsto \tilde{\delta}_{s, i}\in\big(B_{(\A_{\Lambda}^m)^*}, w^*\big)\]
    is a homeomorphic embedding.
\end{enumerate}
\end{prop}

\begin{proof}
Item (i) is obvious and the first equality in item (ii) follows from Lemma \ref{lem:fin-approx}. 

For the second equality in item (ii) and for item (iii), we proceed as follows. We again for simplicity assume that $m=1$ and $\Lambda=\el$; the proofs of other cases are the same except that one may need to use Theorem \ref{A_lambda is C(omega^alpha)} instead of Theorem \ref{A_Lambda is C(2^omega)}. Let $T \colon \A_{\el} \rightarrow \C(2^{\omega})$ be the isomorphism from Theorem \ref{A_Lambda is C(2^omega)}. 
Then, since the adjoint operator $T^{*} \colon \M(2^{\omega}) \rightarrow \A_{\el}^*$ is a weak$^*$--weak$^*$ homeomorphism and $(T^*)^{-1}(\tilde{\delta}_t)=\delta_{R(t)}$ holds for every $t\in\el$, we have
\[\overline{\mbox{co}}^{w^*}\!\big\{\!\pm\!\tilde{\delta}_{t}\colon\ t \in \el\big\}=T^*\Big[\overline{\mbox{co}}^{w^*}\!\Big\{\!\pm\!(T^*)^{-1}(\tilde{\delta}_{t})\colon\ t \in \el\Big\}\Big]=\]
\[=T^*\Big[\overline{\mbox{co}}^{w^*}\!\big\{\!\pm\!\delta_{R(t)}\colon\ t \in \el\big\}\Big]=T^*\Big[\overline{\mbox{co}}^{w^*}\!\big\{\!\pm\!\delta_{x}\colon\ x \in 2^{\omega}\big\}\Big]=\]
\[=T^*\big[B_{\M(2^\omega)^*}\big]=B_{\A_{\el}^*},\]
as desired. Moreover, as $\tilde{\delta}_t=T^*(\delta_{R(t)})$ for every $t\in\el$, we get that the assignment $\el\ni s\mapsto\tilde{\delta}_s\in\big(B_{\A_{\el}^*}, w^*\big)$ is a composition of three injective continuous functions with the continuous inverses.
\end{proof}

\begin{remark}
Notice that, for each $\Lambda\in\TT$ and $m\in\en$, the system 
\[\big(\tilde{\chi}_{s, i}, \widehat{\delta}_{s, i}\big)_{(s, i) \in\Sigma_m(\Lambda)},\]
where for each $(s,i)\in\Sigma_m(\Lambda)$ we set
\[\widehat{\delta}_{s, i}=\begin{cases}
\tilde{\delta}_{s, i}&,\text{ if }s=\emptyseq,\\
\tilde{\delta}_{s, i}-\tilde{\delta}_{pred(s), i}&,\text{ otherwise,}
\end{cases}\]
forms a \emph{Markushevich basis} of the space $\A_{\Lambda}^m$ (cf. \cite{hajek2008biorthogonal}).
\end{remark}

\section{Projectional trees in Banach spaces\label{sec:ctcf}}
This section contains the proofs of our characterizations of the existence of complemented copies of the Banach spaces $\C(L)$ for metric compact spaces $L$ in arbitrary Banach spaces $E$. These characterizations are in a similar spirit as the characterization of the complementability of $c_0$ in Banach spaces obtained by Schlumprecht \cite{schlumprecht_phd_thesis} (cf. also \cite[Theorem 1.1.2]{cembranos2006banach}), which states that a Banach space $E$ contains a complemented copy of $c_0$ if and only if $E$ contains a basic sequence $(e_n)_{n \in \omega}$ equivalent to the standard basis of $c_0$ and admits a weak$^*$-null sequence $(\e_n)_{n \in \omega}$ in $E^*$ such that $\inf_{n \in \omega} \abs{\langle \e_n, e_n \rangle}>0$. 

In a similar fashion, a standard observation asserts that if a compact space $K$ contains a nontrivial convergent sequence, then the Banach space $\C(K)$ contains complemented isometric copies of the spaces $c_0$ and $\C([0,\omega])$ (cf. \cite[Proposition 2.4.6]{dalesbookdual} and \cite[Chapter 19]{KKLPS}). To see this, let $(x_n)_{n\in\omega}$ be a sequence in a compact space $K$ convergent to some point $x\in K$ such that $x_n\neq x_m\neq x$ for every $n\neq m\in\omega$. Using the Urysohn lemma, one can find a sequence $(U_n)_{n\in\omega}$ of pairwise disjoint open subsets of $K$ and a sequence $(f_n)_{n\in\omega}$ in $\C(K,[0,1])$ such that $x_n\in\supp(f_n)\subseteq U_n$ and $f_n(x_n)=1$ for every $n\in\omega$. Let $f=\chi_K$. Then, the subspaces 
\[E=\overline{\span\{f_n\colon n\in\omega\}}\quad\text{and}\quad F=\overline{\span\big(\{f_n\colon n\in\omega\}\cup\{f\}\big)}\]
of $\C(K)$ are isometrically isomorphic to the spaces $c_0$ and $\C([0,\omega])$, respectively. Moreover, the mappings $P\colon\C(K)\to E$ and $\tilde{P}\colon\C(K)\to F$, given for every $g\in\C(K)$ by
\[P(g)=\sum_{n\in\omega}\langle\delta_{x_n}-\delta_x,g\rangle f_n\]
and
\[\tilde{P}(g)=\langle\delta_x,g\rangle f+\sum_{n\in\omega}\langle\delta_{x_n}-\delta_x,g\rangle f_n,\]
are bounded projections onto $E$ and $F$, respectively. 

\medskip

Following the above ideas, we first introduce some auxiliary notions and prove basic results related to them, whereupon we proceed with the proofs of the promised characterizations. We start with the following notion.

\begin{definition}
If $(X,d)$ is a metric space and $E$ is a Banach space, then a mapping $\rho \colon X \rightarrow E^*$ is \emph{uniformly weak$^*$ continuous} if, for each vector $e \in E$ and  sequences $(x_n)_{n \in \omega}, (y_n)_{n \in \omega}$ in $X$, the following implication holds: 
\[\lim_{n \rightarrow \infty} d(x_n, y_n)=0\ \Longrightarrow\ \lim_{n \rightarrow \infty}\big\langle \rho(x_n)-\rho(y_n), e \big\rangle=0.\]\eodef
\end{definition}

\begin{remark}\label{rem:uni_weakstar_cont}
    Let $(X,d)$ be a metric space and $E$ a Banach space. If a mapping $\rho \colon X \rightarrow E^*$ is uniformly weak$^*$ continuous, then it is sequentially $d$--$w^*$ continuous, and hence $d$--$w^*$ continuous. On the other hand, if $X$ is a compact metric space, then any continuous mapping $\rho \colon X \rightarrow (E^*, w^*)$ is uniformly weak$^*$ continuous.
\end{remark} 

\begin{lemma}
\label{lem: the properties of rho}
Fix $\Lambda \in \TT$ and $m \in \en$. Assume that $E$ is a Banach space and let $\rho\colon\Sigma_m(\Lambda) \rightarrow E^*$ be a uniformly weak$^*$ continuous mapping. For each $(s, i) \in \Sigma_m(\Lambda)$ we denote $\rho_{s, i}=\rho((s, i))$. Then, the following conditions are satisfied:
\begin{enumerate}[(L.1)]
    \item\label{tree:unif} for each $1\le i\le m$, $s \in NL(\pi_i^m(\Lambda))$, and $e \in E$, it holds
    \[\lim_{n \rightarrow \infty} \sup\bigg\{\abs{\rho_{t, i}(e)-\rho_{s, i}(e)}\colon\ t\in\pi_i^m(\Lambda),\ s\concat n \preceq t\bigg\}=0,\]
    \item\label{tree:len} for each $1\le i\le m$ and $e\in E$, it holds
    \[\lim_{n\to\infty}\sup\bigg\{\abs{\rho_{t, i}(e)-\rho_{s, i}(e)}\colon\ s,t\in\pi_i^m(\Lambda),\ l(s)\ge n,\ s\preceq t\bigg\}=0.\]
\end{enumerate}
\end{lemma}

\begin{proof}
Assuming that property \ref{tree:unif} does not hold, there exist $1\le i\le m$, $s \in NL(\pi_i^m(\Lambda))$, $e \in E$, $\ep>0$, a strictly increasing sequence $(n_k)_{k\in\omega}$ of natural numbers, and a sequence $(t_{k})_{k\in\omega}$ of points in $\pi_i^m(\Lambda)$, such that for each $k \in \omega$ we have $s\concat n_k\preceq t_{k}$ and
\[\abs{\rho_{{t_{k}}, i}(e)-\rho_{s, i}(e)} \geq \ep.\]
By Lemma \ref{lemma:tree_topology}, the sequence $(t_{k})_{k\in\omega}$ converges to $s$ in the topology of $\pi_i^m(\Lambda_i)$, hence, by the uniform weak$^*$ continuity of $\rho$, the sequence $(\rho_{{t_{k}}, i})_{k\in\omega}$ converges weak$^*$ to $\rho_{s, i}$, which contradicts the above inequality.

For the proof of property \ref{tree:len}, assume that there exist $1\le i\le m$, $e \in E$, $\ep>0$ and sequences $(s_n)_{n\in\omega}, (t_n)_{n\in\omega}$ in $\pi_i^m(\Lambda)$ such that for each $n \in \omega$ we have $l(s_n)\ge n$, $s_n \preceq t_n$, and 
\[\abs{\rho_{t_n, i}(e)-\rho_{s_n, i}(e)} \geq \ep.\]
Then, $\lim_{n\to\infty} d(s_n, t_n)=0$, where $d$ is any metric compatible with the topology of $\Lambda$ (cf. Lemma \ref{lemma:tree_topology}), which contradicts the uniform weak$^*$ continuity of $\rho$.
\end{proof}

The following proposition constitutes the core of the proofs of the main results in this section.

\begin{prop}
\label{prop: complementability of tree spaces}
For a tree $\Lambda \in \TT$, $m \in \en$, and a Banach space $E$, the following assertion are equivalent:
\begin{enumerate}[(i)]
    \item $\A_{\Lambda}^m$ is isomorphic (resp. isometric) to a complemented subspace of $E$,
    \item there exist an isomorphic (resp. isometric) embedding $T\colon \A_{\Lambda}^m \rightarrow E$ and a uniformly weak$^*$ continuous mapping $\rho\colon\Sigma_m(\Lambda) \rightarrow E^*$ onto a bounded subset of $E^*$ such that $T^* \circ \rho =\tilde{\delta}$, where $\tilde{\delta} \colon\Sigma_m(\Lambda) \rightarrow (\A_{\Lambda}^m)^*$ is the canonical embedding mapping each $(s, i)$ onto $\tilde{\delta}_{(s, i)}$.
\end{enumerate}
\vspace{2pt}
In implication (ii)$\Rightarrow$(i), we moreover have:
\begin{enumerate}[(a)]
    \item if $T$ is an isometry and $\rho\big[\Sigma_m(\Lambda)\big] \subseteq B_{E^*}$, then the projection in (i) has norm $1$,
    \item if $E$ is a Banach lattice, $T$ is positive, and $\rho\big[\Sigma_m(\Lambda)\big] \subseteq E^*_+$, then the projection in (i) is positive.
\end{enumerate}
\end{prop}

\begin{proof}
To simplify the argument, we again assume that $m=1$; the more general case follows easily.

First, we show implication (i)$\Rightarrow$(ii). Let $F \subseteq E$ be a closed linear subspace, $T \colon \A_{\Lambda} \rightarrow F$ be a surjective isomorphism (resp. isometry), and $P \colon E \rightarrow F$ be a projection. Set
\[\rho=P^* \circ (T^*)^{-1} \circ\tilde{\delta}.\]
Then, $\rho$ is continuous (by Proposition \ref{prop:density}), maps $\Lambda$ onto a bounded subset of $E^*$, and for each $s \in \Lambda$ and $e \in E$ it holds 
\[\big\langle T^*(\rho(s)), e \big\rangle=\big\langle (T^*)^{-1}(\tilde{\delta}(s)),  PT(e)\big\rangle = \big\langle (T^*)^{-1}(\tilde{\delta}(s)), T(e)\big\rangle=\langle \tilde{\delta}(s), e\rangle,\]
as required. This finishes the proof in the case when $\Lambda \in \TT_{\alpha}$ for some countable successor ordinal $\alpha$, since then $\Lambda$ is compact, and therefore $\rho$ is uniformly weak$^*$ continuous due to its $d$--$w^*$ continuity, where $d$ is any compatible metric on $\Lambda$ (cf. Remark \ref{rem:uni_weakstar_cont}). 

If $\Lambda=\el$, we finish the proof as follows. Let $\delta$ stand for the standard Dirac mapping ($2^\omega\ni x \mapsto \delta_x\in\M(2^\omega)$), $R \colon \el \rightarrow 2^{\omega}$ be the homeomorphic embedding from Section \ref{sec:prelim}, $S \colon \A_{\el} \rightarrow \C(2^{\omega})$ be the isometry provided by Theorem \ref{A_Lambda is C(2^omega)}, and $S^* \colon \M(2^{\omega}) \rightarrow \A_{\el}^*$ be its adjoint operator. We now consider the function $\tilde{\rho}\colon 2^\omega\to(E^*,w^*)$ defined as 
\[\tilde{\rho}=P^* \circ (T^*)^{-1} \circ S^* \circ \delta.\] 
Note that $\tilde{\rho}$ is continuous and therefore uniformly weak$^*$ continuous. Consequently, it is now enough to show that $\rho=\tilde{\rho}\rstr{R(\el)}$. This, however, follows immediately by the equality $(S^* \circ \delta) (R(s))=\tilde{\delta}_s$, which holds for every $s \in \el$ by Theorem \ref{A_Lambda is C(2^omega)} and Proposition \ref{prop:density}.

\medskip

In order to prove implication (ii)$\Rightarrow$(i), we shall construct a projection $P\colon E\to E$ onto the image $T[\A_\Lambda]$. For this, we define the auxiliary mapping $S \colon E \rightarrow \A_{\Lambda}$ by setting for each $e \in E$ and $s\in\Lambda$: 
\[S(e)_s=\begin{cases}
    \langle \rho({\emptyseq}), e\rangle &,\text{ if } s=\emptyseq,\\
   \big\langle \rho(s)-\rho({pred(s)}), e\big\rangle &,\text{ otherwise}.
\end{cases}\]
It is clear that $S$ is a well-defined linear mapping from $E$ into $\er^\Lambda$. We will show that $S$ is actually a bounded operator into $\A_{\Lambda}$. 

We first need to check that for each $e \in E$ we indeed have $S(e) \in \A_{\Lambda}$. Fix $e\in E$. We find $C \in \er$ such that $\rho[\Lambda]\subseteq C \cdot B_{E^*}$, so we get 
\[\tag{$*$}\norm{S(e)}^{\Lambda}=\sup_{t \in \Lambda} \abs{\sum_{s \preceq t}S(e)_s}=\sup_{t\in\Lambda}\bigg|\big\langle\rho(\emptyseq),e\big\rangle+\Big(\big\langle\rho\big(\langle t(0)\rangle\big),e\big\rangle-\big\langle\rho(\emptyseq),e\big\rangle\Big)+\]
\[+\Big(\big\langle\rho\big(\langle t(0),t(1)\rangle\big),e\big\rangle-\big\langle\rho\big(\langle t(0)\rangle\big),e\big\rangle\Big)+\ldots+\Big(\big\langle\rho(t),e\big\rangle-\big\langle\rho\big(pred(t)\big),e\big\rangle\Big)\bigg|=\]
\[=\sup_{t \in \Lambda} \abs{\langle \rho(t), e \rangle } \leq C \cdot \norm{e}, \]
which proves that $S(e)$ satisfies condition (A.1).

Concerning condition (A.2), we pick $s \in NL(\Lambda)$ and, by a computation similar to ($*$) and using property \ref{tree:unif} from Lemma \ref{lem: the properties of rho}, we get
\[\lim_{n\to\infty}\norm{S(e)\rstr{\Lambda_{s\concat n}}}^{\Lambda}=\lim_{n\to\infty}\sup_{\substack{t\in\Lambda\\s\concat n\preceq t}}\abs{\sum_{s\concat n \preceq u \preceq t} S(e)_u}=\]
\[=\lim_{n\to\infty}\sup_{\substack{t\in\Lambda\\s\concat n\preceq t}}\abs{\langle \rho(t)-\rho(s), e \rangle}=0,\]
as required.

Similarly, using property \ref{tree:len} from Lemma \ref{lem: the properties of rho}, we have
\[\lim_{n\to\infty}\norm{S(e)\rstr\{s\in\Lambda\colon l(s) \geq n\}}^{\Lambda}=\]
\[=\lim_{n\to\infty}\sup \Bigg\{\abs{\sum_{t_1 \preceq s \preceq t_2} S(e)_s}\colon\ t_1, t_2 \in \Lambda,\ l(t_1) = n,\ t_1\preceq t_2\Bigg\}=\]
\[=\lim_{n\to\infty}\sup \bigg\{\abs{\langle \rho(t_2)-\rho\big(pred({t_1})\big), e \rangle}\colon\ t_1, t_2 \in \Lambda,\ l(t_1) = n,\ t_1\preceq t_2\bigg\}=0,\]
and thus also condition (A.3) is satisfied. It follows that indeed $S(e)\in \A_\Lambda$. Moreover, since $\norm{S(e)}^\Lambda\le C\cdot\norm{e}$ and $C$ does not depend on $e$, we get that $S$ is bounded with $\norm{S}\le C$.
%

\medskip

Now, let $S^* \colon \A_{\Lambda}^* \rightarrow E^*$ be the operator adjoint to $S$. Notice that for each $s \in \Lambda$ the functional $S^*(\tilde{\delta}_{s})\in E^*$ coincides with $\rho(s)$: for all $s \in \Lambda$ and $e \in E$, we have
\[\big\langle S^*(\tilde{\delta}_s), e \big\rangle=\big\langle \tilde{\delta}_s, S(e) \big\rangle=\sum_{t \preceq s} S(e)_t= \langle \rho(s), e \rangle,\]
where the last equality follows again from a computation similar to ($*$). We will use this observation to show that $T \circ S\colon E \rightarrow E$ is the desired projection onto the range of $T$. Obviously, $(T \circ S)[E] \subseteq T[\A_{\Lambda}]$. Thus, it is enough to show that $T \circ S$ is the identity on the set $T[\A_\Lambda]$. 

We first check that $T^* \circ S^* \colon \A_{\Lambda}^* \rightarrow \A_{\Lambda}^*$ is the identity mapping on $\A_\Lambda^*$. To this end, note that for any element $s \in \Lambda$ by the above observation we have 
\[(T^* \circ S^*)(\tilde{\delta}_s)=T^*(\rho(s))=\tilde{\delta}_s.\]
Thus, by the linearity, $T^* \circ S^*$ equals to the identity on the set $\big\{\!\pm\!\tilde{\delta}_s\colon s \in \Lambda\big\}$, and hence, by Proposition \ref{prop:density}.(ii), on the whole dual space $\A_\Lambda^*$.

Now, for any elements $\olb{a} \in \A_{\Lambda}$ and $\e \in E^*$, we get
\[\big\langle \e, (T \circ S)\big(T(\olb{a})\big)\big\rangle=\big\langle\big(T^* \circ S^* \circ T^*\big)(\e), \olb{a} \big\rangle=\]
\[=\big\langle\big(\!\id_{\A_\Lambda^*} \circ T^*\big)(\e), \olb{a} \big\rangle=\langle \e, T(\olb{a}) \rangle.\]
Thus, $(T \circ S)\big(T(\olb{a})\big)=T(\olb{a})$ for any $\olb{a}\in \A_\Lambda$, and hence $T\circ S=\id_{T[\A_\Lambda]}$, which finishes the proof of implication (ii)$\Rightarrow$(i).

\medskip

If $T$ is an isometry and $\rho[\Lambda]\subseteq B_{E^*}$, i.e. $C\le 1$, then it is apparent that the projection $T \circ S$ has norm $1$.

If $E$ is moreover a Banach lattice, $T$ is positive, and $\rho[\Lambda] \subseteq E^*_+$, then the mapping $S$ is positive as well, as for any $t \in \Lambda$ and $e \in E_+$ we have
\[\sum_{s \preceq t} S(e)_s=\langle \rho(t), e \rangle\ge0,\]
and so the projection $T \circ S$ is positive, too. This finishes the proof.
\end{proof}

\begin{remark}
It can be checked that in item (ii) of Proposition \ref{prop: complementability of tree spaces}, if for $(s, i) \in\Sigma_m(\Lambda)$ one writes $e_{s, i}=T(\tilde{\chi}_{s, i})$ and $\e_{s, i}=\rho((s, i))$, then the projection $P$ onto the complemented isomorphic copy of $\A_{\Lambda}^m$ in $E$ can be written for each $e \in E$ directly in the following two equivalent ways: 
\begin{equation}
\nonumber
\begin{aligned}
P(e)&=\sum_{i=1}^m\Bigg(\left\langle e^*_{\emptyseq, i}, e \right\rangle e_{\emptyseq, i}+\sum_{s \in\Lambda\setminus\{\emptyseq\}}\left\langle e^*_{s, i}-e^*_{pred(s), i}, e\right\rangle e_{s, i}\Bigg)=\\&=\sum_{i=1}^m\Bigg(\left\langle e^*_{\emptyseq, i}, e \right\rangle e_{\emptyseq, i}+\sum_{s \in NL(\Lambda)} \sum_{n\in\omega} \left\langle e^*_{s\concat n, i}-e^*_{s, i}, e\right\rangle e_{s\concat n, i}\Bigg).
\end{aligned}
\end{equation}
\end{remark}

The following notion is crucial for the main results of this paper.

\begin{definition}\label{def:ctcf}
Let $E$ be a Banach space. Let $\alpha$ be a countable successor ordinal or $\infty$, and $m\in\en$. A \emph{projectional tree of rank $\alpha$ and order $m$} in $E$ is a triple 
$\T=(\Sigma, T, \rho)$ such that:
\begin{enumerate}[(T.1)]
    \item $\Sigma$ is the topological disjoint union of $m$ copies of a tree $\Lambda$ of rank $\alpha$,
    \item $T \colon \A_{\Lambda}^m \rightarrow E$ is an isomorphic embedding,
    \item $\rho\colon\Sigma \rightarrow E^*$ is a uniformly weak$^*$ continuous map onto a bounded subset of $E^*$ such that $T^* \circ \rho =\tilde{\delta}$, where $\tilde{\delta}\colon \Sigma \rightarrow (\A_{\Lambda}^m)^*$ is the canonical embedding which maps each $(s, i)\in\Sigma$ onto $\tilde{\delta}_{s, i}$.    
\end{enumerate}

\medskip
The tree $\T$ is \emph{isometric} if $T$ is an isometry. The infimum of $C>0$ satisfying the inclusion $\rho[\Sigma] \subseteq C \cdot B_{E^*}$ is called the \emph{norm} of the tree and denoted by $\norm{\T}$. 

If, moreover, $E$ is a Banach lattice, we say that the tree $\T$ is 
\begin{itemize}
    \item \emph{positive}, if $T$ is a positive operator and $\rho[\Sigma] \subseteq E^*_+$,
    \item \emph{PNPP}, if $T$ is PNPP and $\rho[\Sigma] \subseteq E^*_+$.
    \end{itemize}  
    \eodef
\end{definition}

Of course, by Theorems \ref{A_lambda is C(omega^alpha)} and \ref{A_Lambda is C(2^omega)}, the Banach spaces $\C([1,\omega^\alpha m])$ and $\C(2^\omega)$ are natural examples of spaces admitting PNPP projectional trees (and, in fact, motivated their definition and name). Let us collect some further important and useful remarks concerning projectional trees.

\begin{remark}
\label{rem:tree sufficient 2}
By Proposition \ref{prop:density}, the condition $T^* \circ \rho =\tilde{\delta}$ from Definition \ref{def:ctcf} can be equivalently stated as that for all $(s, i),(t,j) \in \Sigma$ we have
\[\big\langle \rho(s, i), T(\tilde{\chi}_{t, j}) \big\rangle=\big\langle \tilde{\delta}_{s, i}, \tilde{\chi}_{t, j}\big\rangle.\]
\end{remark}

\begin{remark}
\label{rem:tree sufficient 1}
For a tree $\Lambda\in\TT$, $m\in\en$, and a Banach space $E$, in order to check that an operator $T \colon \A_{\Lambda}^m \rightarrow E$ is an isomorphic embedding,  by Lemma \ref{lem:fin-approx} it is enough to show that there exist $C_1, C_2>0$ such that, for each element $\big((a_{s, i})_{s \in \Lambda}\big)_{1\le i\le m} \in \A_\Lambda^m$ and trunk $F\in\mathfrak{Tr}(\Lambda)$, we have
\[\frac{1}{C_1}\max_{i=1,\ldots,m}\norm{(a_{s, i})_{s \in \Lambda}\rstr F}^{\Lambda} \leq \norm{\sum_{i=1}^m\sum_{s \in F} a_{s, i} T(\tilde{\chi}_{s,i})} \leq C_2\max_{i=1,\ldots,m}\norm{(a_{s, i})_{s \in \Lambda}\rstr F}^{\Lambda}.\]
\end{remark}

\begin{remark}
    Notice that if a projectional tree in an AM-space (e.g. in a $\C(K)$-space) is PNPP, then it is isometric, see Lemma \ref{lem:ordered normed spaces and PNPP} and its proof.
\end{remark}

The following theorems, being versions of Theorems \ref{thm:mainA} and \ref{thm:mainB} from Introduction and also the main results of this section, connect the existence of projectional trees with the complementability of separable spaces $\C(K)$ in Banach spaces. 
\begin{thm}
\label{main characterization-zerodim}
Let $E$ be a Banach space. 
\begin{enumerate}[(1)]
    \item If $\alpha$ is a countable ordinal and $m\in\en$, then  the following assertions are equivalent:
    \begin{enumerate}[(i)]
        \item the space $\C([1, \omega^\alpha m])$ is (isometrically) isomorphic to a complemented subspace of $E$, 
        \item $E$ admits an (isometric) projectional tree of rank $\alpha+1$ and order $m$.
    \end{enumerate}
    \vspace{1mm}
    \item The following assertions are equivalent:
    \begin{enumerate}[(i)]
        \item the space $\C(2^\omega)$ is (isometrically) isomorphic to a complemented subspace of $E$,
        \item $E$ admits an (isometric) projectional tree of rank $\infty$ and order $1$.
    \end{enumerate}
\end{enumerate}

In both cases, we moreover have:
\begin{enumerate}[(a)]
    \item If the tree is isometric and has norm $1$, then the projection has norm $1$.
    \item If $E$ is a Banach lattice and the tree is positive, then the projection is positive.
\end{enumerate}
\end{thm}

\begin{proof}
    Items (1) and (2) immediately follow by combining Proposition \ref{prop: complementability of tree spaces} with Theorems \ref{A_lambda is C(omega^alpha)} and \ref{A_Lambda is C(2^omega)}. Item (a) follows from Proposition \ref{prop: complementability of tree spaces}.(a) and item (b) from Proposition \ref{prop: complementability of tree spaces}.(b).
\end{proof}

The isomorphic version of the above characterization holding for \textit{all} compact metric spaces $K$ reads as follows.

\begin{thm}
\label{main characterization}
Let $E$ be a Banach space and $L$ be a metric compact space. Then, the following assertions are equivalent:
\begin{enumerate}[(i)]
    \item $\C(L)$ is isomorphic to a complemented subspace of $E$,
    \item $E$ admits a projectional tree of rank $ht(L)$ and order $1$.
\end{enumerate}
\end{thm}
\begin{proof}
The result follows immediately by Theorem \ref{main characterization-zerodim} and the fact that $\C(L)$ is isomorphic either to $\C(2^\omega)$, if $L$ is uncountable (by Miljutin's theorem), or to $\C([1,\omega^\alpha])$, where $\alpha=ht(L)-1$, if $L$ is countable (by Bessaga--Pe\l czy\'nski's classification).
\end{proof}

\section{Applications to $\C(K)$-spaces\label{sec:apps_to_ck}}

In this section we will present two applications of the results obtained in Section \ref{sec:ctcf} to the special case of Banach spaces $E$ of the form $\C(K)$.

\subsection{A variant of the theorem of Holszty\'nski\label{sec:holsztynski}}

First, we prove the following variant of the Holszty\'nski theorem (see \cite{Holsztynski1966} and \cite{Pelczynski_separable}) for the case of isometric embeddings onto complemented subspaces. Note that this time we do not assume that the space $L$ is metric. 

\begin{thm}
\label{holsztynski}
Let $K$ and $L$ be compact spaces. Let $T\colon\C(L) \rightarrow\C(K)$ be an isometry onto a complemented subspace $E$ of $\C(K)$ and $P\colon\C(K) \rightarrow E$ be a norm-$1$ projection onto $E$. Then, there exist a closed subset $F$ of $K$, a continuous surjection $\rho\colon F \rightarrow L$, and continuous mappings $\sigma\colon F \rightarrow\{-1,1\}$ and $\phi\colon L \rightarrow S_{\M(F)}$ such that:
\begin{enumerate}[(i)]
    \item for each $x \in L$ it holds
    \[\supp(\phi(x))\subseteq\rho^{-1}(x),\]
    \item for each $f \in \C(L)$ and $y \in F$ it holds
    \[T(f)(y)=\sigma(y) f(\rho(y)),\]
    \item for each $g \in \C(K)$ and $y \in F$ it holds
    \[P(g)(y)=\sigma(y) \big\langle \phi(\rho(y)), g \big\rangle.\]
\end{enumerate}
Moreover,
\begin{enumerate}[(a)]
    \item if $T$ is positive, then $\sigma(y)=1$ for each $y \in F$,
    \item if both $T$ and $P$ are positive, then $\phi(x)\in \P(F)$ for each $x \in L$.
\end{enumerate}
\end{thm}

Before proving the result, we first need to establish the following key lemma, which enables us to identify certain elements of the bidual space of $\C(K)$ as bounded Borel functions on $K$. Note that, in general, the space $\C(K)^{**}$ is strictly larger than the space of bounded Borel functions on $K$, see e.g. \cite[Section 6.6]{dalesbookdual}. Also, recall that a function $h\colon X \rightarrow \er$ on a topological space $X$ is \emph{of the first Borel class}, if for each open subset $U$ of $\er$ the set $h^{-1}[U]$ is a countable union of differences of closed subsets of $X$ (see \cite{spurny-amh} or \cite[Definition 5.13]{lmns} for more information).

\begin{lemma}
\label{lem:bidual of C(K) and Borel functions on K}    
Let $h^{**} \in \C(K)^{**}$ be a function which is of the first Borel class on the space $(\M(K),w^*)$. Then, for each $\mu \in \M(K)$, it holds
\[\langle h^{**}, \mu \rangle=\int_{K} h d\mu,\]
where the bounded Borel function $h\colon K \rightarrow \er$ is defined by
\[h(y)=\langle h^{**}, \delta_y \rangle\]
for all $y\in K$.
\end{lemma}

\begin{proof}
We first recall that, since the function $h\colon  K \rightarrow \er$ is bounded and Borel, the $\er$-valued mapping
\[\tilde{h}\colon\M(K)\ni \mu \longmapsto\int_{K} h d\mu\]
is a bounded linear functional on $\M(K)$, i.e. an element of $\C(K)^{**}$. We will show that $\tilde{h}$ coincides with $h^{**}$.  

Firstly, we observe that the function $h$ is of the first Borel class on $K$, too, since the mapping $\delta\colon  K\ni y\mapsto \delta_y\in \M(K)$ is a homeomorphic embedding of $K$ onto a compact subspace of $\M(K)$ endowed with the weak$^*$ topology. This implies that $\tilde{h}$ is also of the first Borel class on $\M(K)$, see \cite[Lemma 3.2]{lusp}, and hence, so is the difference $h^{**}-\tilde{h}$. Consequently, 
we have
\[\sup_{\mu \in B_{\M(K)}} \abs{\langle h^{**}, \mu \rangle-\tilde{h}(\mu)}=\sup_{\mu\in\ext(B_{\M(K)})}\abs{\langle h^{**}, \mu \rangle-\tilde{h}(\mu)},\]
which follows from \cite[Corollary 1.5.(b)]{dosp}, since each function of the first Borel class has the so-called point of continuity property, by \cite[Theorem 2.3]{koumou}. We further recall the standard fact that $\ext(B_{\M(K)})=\delta[K] \cup -\delta[K]$, see e.g. \cite[Section 4.3]{lmns}. 
Altogether, we get
\begin{equation}
\nonumber
\begin{aligned}
\norm{h^{**}-\tilde{h}}&=\sup_{\mu \in B_{\M(K)}} \abs{\langle h^{**}, \mu \rangle-\tilde{h}(\mu)}=\sup_{\mu\in\ext(B_{\M(K)})}\abs{\langle h^{**}, \mu \rangle-\tilde{h}(\mu)}=\\
&=\sup_{\mu \in \delta[K] \cup -\delta[K]} \abs{\langle h^{**}, \mu \rangle-\tilde{h}(\mu)}=\sup_{y \in K} \abs{\langle h^{**}, \delta_y \rangle-\tilde{h}(\delta_y)}=\\&=\sup_{y \in K} \abs{h(y)-h(y)}=0,    
\end{aligned}
\end{equation}
as required.
\end{proof}

\begin{proof}[Proof of Theorem \ref{holsztynski}.]
Let $T^*\colon E^{*} \rightarrow \M(L)$, $T^{**}\colon\C(L)^{**} \rightarrow E^{**}$, and $P^{*}\colon E^{*} \rightarrow \M(K)$ be the respective adjoint mappings. Note that $T^*$ and $T^{**}$ are isometries. Moreover, observe that for each $x\in L$ the function $\chi_{\{x\}}$ may be identified with the functional $\widehat{\chi}_{\{x\}}\in \M(L)^*$, given for every $\mu\in\M(L)$ by the formulas 
\[\widehat{\chi}_{\{x\}}(\mu)=\int_L\chi_{\{x\}}d\mu=\mu(\{x\}).\]
Finally, we recall the well-known fact that the Hahn--Banach theorem implies that $E^{**}$ can be isometrically identified with a subspace of $\C(K)^{**}$.

Put
\[F=\Big\{y \in K\colon\text{ there exists } x \in L \text{ such that } \abs{\left\langle T^{**}(\chi_{\{x\}}), \delta_y\right\rangle}=1\Big\};\]
observe that for each $y\in K$ we have $y\in F$ if and only if there are $x\in L$ and $\alpha\in\{-1,1\}$ such that $T^*(\delta_y)=\alpha\delta_x$. 

We define the required mappings $\rho\colon F \rightarrow L$ and $\phi\colon L \rightarrow S_{\M(K)}$ by setting for every $y\in F$ and $z\in L$
\[ \rho(y)=\text{the unique point }x \in L\text{ such that }\abs{\left\langle T^{**}(\chi_{\{x\}}), \delta_y \right\rangle}=1\]
and
\[\phi(z)=P^{*}\big((T^{*})^{-1}(\delta_z)\big).\]

\medskip

It is clear that $\phi$ is a well-defined continuous mapping into $B_{\M(K)}$. We now check further required properties of the above-defined objects.

\medskip

\noindent $\bullet$ \textit{$\rho$ is well-defined.} 

We need to show that for each $y \in F$ there is exactly one $x \in L$ such that $\abs{\left\langle T^{**}(\chi_{\{x\}}), \delta_y \right\rangle}=1$. Thus, assume that there exist $y\in F$ and two distinct points $x_1, x_2 \in L$ which satisfy $\abs{\left\langle T^{**}(\chi_{\{x_i\}}), \delta_y\right\rangle}=1$ for each  $i=1, 2$. We find $\alpha_1, \alpha_2 \in\{-1,1\}$ such that 
\[\alpha_i\cdot\left\langle T^{**}(\chi_{\{x_i\}}), \delta_y\right\rangle=1\]
for each $i=1, 2$. As $\delta_y\in \M(K)$ and $\norm{\delta_y}=1$, we get 
\[\norm{T^{**}\big(\alpha_1 \chi_{\{x_1\}}+\alpha_2 \chi_{\{x_2\}}\big)} \geq\abs{\left\langle T^{**}\big(\alpha_1 \chi_{\{x_1\}}+\alpha_2 \chi_{\{x_2\}}\big),\ \delta_y\right\rangle}=\]
\[=\abs{\alpha_1\cdot\left\langle T^{**}(\chi_{\{x_1\}}),\delta_y\right\rangle+\alpha_2\cdot\left\langle T^{**}(\chi_{\{x_2\}}), \delta_y\right\rangle}=\abs{1+1}=2.\]
On the other hand, it is clear that $\norm{\alpha_1 \chi_{\{x_1\}}+\alpha_2 \chi_{\{x_2\}}}=1$, which contradicts the fact that $T$ is an isometry. Hence, $\rho$ is indeed a well-defined mapping.

\medskip

\noindent $\bullet$ \textit{For every $x\in L$ we have $\norm{\phi(x)}=1$, i.e. $\phi(x)\in S_{\M(K)}$.} 

Let $x\in L$. For each $f \in \C(L)$ we have $Tf \in E$ and hence $PTf=Tf$, so, consequently, 
\[\tag{$*$}\left\langle P^*\big((T^*)^{-1}(\delta_x)\big),\ Tf\right\rangle=\left\langle (T^*)^{-1}(\delta_x),\ PTf\right\rangle=\left\langle (T^*)^{-1}(\delta_x),\ Tf\right\rangle=f(x).\]
Let $(f_{\alpha})_{\alpha\in A}$ be a net of functions in $\C(L)$ converging weak$^*$ to $\chi_{\{x\}}$ in $\C(L)^{**}$ and such that for each $\alpha\in A$ we have $f_{\alpha}(x)=1$. Then, since $T^{**}$ is weak$^*$--weak$^*$ continuous and $T^{**}(f_\alpha)=T(f_\alpha)$ for each $\alpha\in A$, the net $(T(f_{\alpha}))_{\alpha\in A}$ converges weak$^*$ to $T^{**}(\chi_{\{x\}})$ in $E^{**}$. Combining these facts with ($*$), we get
\[1=\lim_{\alpha\in A} f_{\alpha}(x)=\lim_{\alpha\in A} \left\langle P^*\big((T^*)^{-1}(\delta_x)\big),\ T(f_{\alpha}) \right\rangle=\]
\[=\left\langle P^{*}\big((T^*)^{-1}(\delta_x)\big),\ T^{**}(\chi_{\{x\}})\right\rangle=\left\langle T^{**}(\chi_{\{x\}}), \phi(x)\right\rangle=\left\langle\chi_{\{x\}},T^*(\phi(x))\right\rangle=\]
\[=T^*(\phi(x))(\{x\}),\]
hence $\norm{T^*(\phi(x))}\ge1$. However, as $T^*$ is an isometry and $\norm{\phi(x)}\le1$, we get that $\norm{T^*(\phi(x))}=1$, and hence $\norm{\phi(x)}=1$.

\medskip

\noindent $\bullet$ \textit{For each $x \in L$, $T^{**}(\chi_{\{x\}})$ can be identified with a $1$-bounded Borel $\er$-valued function on $K$, i.e., there is a Borel function $h_x\colon K\to\er$ such that $T^{**}(\chi_{\{x\}})(\mu)=\int_Kh_xd\mu$ for every $\mu\in\M(K)$ and $\abs{h_x(y)}\le1$ for every $y\in K$.}

Let $x\in L$. We first note that the characteristic function $\chi_{\{x\}}\colon  L \rightarrow \er$ is clearly of the first Borel class. Consequently, the function $\widehat{\chi}_{\{x\}}\colon (\M(L),w^*) \rightarrow \er$, given for every $\mu\in\M(L)$ by the formula 
\[\widehat{\chi}_{\{x\}}(\mu)=\int_L\chi_{\{x\}}d\mu,\]
is of the first Borel class as well, see again \cite[Lemma 3.2]{lusp}. Thus, also the function $T^{**}(\chi_{\{x\}})=\widehat{\chi}_{\{x\}} \circ T^*\colon \M(K) \rightarrow \er$ is of the first Borel class on $(\M(K),w^*)$, as a composition of a first Borel class function with a continuous mapping (recall that $T^{*}$ is weak$^*$--weak$^{*}$ continuous). Hence, the conclusion follows from Lemma \ref{lem:bidual of C(K) and Borel functions on K}.

\medskip

\noindent $\bullet$ \textit{For each $x\in L$ we have $\abs{\phi(x)}\big(\rho^{-1}(x)\big)=1$ and $\phi(x)\in S_{\M(F)}$, and hence $\rho$ is surjective.} 

First, since by the previous point for each $x\in L$ we may consider the functional $T^{**}(\chi_{\{x\}})$ as a Borel function $h_x\colon K\to\er$, the set
\[\rho^{-1}(x)=\Big\{y\in K\colon\ \abs{\left\langle T^{**}(\chi_{\{x\}}), \delta_y \right\rangle}=1\Big\}=\big\{y\in K\colon\ \abs{h_x(y)}=1\big\}\]
is a Borel subset of $K$, and so is $\phi(x)$-measurable.

Now, for the sake of contradiction, assume that there exists $x \in L$ such that 
\[\abs{\phi(x)}\big(\rho^{-1}(x)\big)=\abs{\phi(x)}\Big(\Big\{ y \in K\colon\ \abs{\left\langle T^{**}(\chi_{\{x\}}), \delta_y \right\rangle}=1\Big\}\Big)<1.\]
By the $\sigma$-additivity of the measure $\abs{\phi(x)}$ and the equality
\[\Big\{ y \in K\colon \abs{\left\langle T^{**}(\chi_{\{x\}}), \delta_y \right\rangle}=1\Big\}=\bigcap_{n \in \en} \Big\{ y \in K\colon \abs{\left\langle T^{**}(\chi_{\{x\}}), \delta_y \right\rangle}>1-1/n\Big\},\]
there exists $\ep>0$ such that, for the set 
\[M_{\ep}=\Big\{ y \in K\colon \abs{\left\langle T^{**}(\chi_{\{x\}}), \delta_y \right\rangle}>1-\ep\Big\},\]
it holds $\abs{\phi(x)}(M_{\ep})<1$, and so
\[\abs{\phi(x)}(K\setminus M_\ep)>0,\]
as $\norm{\phi(x)}=1$.

Next, having again in mind that we treat $T^{**}(\chi_{\{x\}})$ as a Borel $\er$-valued function on $K$ such that
\[\sup_{y\in K}\abs{T^{**}(\chi_{\{x\}})(y)}\le 1,\] 
as previously, we get
\[1=\left\langle T^{**}(\chi_{\{x\}}), \phi(x)\right\rangle\leq\int_K \abs{T^{**}(\chi_{\{x\}})(y)} d\abs{\phi(x)}(y)=\]
\[=\int_{M_{\ep}} \abs{T^{**}(\chi_{\{x\}})(y)} d\abs{\phi(x)}(y)+\int_{K \setminus M_{\ep}} \abs{T^{**}(\chi_{\{x\}})(y)} d\abs{\phi(x)}(y)\leq\]
\[\leq\abs{\phi(x)}(M_{\ep})+(1-\ep)\abs{\phi(x)}(K \setminus M_{\ep})=\norm{\phi(x)}-\eps\abs{\phi(x)}(K\setminus M_\ep)=\]
\[=1-\eps\abs{\phi(x)}(K\setminus M_\ep)<1,\]
which is the desired contradiction.

\medskip

\noindent $\bullet$ \textit{For every $y \in F$ there exists $\sigma(y) \in\{-1,1\}$ such that for each $f\in\C(L)$ we have $T(f)(y)=\sigma(y) f(\rho(y))$ and so $\abs{f(\rho(y))}=\abs{T(f)(y)}$.} 

Let $y\in F$. We put
\[\sigma(y)=\left\langle T^{**}\big(\chi_{\{\rho(y)\}}\big), \delta_y\right\rangle= \left\langle \chi_{\{\rho(y)\}}, T^{*}(\delta_y) \right\rangle.\]
Then, by the definition of $\rho(y)$, we have $\abs{\sigma(y)}=1$, so that $\sigma(x)\in\{-1,1\}$. As $T^*$ is an isometry, we have $\norm{T^*(\delta_y)}=1$. Also, since $\sigma(y)=\left\langle\chi_{\{\rho(y)\}}, T^*(\delta_y)\right\rangle$, we get
\[\abs{T^*(\delta_y)\big(\{\rho(y)\}\big)}=\abs{\sigma(y)}=1.\] 
Consequently, $T^*(\delta_y)=\sigma(y)\delta_{\rho(y)}$.
Thus, for each $f \in \C(L)$ we have
\[T(f)(y)=\langle\delta_y, Tf\rangle=\left\langle T^{*}(\delta_y), f\right\rangle=\left\langle \sigma(y) \delta_{\rho(y)}, f \right\rangle=\sigma(y)f(\rho(y)),\]
as desired.

\medskip

\noindent $\bullet$ \textit{$\sigma$ is continuous, and if $T$ is positive, than $\sigma(y)=1$ for each $y\in F$.} 

Let $f\in\C(L)$ be the function which is constant $1$ on $L$, i.e. $f=\chi_L$. By the above equality, for every $y \in F$ we get
\[\sigma(y)=\sigma(y)\cdot 1=\sigma(y) f(\rho(y))=T(f)(y).\]
Hence, $\sigma$ is continuous due to the continuity of the function $Tf$. Also, if $T$ is positive, then for all $y\in F$ we have $T(f)(y)\ge0$, hence $\sigma(y)=1$.

\medskip

\noindent $\bullet$ \textit{$\rho$ is continuous.} 

Assume for the sake of contradiction that there exists a net $(y_{\lambda})_{\lambda\in\Lambda}$ in $F$ converging to some point $y \in F$, but such that the points $x_{\lambda}=\rho(y_{\lambda})$, $\lambda\in\Lambda$, do not converge to $x=\rho(y)$ in $L$. Then, there exists an open neighborhood $U$ of $x$ in $L$ such that for each $\lambda_0\in\Lambda$ there exists $\lambda \geq \lambda_0$ for which $x_{\lambda} \not\in U$. Find a function $f \in \C(L, [0, 1])$ which is supported by $U$ and such that $f(x)=1$. Then, by the above, we have
\[\abs{T(f)(y)}=\abs{\sigma(y)f(\rho(y))}=\abs{\sigma(y) f(x)}=1.\]
Thus, there exists $\lambda_0\in\Lambda$ such that for each $\lambda \geq \lambda_0$ we have $\abs{T(f)(y_{\lambda})}>0$. By taking $\lambda \geq \lambda_0$ such that $x_{\lambda} \notin U$, we get
\[0=\abs{f(x_{\lambda})}=\abs{\sigma(y_\lambda)f(\rho(y_\lambda))}=\abs{T(f)(y_{\lambda})}>0,\]
a contradiction.

\medskip

\noindent $\bullet$ \textit{$F$ is closed.} 

Let $(y_{\lambda})_{\lambda\in\Lambda}$ be a net in $F$ converging to some point $y \in K$. For each $\lambda\in\Lambda$ we denote $x_{\lambda}=\rho(y_{\lambda})$. By the compactness of $L$, passing to a subnet if necessary, we may assume that the net $(x_{\lambda})_{\lambda\in\Lambda}$ converges to some point $x \in L$. For each $f \in \C(L)$ and $\lambda\in\Lambda$ we have
\[\abs{f(x_{\lambda})}=\abs{\sigma(y_\lambda)f(\rho(y_\lambda))}=\abs{T(f)(y_{\lambda})},\]
and thus $\abs{f(x)}=\abs{Tf(y)}$, by the continuity of $f$ and $Tf$.

We again find a net $(f_{\alpha})_{\alpha\in A}$ in $\C(L)$ of functions converging weak$^*$ to $\chi_{\{x\}}$ in $\C(L)^{**}$ and such that for each $\alpha\in A$ we have $f_\alpha(x)=1$. By the weak$^*$--weak$^*$ continuity of $T^{**}$ and the fact that $T^{**}(f_\alpha)=T(f_\alpha)$ for each $\alpha\in A$, we get
\[\abs{\left\langle T^{**}(\chi_{\{x\}}), \delta_y \right\rangle}=\lim_{\alpha\in A} \abs{\left\langle T^{**}(f_{\alpha}), \delta_y\right\rangle}=\lim_{\alpha\in A} \abs{T(f_{\alpha})(y)}=\lim_{\alpha\in A} \abs{f_{\alpha}(x)}=1.\]
By the definition of $F$, it follows that $y \in F$. Consequently, $F$ is closed.

\medskip

\noindent $\bullet$ \textit{For each $g\in\C(K)$ and $y\in F$ we have $P(g)(y)=\sigma(y)\left\langle \phi(\rho(y)), g\right\rangle$.} 

Let $g\in\C(K)$ and $y\in F$. We have
\[\left\langle \phi(\rho(y)), g\right\rangle=\left\langle P^*\big((T^*)^{-1}(\delta_{\rho(y)})\big),\ g \right\rangle=\left\langle\delta_{\rho(y)}, (T^{-1}P)(g)\right\rangle=\]
\[=\big((T^{-1}P)(g)\big)(\rho(y))=\sigma(y)^{-1} T\big((T^{-1}P)(g)\big)(y)=\sigma(y)^{-1}P(g)(y),\]
so the required equality holds.

\medskip

\noindent $\bullet$ \textit{If $T$ and $P$ are positive, then $\phi(x)\in\P(K)$ for all $x\in L$.}

Assume that $T$ and $P$ are positive. By above, for each $g\in\C(K)$ and $y\in F$ we have $\sigma(y)=1$ and so $P(g)(y)=\left\langle \phi(\rho(y)), g\right\rangle$. As $P$ is positive and $\rho$ is surjective, the latter equality implies that for every $x\in L$ and positive $g\in C(K)$ we have $\left\langle \phi(x), g\right\rangle\ge0$. Consequently, for all $x\in L$ and all Borel sets $B\subseteq K$ it holds $\phi(x)(B)\ge0$, which means that $\phi(x)\in\P(K)$.

\medskip

The proof of the theorem is finished.
\end{proof}

\subsection{Characterization of isometric complemented copies of spaces $\C(L)$ in spaces $\C(K)$\label{sec:holsztynski_separable}}

Without any additional assumption on the space $L$ the converse to Theorem \ref{holsztynski} does not hold, that is, the existence of a closed set $F\subseteq K$, a continuous surjection $\rho\colon F\to L$, and a continuous mapping $\phi\colon L\to S_{\M(K)}$ as in the theorem, does not imply that $\C(L)$ isometrically maps onto a complemented subspace of $\C(K)$---to see this, consider for example the compact spaces $K=[0,1]^{\omega_1}$ and $L=[1,\omega_1]$, where $L$ is even homeomorphic to a closed subspace of $K$, but $\C(L)$ cannot be even isomorphically embedded into $\C(K)$, as $K$ satisfies the countable chain condition, while $L$ does not, cf. \cite[Theorem 1.3]{rondos_cardinal_invariants}. However, we get the converse statement in the case when $L$ is metrizable, which is the content of our next result. We also obtain further equivalent statements, e.g concerning the existence of a PNPP projectional tree of norm $1$ or regarding the existence of certain families of open sets and measures (note that in the case of countable compact spaces similar families were already considered by Alspach \cite{Alspach_operators_on_separable}).

We start with the following auxiliary lemma, which describes a canonical way of constructing a copy of the space $\A_{\Lambda}$ in $\C(K)$-spaces.

\begin{lemma}
\label{lem:tree of continuous functions}
Assume that $K$ is a compact space, $\Lambda \in \TT$, $m \in \en$, and 
\[(f_{s, i})_{(s, i) \in\Sigma_m(\Lambda)} \subseteq \C(K, [0, 1])\] 
is a system of norm-$1$ functions satisfying for each $(s, i), (t, j) \in\Sigma_m(\Lambda)$ the following conditions: 
\begin{itemize}
    \item $\supp(f_{s, i}) \subseteq f_{t,j}^{-1}(1)\ \Leftrightarrow\ (i=j \text{ and } t \preceq s)$, 
    \item $\supp(f_{s, i})\cap \supp(f_{t, j})=\emptyset\ \Leftrightarrow\ (i \neq j \text{ or } s \perp t)$.
\end{itemize}
Then, the mapping $\tilde{T}$ assigning to each sequence $\tilde{\chi}_{s, i}$ the function $f_{s, i}$ extends to a unique PNPP isometric embedding $T\colon \A_{\Lambda}^m \rightarrow \C(K)$.
\end{lemma}

\begin{proof}
As previously, it is enough to consider the case $m=1$. First, the mapping $\tilde{T}$ extends uniquely to an operator $T'$ mapping the space $\mathcal{FS}(\A_\Lambda)$ into $\C(K)$; more precisely, for each $(a_s)_{s\in\Lambda}\in\mathcal{FS}(\A_\Lambda)$ we set
\[T'\big((a_s)_{s\in\Lambda}\big)=\sum_{s\in\Lambda}a_s f_s,\]
so that $T'$ is clearly linear. As in the proof of Lemma \ref{fin-norm}, one may check that $T'$ is actually a PNPP isometry from $\mathcal{FS}(\A_\Lambda)$ into $\C(K)$ such that for each element $\olb{a}=(a_s)_{s \in \Lambda}\in \A_\Lambda$ supported by a trunk $F \in \mathfrak{Tr}(\Lambda)$ it holds 
\[\tag{$*$}\norm{T'(\olb{a})^+}=\norm{\bigg( \sum_{s \in F} a_{s} f_{s} \bigg)^{+}} =\norm{\olb{a}^+}^\Lambda,\]
and so
\[\norm{T'(\olb{a})}=\norm{\sum_{s \in F} a_{s} f_{s}} =\norm{\olb{a}}^\Lambda.\]
Then, since the space $\mathcal{FS}(\A_\Lambda)$ is norm-dense in $\A_\Lambda$ (Lemma \ref{lem:fin-approx}), $T'$ extends in a unique way to an isometric embedding $T\colon \A_\Lambda\to\C(K)$ (cf. Remark \ref{rem:tree sufficient 1}). Moreover, as in the proof of Theorem \ref{A_Lambda is C(2^omega)}, using ($*$) and the continuity of the lattice operations and of $T$, we show that for every $\olb{a}\in\A_\Lambda$ it holds $\norm{\olb{a}^+}^\Lambda=\norm{T(\olb{a})^+}$, that is, that $T$ is a PNPP operator.
\end{proof}

Now, we are ready to prove our last main result. 

\begin{thm}
\label{characterization-almost isometric}
Let $(K,\tau)$ be a compact space and $L$ be a metric compact space. Set $\alpha=ht(L)$ and let $m=\abs{L^{(\alpha-1)}}$ if $\alpha<\omega_1$, and $m=1$ otherwise. Then, the following assertions are equivalent:
\begin{enumerate}[(1),itemsep=1mm]
    \item There exist a closed linear subspace $E$ of $\C(K)$, a surjective positive isometry $T\colon\C(L) \rightarrow E$, and a positive norm-$1$ projection $P\colon\C(K) \rightarrow E$.
    \item There exist a closed subset $F$ of $K$, a continuous surjection $\rho\colon F \rightarrow L$, and a continuous mapping $\phi\colon L \rightarrow \P(F)$ such that $\supp(\phi(x))\subseteq\rho^{-1}(x)$ for each $x\in L$.
    \item There exist a tree $\Lambda\subseteq\el$ of rank $\alpha$ and a system
 \[\Big(\big(U_{s, i}, V_{s, i}, \mu_{s, i}\big)_{s \in \Lambda}\colon\ 1\le i\le m\Big) \subseteq \big(\tau\times\tau\times \P(K)\big)^m\]
    such that the following conditions are satisfied:
    \begin{enumerate}[(a)]
        \item $\overline{U_{\emptyseq,i}}\cap\overline{U_{\emptyseq,j}}=\emptyset$ for every $1\le i\neq j\le m$,
        \item for each $s\neq t\in\Lambda$ and $1\le i\le m$ we have:
        \begin{itemize}
        \item $\overline{V_{s, i}}  \subseteq U_{s, i}$,
	\item if $s\in NL(\Lambda)$, then $\overline{\bigcup_{k\in\omega} U_{s \concat k, i}}\subseteq V_{s, i}$,
        \item if $s\in NL(\Lambda)$, then $\overline{U_{s \concat n, i}} \cap \overline{\bigcup_{k>n} U_{s \concat k, i}}=\emptyset$ for each $n\in\omega$,
        \item $s \prec t\ \Leftrightarrow\ \overline{U_{t, i}} \subseteq V_{s, i}$, 
        \item $\supp(\mu_{s, i}) \subseteq V_{s, i}$,
        \item $s\prec t\ \Rightarrow\ \supp(\mu_{s, i}) \cap \overline{U_{t, i}}=\emptyset$;
        \end{itemize}
        \item the mapping $\Sigma_m(\Lambda)\ni(s, i) \mapsto \mu_{s, i}$ uniformly weak$^*$ continuous.
    \end{enumerate}
    \item $\C(K)$ admits a PNPP projectional tree of rank $\alpha$, order $m$, and norm $1$.
    \item There exist a closed linear subspace $E$ of $\C(K)$, a surjective PNPP isometry $T\colon\C(L) \rightarrow E$, and a positive norm-$1$ projection $P\colon\C(K) \rightarrow E$.
\end{enumerate}
\end{thm}

\begin{proof}
Implication (1)$\Rightarrow$(2) immediately follows from Theorem \ref{holsztynski} and implication (5)$\Rightarrow$(1) is trivial.
\medskip
 
In the proof of implication (2)$\Rightarrow$(3), we distinguish two cases. First, we assume that $L$ is uncountable, so $L$ contains a closed subspace $M$ homeomorphic to $2^{\omega}$. Thus, the set $\rho^{-1}[M] \subset K$ is closed, and hence it follows that without loss of generality we may assume that $L=2^{\omega}$. Let $\Lambda$ be the full tree, i.e. $\Lambda=\el=\omega^{<\omega}$, and $m=1$. We construct the desired system $(U_{s,1}, V_{s,1}, \mu_{s,1})_{s \in \Lambda}=(U_s, V_s, \mu_s)_{s \in \Lambda}$ recursively in such a way that for each $s \in \Lambda$ there is  $t_s\in\el$ with the following properties:
\begin{enumerate}[(i)]
    \item $\mu_s=\phi({R(t_s)})$,
    \item $\rho^{-1}(R(t_s))\subseteq V_s$,
    \item $l(s)=l(t_s)$,
    \item $t_s\prec t_{s\concat n}$ and $n\le t_{s\concat n}(l(s))<t_{s\concat(n+1)}(l(s))$ for any $n\in\omega$.
\end{enumerate}

Let $U_{\emptyseq}=V_{\emptyseq}=K$, and $\mu_{\emptyseq}=\phi({R(\emptyseq)})$, i.e. $t_\emptyseq=\emptyseq$. Now, let $N\in\omega$ and assume that the system
\[\big\{(U_s, V_s, \mu_s)\colon\ s\in\Lambda,\ l(s)\le N\big\}\]
have been constructed in such a way that it satisfies condition (b) as well as that for all $s\in\Lambda$ the objects $(U_{s}, V_{s},\mu_{s})$ and $t_s$ satisfy conditions (i)--(iv) if $l(s)<N$, and conditions (i)--(iii) if $l(s)=N$. 

Let $s\in\Lambda$ be of length $N$. It follows that $\mu_s=\phi({R(t_s)})$ for some $t_s \in \omega^{<\omega}$. Since $\rho^{-1}$ is an upper-semicontinuous set-valued mapping between the spaces of all closed nonempty subsets of $L$ and $F$, both endowed with the Vietoris topologies, and $V_s$ is an open set containing $\rho^{-1}(R(t_s))$ (by (ii)), there exists $n_0 \in \omega$ such that for each $n \ge n_0$ we have $\rho^{-1}(R(t_s\concat n)) \subseteq V_s$ (note that $\lim_{n\to\infty}R(t_s\concat n)=R(t_s)$). For each $n\in\omega$ we define
\[\mu_{s\concat n}=\phi\Big({R\big(t_s\concat (n+n_0)\big)\Big)}\]
i.e. we set $t_{s\concat n}=t_s\concat (n+n_0)$; we have $\supp(\mu_{s\concat n})\subseteq V_s$. So, conditions (i), (iii), and (iv) are satisfied by $s\concat n$ and $t_{s\concat n}$ for each $n\in\omega$.

We need to show that condition (ii) is also satisfied. We construct the open sets $U_{s\concat n}$'s by induction on $n\in\omega$. Firstly, since the sets
\[\big\{R(t_s\concat n_0)\big\}\quad\text{and}\quad\big\{R(t_s\concat n)\colon n >n_0\big\} \cup \{R(t_s)\}\]
are pairwise disjoint closed subsets of $L$, the sets
\[\rho^{-1}\big(R(t_s\concat n_0)\big)\quad\text{and}\quad\rho^{-1}\Big[\big\{R(t_s\concat n)\colon n>n_0\big\} \cup \{R(t_s)\}\Big]\]
are pairwise disjoint closed subsets of $K$ (recall that the set $F$ is closed), and hence there exist disjoint open sets $U_{s\concat 0}$ and $\tilde{U}_{s\concat 0}$ in $K$ which respectively separate the latter preimages. By making the open sets smaller if necessary, we may assume that
\[\overline{U_{s\concat 0}},\overline{\tilde{U}_{s\concat 0}}\subseteq V_s\quad\text{and}\quad\overline{U_{s\concat 0}}\cap\overline{\tilde{U}_{s\concat 0}}=\emptyset.\]

Assume now that $k > 0$ and the sets $U_{s\concat 0},\ldots,U_{s\concat(k-1)}$ and $\tilde{U}_{s\concat 0},\ldots,\tilde{U}_{s\concat(k-1)}$ have been defined in such a way that for every $0\le l<k$ we have:
\begin{itemize}
    \item $\overline{U_{s\concat l}}\cap\overline{\tilde{U}_{s\concat l}}=\emptyset$,
    \item $\overline{U_{s\concat l}},\overline{\tilde{U}_{s\concat l}}\subseteq\tilde{U}_{s\concat(l-1)}$ if $l>0$,
    \item $\rho^{-1}\Big(R\big(t_s\concat(n_0+l)\big)\Big)\subseteq U_{s\concat l}$,
    \item $\rho^{-1}\Big[\big\{R(t_s\concat n)\colon n> n_0+l\big\} \cup \{R(t_s)\}\Big]\subseteq\tilde{U}_{s\concat(l-1)}$.
\end{itemize}
Then, since the sets
\[\Big\{R\big(t_s\concat(n_0+k)\big)\Big\}\quad\text{and}\quad\big\{R(t_s\concat n)\colon n >n_0+k\big\} \cup \{R(t_s)\}\]
are pairwise disjoint closed subsets of $L$, the sets
\[\rho^{-1}\Big(R\big(t_s\concat(n_0+k)\big)\Big)\quad\text{and}\quad\rho^{-1}\Big[\big\{R(t_s\concat n)\colon n>n_0+k\big\} \cup \{R(t_s)\}\Big]\]
are pairwise disjoint closed subsets of $K$ contained in $\tilde{U}_{s\concat(k-1)}$, thus there exist disjoint open sets $U_{s\concat k}$ and $\tilde{U}_{s\concat k}$ in $K$ which respectively separate the latter preimages. By making the open sets smaller if necessary, we may assume that
\[\overline{U_{s\concat k}},\overline{\tilde{U}_{s\concat k}}\subseteq\tilde{U}_{s\concat (k-1)}\quad\text{and}\quad\overline{U_{s\concat k}}\cap\overline{\tilde{U}_{s\concat k}}=\emptyset.\] 
This finishes the inductive construction of the sets $U_{s\concat n}$'s.  

Finally, for each $n \in \omega$ we find an open set $V_{s\concat n}$ in $K$ such that
\[\supp(\mu_{s\concat n})\subseteq\rho^{-1}\big(R(t_{s\concat n})\big)\subseteq V_{s\concat n}\subseteq\overline{V_{s\concat n}} \subseteq U_{s\concat n},\]
so that condition (ii) is satisfied as well.

Directly by its construction, the system $(U_s,V_s,\mu_s)_{s\in\Lambda}$ satisfies conditions (a) and (b). We need to show that it also satisfies condition (c). First, notice that the mapping $\Lambda\ni s\mapsto t_s\in\el$ is uniformly continuous by conditions (iii) and (iv) (recall that the function $R$ was used to define a metrizable topology on $\el$, making $\el$ and $R[\el]$ homeomorphic, and appeal to Lemma \ref{lemma:tree_topology}). Further, the mapping $\phi \colon 2^{\omega} \rightarrow \P(K)$ is continuous, and hence it is uniformly weak$^*$ continuous, as $2^{\omega}$ is compact. Consequently, since for each $s \in \Lambda$ we have $\mu_s=\phi({R(t_s)})$, the mapping $\Lambda\ni s\mapsto\mu_s\in \P(K)$ is uniformly weak$^*$ continuous, too. 

\medskip

Let us now suppose that $L$ is countable, so that $ht(L)=\alpha<\omega_1$ and $m=\abs{L^{(\alpha-1)}}$. We may assume that $L$ is the disjoint union of $m$ copies $L_1,\ldots,L_m$ of the interval $[1,\omega^{\alpha-1}]$. Fix a tree $\Lambda\in\TT_\alpha$ and set $\Lambda_1=\ldots=\Lambda_m=\Lambda$. Fix also a partition $W_1,\ldots,W_m$ of the space $2^\omega$ into $m$ nonempty clopen subsets, and for each $1\le i\le m$ set $R_i=h_i\circ(R\rstr\Lambda_i)$, where $h_i\colon2^\omega\to W_i$ is a homeomorphism. Then, the disjoint union of the mappings $R_1,\ldots,R_m$ may be treated as an embedding of the disjoint union of the trees $\Lambda_1,\ldots,\Lambda_m$ into $2^\omega$. As each $\Lambda_i$ is homeomorphic to $L_i$, it follows that without loss of generality we may also assume that $L$ is in fact a subset of $2^\omega$ such that $L_i\subseteq W_i$ for each $i=1,\ldots,m$, and that $F$ is the disjoint union of the closed sets $F_i=\rho^{-1}[L_i]$, $i=1,\ldots,m$. Then, similarly as in the uncountable case, we construct the desired system $\Big(\big(U_{s,i}, V_{s,i}, \mu_{s,i}\big)_{s \in \Lambda}\colon\ 1\le i\le m\Big)$ recursively in such a way that for each $1\le i\neq j\le m$ we have
\[\overline{U_{\emptyseq,i}}\cap\overline{U_{\emptyseq,j}}=\emptyset\quad\text{and}\quad F_i\subseteq V_{\emptyseq,i}\subseteq\overline{V_{\emptyseq,i}}\subseteq U_{\emptyseq,i},\]
and for each $1\le i\le m$ and each $s \in \Lambda_i$ there is $t_{s,i}\in\Lambda_i$ for which the following conditions hold:
\begin{enumerate}[(i')]
    \item $\mu_{s,i}=\phi\big({R_i(t_{s,i})}\big)$,
    \item $\rho^{-1}\big(R_i(t_{s,i})\big)\subseteq V_{s,i}$,
    \item $l(s)=l(t_{s,i})$,
    \item $t_{s,i}\prec t_{s\concat n,i}$ and $n\le t_{s\concat n,i}(l(s))< t_{s\concat(n+1),i}(l(s))$ for any $n\in\omega$.
\end{enumerate}
The argument that the constructed system satisfies conditions (a)--(c) is then similar to the one from the uncountable case.

\medskip

We now prove implication (3)$\Rightarrow$(4). To this end, for each $1\le i\le m$ and $s \in \Lambda$, we simply find a function $f_{s,i} \in \C(K, [0, 1])$ such that $f_{s, i}(x)=1$ for all $x\in\overline{V_{s,i}}$, and $\supp(f_{s, i})\subseteq U_{s, i}$. Then, we claim that the triple $\T=\big(\Sigma, T, \rho)$, where $\Sigma$ is the disjoint union of $m$ copies of $\Lambda$, $T\colon \A_\Lambda^m\to\C(K)$ is a unique operator mapping each $\tilde{\chi}_{s,i}$ to $f_{s, i}$, and $\rho\colon\Sigma\to\P(K)$ is such that $\rho(s, i)=\mu_{s, i}$ for each $(s,i)\in\Sigma$, is a PNPP projectional tree of rank $\alpha$, order $m$, and norm $1$.

First, notice that the system $(f_{s,i})_{(s,i)\in\Sigma}$ of functions satisfies the assumptions of Lemma \ref{lem:tree of continuous functions}, and so $T$ is a PNPP isometric embedding. The mapping $\rho$ is immediately uniformly weak* continuous by (c). Since for each $(s,i)\in\Sigma$ we also have $\norm{\rho(s,i)}=1$, the tree $\T$ has norm $1$. 

Lastly, we need to show that $T^*\circ\rho=\tilde{\delta}$. Pick arbitrary $s, t \in \Lambda$. For $1\le i\neq j\le m$ we have $\supp(\mu_{s,i})\cap\supp(f_{t,j})=\emptyset$ (by (a)), hence $\langle\mu_{s,i},f_{t,j}\rangle=0=\tilde{\delta}_{s,i}(\tilde{\chi}_{t,j})$. So, fix $1\le i\le m$. If $t \preceq s$, then, since $\supp(\mu_{s,i}) \subseteq V_{s,i}$ and $f_{t,i}$ is constant $1$ on $V_{t,i}$ and hence on $V_{s,i}$, we get $\langle \mu_{s,i}, f_{t,i} \rangle=1=\tilde{\delta}_{s,i}(\tilde{\chi}_{t,i})$, while otherwise, we have $\supp(f_{t,i})\subseteq U_{t,i}$ and $\supp(\mu_{s,i})\cap \overline{U_{t,i}}=\emptyset$, thus $\langle \mu_{s,i}, f_{t,i} \rangle=0=\tilde{\delta}_{s,i}(\tilde{\chi}_{t,i})$. By Remark \ref{rem:tree sufficient 2}, the proof of (3)$\Rightarrow$(4) is complete. 

\medskip

Finally, we establish implication (4)$\Rightarrow$(5). In the case when $L$ is countable or $L=2^{\omega}$, this follows immediately by Theorem \ref{main characterization-zerodim} and Lemma \ref{lem:ordered normed spaces and PNPP}. In the case when $L$ is a general uncountable metric space, so that we have $\alpha=ht(L)=\infty$, again by Theorem \ref{main characterization-zerodim} and Lemma \ref{lem:ordered normed spaces and PNPP}, we obtain a closed linear subspace $E$ of $\C(K)$, a surjective positive norm-$1$ projection $P\colon\C(K)\to E$, and a surjective PNPP isometry $T \colon \C(2^{\omega}) \rightarrow E$. Further, by the result known as Miljutin's lemma (\cite[Theorem 2.4]{RosenthalC(K)}), there exists a continuous surjection of $2^{\omega}$ onto $L$ which admits a regular averaging operator (see the comment following \cite[Definition 2.10]{RosenthalC(K)}), that is, there are a closed linear sublattice $F$ of $\C(2^{\omega})$, a surjective positive norm-$1$ projection $\tilde{P}\colon\C(2^{\omega})\to F$, and a surjective lattice isometry $S \colon \C(L) \rightarrow F$. Since $T$ is PNPP and surjective, the inverse $T^{-1}\colon E\to\C(2^\omega)$ is too a positive isometry (by Lemmas \ref{lem:pnpp_inverse_positive} and \ref{lem:ordered normed spaces and PNPP}). As $S$ is also PNPP, it follows that $T \circ S \colon \C(L) \rightarrow T[F]$ is a PNPP isometry onto the subspace $T[F]$ of $\C(K)$, which is complemented by a surjective positive norm-$1$ projection $T \circ \tilde{P} \circ T^{-1} \circ P$. The proof of the theorem is thus finished.
\end{proof}


Note that the property of the space $K$ described in condition (2) in Theorem \ref{characterization-almost isometric} is weaker than containing a homeomorphic copy of the space $L$, hence the theorem generalizes Folklore Fact from Introduction. Yet another generalization is provided in the below corollary; note that its proof does not use the notions of Dugundji or Miljutin spaces.

\begin{cor}\label{cor:generalized_folklore_fact}
Assume that $K_1, K_2$ are compact spaces such that $K_1$ is homeomorphic to a subset of $K_2$. For any metric compact space $L$, if $\C(K_1)$ contains a positively $1$-complemented subspace which is positively isometric to $\C(L)$, then so $\C(K_2)$ contains such a subspace, too.      
\end{cor}
\begin{proof}
    Let $\tau_1,\tau_2$ denote the topologies of $K_1, K_2$, respectively. Let $h\colon F\to K_1$ be a homeomorphism from a closed subset $F$ of $K_2$ onto $K_1$, and by $\tau_F$ denote the subspace topology of $F$. Assume that $L$ is a compact metric space such that $\C(K_1)$ contains a positively $1$-complemented subspace $E$ which is positively isometric to $\C(L)$. Let thus
    \[\Big(\big(U_{s,i},V_{s,i},\mu_{s,i}\big)_{s\in\Lambda}\colon\ 1\le i\le m\Big)\subseteq\big(\tau_1\times\tau_1\times \P(K_1)\big)^m\]
    be a system like in Theorem \ref{characterization-almost isometric}.(3). For each $1\le i\le m$ and $s\in\Lambda$, set $U_{s,i}'=h^{-1}[U_{s,i}]$ and $V_{s,i}'=h^{-1}[V_{s,i}]$ as well as define the measure $\nu_{s,i}\in \P(K_2)$ by setting $\nu_{s,i}(B)=\mu_{s,i}(h[B\cap F])$  for each Borel subset $B$ of $K_2$---of course, by the natural identification we also have $\nu_{s,i}\in\P(F)$. Then, the system
    \[\Big(\big(U_{s,i}',V_{s,i}',\nu_{s,i}\big)_{s\in\Lambda}\colon\ 1\le i\le m\Big)\subseteq\big(\tau_F\times\tau_F\times \P(F)\big)^m\]
    satisfies conditions (a)--(c) of Theorem \ref{characterization-almost isometric}.(3) for $F$.
    Now, using the normality of $K_2$, we find open subsets $\tilde{U}_{\emptyseq,1},\ldots,\tilde{U}_{\emptyseq,m}$ of $K_2$ such that $\overline{U'_{\emptyseq,i}}\subseteq\tilde{U}_{\emptyseq,i}$ for each $1\le i\le m$ and $\overline{\tilde{U}_{\emptyseq,i}}\cap\overline{\tilde{U}_{\emptyseq,j}}=\emptyseq$ for each $1\le i\neq j\le m$. We also find for each $1\le i\le m$ an open set $\tilde{V}_{\emptyseq,i}$ in $K_2$ such that 
    \[\overline{V_{\emptyseq,i}'} \subseteq \tilde{V}_{\emptyseq,i} \subseteq \overline{\tilde{V}_{\emptyseq,i}} \subseteq \tilde{U}_{\emptyseq,i}.\]
    Further, again using the normality of $K_2$ and working in each $\tilde{V}_{\emptyseq,i}$ independently, for each $i=1,\ldots,m$ we construct by induction on the length of nodes $s\in\Lambda$ a system $\big(\tilde{U}_{s, i}, \tilde{V}_{s, i}\big)_{s \in \Lambda}$ of open subsets of $K_2$ out of the system $\big(U'_{s, i}, V'_{s, i}\big)_{s \in \Lambda}$ in such a way that ultimately the system
    \[\Big(\big(\tilde{U}_{s,i},\tilde{V}_{s,i},\nu_{s,i}\big)_{s\in\Lambda}\colon\ 1\le i\le m\Big)\subseteq\big(\tau_2\times\tau_2\times \P(K_2)\big)^m\]
    also satisfies conditions (a)--(c) of Theorem \ref{characterization-almost isometric}.(3) for $K_2$, and hence, by Theorem \ref{characterization-almost isometric}.(1), the space $\C(K_2)$ contains a positively $1$-complemented subspace which is positively isometric to $\C(L)$.
\end{proof}

Finally, if we consider isometric embeddings onto subspaces which are not necessarily complemented, then as a consequence of the proof of Theorem \ref{characterization-almost isometric} we obtain the following result, supplementary to \cite[Theorem A]{Rondos-Sobota_copies_of_separable_C(L)}.

\begin{thm}
\label{thm:embeddability-non-complemented}
Let $(K,\tau)$ be a compact space and $L$ be a metric compact space. Set $\alpha=ht(L)$ and let $m=\abs{L^{(\alpha-1)}}$ if $\alpha<\omega_1$, and $m=1$ otherwise. Then, the following assertions are equivalent:
\begin{enumerate}[(1),itemsep=1mm]
    \item $\C(K)$ contains an isometric copy of $\C(L)$.
    \item There exist a closed subset $F$ of $K$ and a continuous surjection $\rho\colon F \rightarrow L$.
    \item There exist a tree $\Lambda\subseteq\el$ of rank $\alpha$ and a system
 \[\Big(\big(U_{s, i}, V_{s, i}\big)_{s \in \Lambda}\colon\ 1\le i\le m\Big) \subseteq (\tau\times\tau)^m\]
    such that for every $1\le i\neq j\le m$ it holds $\overline{U_{\emptyseq,i}}\cap\overline{U_{\emptyseq,j}}=\emptyset$, and for each $s\neq t\in\Lambda$ and $1\le i\le m$ we have:
        \begin{itemize}
        \item $\overline{V_{s, i}}  \subseteq U_{s, i}$,
	\item if $s\in NL(\Lambda)$, then $\overline{\bigcup_{k\in\omega} U_{s \concat k, i}}\subseteq V_{s,i}$,
        \item if $s\in NL(\Lambda)$, then $\overline{U_{s \concat n, i}} \cap \overline{\bigcup_{k>n} U_{s \concat k, i}}=\emptyset$ for each $n\in\omega$,
        \item $s \prec t\ \Leftrightarrow\ \overline{U_{t, i}} \subseteq V_{s, i}$, 
        \end{itemize}
    \item There exists a PNPP isometric embedding $T\colon\C(L)\to\C(K)$.
\end{enumerate}
\end{thm}
\begin{proof}
Implication (1)$\Rightarrow$(2) follows from Holszty\'nski's theorem. The arguments for implications (2)$\Rightarrow$(3) and (3)$\Rightarrow$(4) were provided in the proof of Theorem \ref{characterization-almost isometric}. Finally, the implication (4)$\Rightarrow$(1) is trivial.
\end{proof}

\bibliography{bib}\bibliographystyle{siam}

@book {lacey,
    AUTHOR = {Lacey, H.E.},
     TITLE = {The {I}sometric {T}heory of {C}lassical {B}anach {S}paces},
    SERIES = {Die Grundlehren der mathematischen Wissenschaften},
    VOLUME = {208},
 PUBLISHER = {Springer-Verlag},
   ADDRESS = {New York},
      YEAR = {1974},
     PAGES = {x+270},
   MRCLASS = {46EXX (46BXX)},
  MRNUMBER = {0493279 (58 \#12308)},
MRREVIEWER = {James Hagler},
}

@article {lusp,
    AUTHOR = {Ludv{\'{\i}}k, P. and Spurn{\'y}, J.},
     TITLE = {Descriptive properties of elements of biduals of {B}anach spaces},
   JOURNAL = {Studia Math.},
  FJOURNAL = {Studia Mathematica},
    VOLUME = {209},
      YEAR = {2012},
    NUMBER = {1},
     PAGES = {71--99},
      ISSN = {0039-3223},
   MRCLASS = {46B10 (26A21 54H05)},
  MRNUMBER = {2914930},
MRREVIEWER = {Vicente Montesinos Santaluc{\'{\i}}a},
       DOI = {10.4064/sm209-1-6},
       URL = {http://dx.doi.org/10.4064/sm209-1-6},
}

@book {kechris,
    AUTHOR = {Kechris, A. S.},
     TITLE = {Classical {D}escriptive {S}et {T}heory},
    SERIES = {Graduate Texts in Mathematics},
    VOLUME = {156},
 PUBLISHER = {Springer-Verlag},
   ADDRESS = {New York},
      YEAR = {1995},
     PAGES = {xviii+402},
      ISBN = {0-387-94374-9},
   MRCLASS = {03E15 (03-01 03-02 04A15 28A05 54H05 90D44)},
  MRNUMBER = {1321597 (96e:03057)},
MRREVIEWER = {Jakub Jasi{\'n}ski},
       DOI = {10.1007/978-1-4612-4190-4},
       URL = {http://dx.doi.org/10.1007/978-1-4612-4190-4},
}

@book {lmns,
    AUTHOR = {Luke{\v{s}}, J. and Mal{\'y}, J. and Netuka, I. and Spurn{\'y}, J.},
     TITLE = {Integral {R}epresentation {T}heory. Applications to {C}onvexity, {B}anach {S}paces and {P}otential {T}heory},
    SERIES = {de Gruyter Studies in Mathematics},
    VOLUME = {35},
 PUBLISHER = {Walter de Gruyter \& Co.},
   ADDRESS = {Berlin},
      YEAR = {2010},
     PAGES = {xvi+715},
      ISBN = {978-3-11-020320-2},
   MRCLASS = {46-02 (28A05 31B05 31B20 46A55 52A07 54H05)},
  MRNUMBER = {2589994 (2011e:46002)},
MRREVIEWER = {Wolfhard Hansen},
}

@article {dosp,
	AUTHOR = {Dost{\'a}l, P and Spurn{\'y}, J.},
	TITLE = {The minimum principle for affine functions and isomorphisms of continuous affine function spaces},
	FJOURNAL = {Archiv der Mathematik},
	JOURNAL = {Arch. Math.},
	VOLUME = {114},
	YEAR = {2020},
	NUMBER = {1},
	PAGES = {61-70},
	ISSN = {0236-5294},
	MRCLASS = {54H05 (28A05)},
	MRNUMBER = {2725834 (2011k:54047)},
	MRREVIEWER = {Pandelis Dodos},
	DOI = {10.1007/s10474-010-9223-6},
	URL = {http://dx.doi.org/10.1007/s10474-010-9223-6},
}

@article {spurny-amh,
    AUTHOR = {Spurn{\'y}, J.},
     TITLE = {Borel sets and functions in topological spaces},
   JOURNAL = {Acta Math. Hungar.},
  FJOURNAL = {Acta Mathematica Hungarica},
    VOLUME = {129},
      YEAR = {2010},
    NUMBER = {1-2},
     PAGES = {47--69},
      ISSN = {0236-5294},
   MRCLASS = {54H05 (28A05)},
  MRNUMBER = {2725834 (2011k:54047)},
MRREVIEWER = {Pandelis Dodos},
       DOI = {10.1007/s10474-010-9223-6},
       URL = {http://dx.doi.org/10.1007/s10474-010-9223-6},
}

@article {koumou,
    AUTHOR = {Koumoullis, George},
     TITLE = {A generalization of functions of the first class},
   JOURNAL = {Topol. Appl.},
  FJOURNAL = {Topology and its Applications},
    VOLUME = {50},
      YEAR = {1993},
    NUMBER = {3},
     PAGES = {217--239},
      ISSN = {0166-8641},
     CODEN = {TIAPD9},
   MRCLASS = {26A21 (54H05)},
  MRNUMBER = {1227551 (94e:26008)},
MRREVIEWER = {Pavel Kostyrko},
       DOI = {10.1016/0166-8641(93)90022-6},
       URL = {http://dx.doi.org/10.1016/0166-8641(93)90022-6},
}

@book {dalesbookdual,
    AUTHOR  = {Dales, H. G. and Dashiell, Jr., F. K. and Lau, A. T.-M.},
    TITLE   = {Banach Spaces of Continuous Functions as Dual Spaces},
    PUBLISHER   = {Springer Nature},
    YEAR    = {2016},
    ADDRESS = {Cham}
}

@book{semadeni,
	title={Banach Spaces of Continuous Functions, Vol. 1},
	author={Semadeni, Z.},
	lccn={72176494},
	url={https://books.google.cz/books?id=qX9sAAAAMAAJ},
	year={1971},
	publisher={PWN---Polish Scientific Publishers},
    address="Warsaw"
}

@article{Gordon3,
	title={On the distance coefficient between isomorphic function spaces},
	author={Yehoram Gordon},
	journal={Israel J. Math.},
	year={1970},
	volume={8},
	pages={391-397}
}

@article {Pelczynski_separable,
    AUTHOR = {Pe{\l}czy\'nski, A.},
     TITLE = {On {$C(S)$}-subspaces of separable {B}anach spaces},
   JOURNAL = {Studia Math.},
  FJOURNAL = {Studia Mathematica},
    VOLUME = {31},
      YEAR = {1968},
     PAGES = {513--522},
      ISSN = {0039-3223,1730-6337},
   MRCLASS = {46.10},
  MRNUMBER = {234261},
MRREVIEWER = {H.\ Nakano},
       DOI = {10.4064/sm-31-5-513-522},
       URL = {https://doi.org/10.4064/sm-31-5-513-522},
}

@article{BessagaPelcynski_classification,
	author = {Bessaga, C. and Pełczyński, A.},
	journal = {Studia Math.},
	language = {eng},
	number = {1},
	pages = {53-62},
	title = {Spaces of continuous functions {(IV)} ({O}n isomorphical classification of spaces of continuous functions)},
	url = {http://eudml.org/doc/216968},
	volume = {19},
	year = {1960},
}

@incollection{Diestel_Grothendieck,
    AUTHOR = {Diestel, J.},
     TITLE = {Grothendieck spaces and vector measures},
 BOOKTITLE = {Vector and {O}perator {V}alued {M}easures and {A}pplications ({P}roc.
              {S}ympos., {A}lta, {U}tah, 1972)},
     PAGES = {97--108},
 PUBLISHER = {Academic Press, New York--London},
      YEAR = {1973},
   MRCLASS = {46G10 (46B99)},
  MRNUMBER = {338774},
MRREVIEWER = {W.\ J.\ Davis},
}

@article{Pelczynski_book_1968,
author = {A. Pełczyński},
language = {eng},
location = {Warszawa},
publisher = {Instytut Matematyczny Polskiej Akademii Nauk},
title = {Linear extensions, linear averagings, and their applications to linear topological classification of spaces of continuous functions},
url = {http://eudml.org/doc/268522},
journal = {Diss. Math.},
fjournal = {Dissertationes Mathematicae (Rozprawy Matematyczne)},
year = {1968},
volume = {58},
pages = {1--88},
}

@article{rondos_cardinal_invariants,
	title = "On isomorphisms of {$\mathcal{C}(K)$} spaces and cardinal invariants of derivatives of {K}",
	JOURNAL = {Israel J. Math},
        FJOURNAL = {Israel Journal of Mathematics},
        Volume ={267}, 
        PAGES = {871--899},
        YEAR = {2025},
	author = {Rondo\v{s}, Jakub},
}

@article{plebanek_rondos_sobota_products,    
    AUTHOR = {Plebanek, Grzegorz and Rondo\v{s}, Jakub and Sobota, Damian},
     TITLE = {Complemented subspaces of {B}anach spaces {$\mathcal{C}(K\times L)$}},
   JOURNAL = {J. Funct. Anal.},
  FJOURNAL = {Journal of Functional Analysis},
    VOLUME = {290},
      YEAR = {2026},
    NUMBER = {2},
     PAGES = {20, paper no. 111236},
      ISSN = {0022-1236,1096-0783},
   MRCLASS = {46B20 (28A33 46B26 46E15)},
  MRNUMBER = {4973608},
       DOI = {10.1016/j.jfa.2025.111236},
       URL = {https://doi.org/10.1016/j.jfa.2025.111236}
}

@article {Galego-Villamizar_continuous_maps,
    AUTHOR = {Galego, El\'{o}i Medina and Rinc\'{o}n-Villamizar, Michael A.},
     TITLE = {Continuous maps induced by embeddings of {$C_0(K)$} spaces
              into {$C_0(S,X)$} spaces},
   JOURNAL = {Monatsh. Math.},
  FJOURNAL = {Monatshefte f\"{u}r Mathematik},
    VOLUME = {186},
      YEAR = {2018},
    NUMBER = {1},
     PAGES = {37--47},
      ISSN = {0026-9255,1436-5081},
   MRCLASS = {46B03 (46B25 46E15 46E40)},
  MRNUMBER = {3788225},
MRREVIEWER = {Raymond\ H.\ Cox},
       DOI = {10.1007/s00605-016-1014-x},
       URL = {https://doi.org/10.1007/s00605-016-1014-x},
}

@article{Pelczynski_Semadeni_scattered,
author = {Pełczyński, A. and Semadeni, Z.},
journal = {Studia Math.},
fjournal = {Studia Mathematica},
language = {eng},
number = {2},
pages = {211-222},
title = {Spaces of continuous functions ({III}) ({Spaces} {${C}(\Omega)$} for {$\Omega$} without perfect subsets)},
url = {http://eudml.org/doc/216936},
volume = {18},
year = {1959},
}

@article{Botelho_Jamison_complex_Cambern-Holsztynski,
  title={Algebraic reflexivity of {$C(X,E)$} and {C}ambern's theorem},
  author={Fernanda Botelho and James E. Jamison},
  fjournal={Studia Mathematica},
  journal = {Studia Math.},
  year={2008},
  volume={186},
  pages={295-302},
  url={https://api.semanticscholar.org/CorpusID:121324376}
}

@article{Holsztynski1966,
author = {Holsztyński, W.},
journal = {Studia Math.},
fjournal = {Studia Mathematica},
language = {eng},
number = {2},
pages = {133-136},
title = {Continuous mappings induced by isometries of spaces of continuous functions},
url = {http://eudml.org/doc/217141},
volume = {26},
year = {1966},
}

@article {Bourgain1979,
    AUTHOR = {Bourgain, J.},
     TITLE = {The {S}zlenk index and operators on {$C(K)$}-spaces},
   JOURNAL = {Bull. Soc. Math. Belg. S\'{e}r. B},
  FJOURNAL = {Bulletin de la Soci\'{e}t\'{e} Math\'{e}matique de Belgique.
              S\'{e}rie B},
    VOLUME = {31},
      YEAR = {1979},
    NUMBER = {1},
     PAGES = {87--117},
      ISSN = {0037-9476},
   MRCLASS = {46B20 (46E99 47B38)},
  MRNUMBER = {592664},
}

@article {Dilworth-et-al,
    AUTHOR = {Dilworth, S. J. and Odell, E. and Schlumprecht, T. and Zs\'ak,
              A.},
     TITLE = {Coefficient quantization in {B}anach spaces},
   JOURNAL = {Found. Comput. Math.},
  FJOURNAL = {Foundations of Computational Mathematics. The Journal of the Society for the Foundations of Computational Mathematics},
    VOLUME = {8},
      YEAR = {2008},
    NUMBER = {6},
     PAGES = {703--736},
      ISSN = {1615-3375,1615-3383},
   MRCLASS = {46B20 (41A45 46B15 94A15)},
  MRNUMBER = {2461244},
MRREVIEWER = {David\ Yost},
       DOI = {10.1007/s10208-007-9002-0},
       URL = {https://doi.org/10.1007/s10208-007-9002-0},
}

@article {Grothendieck,
    AUTHOR = {Grothendieck, A.},
     TITLE = {Sur les applications lin\'eaires faiblement compactes
              d'espaces du type {$C(K)$}},
   JOURNAL = {Canad. J. Math.},
  FJOURNAL = {Canadian Journal of Mathematics},
    VOLUME = {5},
      YEAR = {1953},
     PAGES = {129--173},
      ISSN = {0008-414X,1496-4279},
   MRCLASS = {46.3X},
  MRNUMBER = {58866},
MRREVIEWER = {L.\ Nachbin},
       DOI = {10.4153/cjm-1953-017-4},
       URL = {https://doi.org/10.4153/cjm-1953-017-4},
}

@article{Gonzalez_Kania,
    author = "Gonz\'{a}lez, Manuel and Kania, Tomasz",
    title = "Grothendieck spaces: the landscape and perspectives",
    journal = "Japan. J. Math.",
    fjournal = "Japananese Journal of Mathematics",
    volume = "16",
    year = "2021",
    number = "2",
    pages = "247--313"
}

@article {Alspach_operators_on_separable,
    AUTHOR = {Alspach, Dale E.},
     TITLE = {{$C(\alpha )$} preserving operators on separable {B}anach
              spaces},
   JOURNAL = {J. Funct. Anal.},
  FJOURNAL = {Journal of Functional Analysis},
    VOLUME = {45},
      YEAR = {1982},
    NUMBER = {2},
     PAGES = {139--168},
      ISSN = {0022-1236},
   MRCLASS = {46B25 (46E25 47B38)},
  MRNUMBER = {647068},
MRREVIEWER = {Ivan\ Singer},
       DOI = {10.1016/0022-1236(82)90015-5},
       URL = {https://doi.org/10.1016/0022-1236(82)90015-5},
}

@article {Plebanek_Alarcon_CSP,
    AUTHOR = {Plebanek, Grzegorz and Salguero Alarc\'{o}n, Alberto},
     TITLE = {The complemented subspace problem for {$C(K)$}-spaces: a
              counterexample},
   JOURNAL = {Adv. Math.},
  FJOURNAL = {Advances in Mathematics},
    VOLUME = {426},
      YEAR = {2023},
     PAGES = {20, paper no. 109103},
      ISSN = {0001-8708,1090-2082},
   MRCLASS = {46E15 (46B03 46B25 54G12)},
  MRNUMBER = {4592269},
       DOI = {10.1016/j.aim.2023.109103},
       URL = {https://doi.org/10.1016/j.aim.2023.109103},
}

@article {Cembranos_complemented_c0,
    AUTHOR = {Cembranos, Pilar},
     TITLE = {{$C(K,\,E)$} contains a complemented copy of {$c\sb{0}$}},
   JOURNAL = {Proc. Amer. Math. Soc.},
  FJOURNAL = {Proceedings of the American Mathematical Society},
    VOLUME = {91},
      YEAR = {1984},
    NUMBER = {4},
     PAGES = {556--558},
      ISSN = {0002-9939,1088-6826},
   MRCLASS = {46B25 (46E40)},
  MRNUMBER = {746089},
MRREVIEWER = {Paulette\ Saab},
       DOI = {10.2307/2044800},
       URL = {https://doi.org/10.2307/2044800},
}

@book {KKLPS,
    AUTHOR = {K\k{a}kol, Jerzy and Kubi\'s, Wies{\l}aw and
              L\'opez-Pellicer, Manuel and Sobota, Damian},
     TITLE = {Descriptive {T}opology in {S}elected {T}opics of {F}unctional
              {A}nalysis---{U}pdated and {E}xpanded},
    SERIES = {Developments in Mathematics},
    VOLUME = {24},
   EDITION = {Second},
 PUBLISHER = {Springer, Cham},
      YEAR = {2025},
     PAGES = {xv+700},
      ISBN = {978-3-031-76061-7; 978-3-031-76062-4},
   MRCLASS = {46-02 (54-02)},
  MRNUMBER = {4881974},
       DOI = {10.1007/978-3-031-76062-4},
       URL = {https://doi.org/10.1007/978-3-031-76062-4},
}

@article {Seever,
    AUTHOR = {Seever, G. L.},
     TITLE = {Measures on {$F$}-spaces},
   JOURNAL = {Trans. Amer. Math. Soc.},
  FJOURNAL = {Transactions of the American Mathematical Society},
    VOLUME = {133},
      YEAR = {1968},
     PAGES = {267--280},
      ISSN = {0002-9947,1088-6850},
   MRCLASS = {46.25 (06.00)},
  MRNUMBER = {226386},
MRREVIEWER = {S.\ Newberger},
       DOI = {10.2307/1994941},
       URL = {https://doi.org/10.2307/1994941},
}

@article {Lindenstrauus_ell_infty,
    AUTHOR = {Lindenstrauss, Joram},
     TITLE = {On complemented subspaces of {$m$}},
   JOURNAL = {Israel J. Math.},
  FJOURNAL = {Israel Journal of Mathematics},
    VOLUME = {5},
      YEAR = {1967},
     PAGES = {153--156},
      ISSN = {0021-2172},
   MRCLASS = {46.10},
  MRNUMBER = {222616},
MRREVIEWER = {M.\ S.\ Ramanujan},
       DOI = {10.1007/BF02771101},
       URL = {https://doi.org/10.1007/BF02771101},
}

@article {Wolfr-Gebeily_isometries,
    AUTHOR = {El-Gebeily, Mohamad and Wolfe, John},
     TITLE = {Isometries of the disc algebra},
   JOURNAL = {Proc. Amer. Math. Soc.},
  FJOURNAL = {Proceedings of the American Mathematical Society},
    VOLUME = {93},
      YEAR = {1985},
    NUMBER = {4},
     PAGES = {697--702},
      ISSN = {0002-9939,1088-6826},
   MRCLASS = {46J15 (30H05)},
  MRNUMBER = {776205},
MRREVIEWER = {Richard\ Rochberg},
       DOI = {10.2307/2045547},
       URL = {https://doi.org/10.2307/2045547},
}

@article {Araujo-Font_isometries,
    AUTHOR = {Araujo, Jes\'us and Font, Juan J.},
     TITLE = {Linear isometries between subspaces of continuous functions},
   JOURNAL = {Trans. Amer. Math. Soc.},
  FJOURNAL = {Transactions of the American Mathematical Society},
    VOLUME = {349},
      YEAR = {1997},
    NUMBER = {1},
     PAGES = {413--428},
      ISSN = {0002-9947,1088-6850},
   MRCLASS = {46E15 (46E25)},
  MRNUMBER = {1373627},
MRREVIEWER = {Francesco\ Altomare},
       DOI = {10.1090/S0002-9947-97-01713-3},
       URL = {https://doi.org/10.1090/S0002-9947-97-01713-3},
}

@article {Galindo-Palacios_isometries,
    AUTHOR = {Moreno Galindo, Antonio and Rodr\'iguez Palacios, \'Angel},
     TITLE = {A bilinear version of {H}olszty\'nski's theorem on isometries
              of {$C(X)$}-spaces},
   JOURNAL = {Studia Math.},
  FJOURNAL = {Studia Mathematica},
    VOLUME = {166},
      YEAR = {2005},
    NUMBER = {1},
     PAGES = {83--91},
      ISSN = {0039-3223,1730-6337},
   MRCLASS = {46E15 (46B04)},
  MRNUMBER = {2108320},
MRREVIEWER = {Krzysztof\ Jarosz},
       DOI = {10.4064/sm166-1-6},
       URL = {https://doi.org/10.4064/sm166-1-6},
}

@article {Chu-Michael_isometries,
    AUTHOR = {Chu, Cho-Ho and Mackey, Michael},
     TITLE = {Isometries between {${\rm JB}^*$}-triples},
   JOURNAL = {Math. Z.},
  FJOURNAL = {Mathematische Zeitschrift},
    VOLUME = {251},
      YEAR = {2005},
    NUMBER = {3},
     PAGES = {615--633},
      ISSN = {0025-5874,1432-1823},
   MRCLASS = {46G20 (17C65 46L70 58B20)},
  MRNUMBER = {2190348},
MRREVIEWER = {J.\ M.\ Isidro},
       DOI = {10.1007/s00209-005-0826-5},
       URL = {https://doi.org/10.1007/s00209-005-0826-5},
}

@article{Argyros-Arvanitakis_averaging_operators,
  title={A characterization of regular averaging operators and its consequences},
  author={Spiros A. Argyros and Alexander D. Arvanitakis},
  fjournal={Studia Mathematica},
  journal = {Studia Math.},
  year={2002},
  volume={151},
  pages={207-226},
  url={https://api.semanticscholar.org/CorpusID:73690538}
}

@article{Pelczynski_projections,
author = {Pełczyński, A.},
fjournal = {Studia Mathematica},
journal = {Studia Math.},
language = {eng},
number = {2},
pages = {209-228},
title = {Projections in certain {B}anach spaces},
url = {http://eudml.org/doc/216957},
volume = {19},
year = {1960},
}

@article{benyamini1978extension,
  title={An extension theorem for separable {B}anach spaces},
  author={Benyamini, Y},
  fjournal={Israel Journal of Mathematics},
  journal = {Israel J. Math.},
  volume={29},
  number={1},
  pages={24--30},
  year={1978},
  publisher={Springer}
}

@article{JohnsonKaniaSchchtman2016complementedcountablesuport,
  title={Closed ideals of operators on and complemented subspaces of {B}anach spaces of functions with countable support},
  author={Johnson, William and Kania, Tomasz and Schechtman, Gideon},
  fjournal={Proceedings of the American Mathematical Society},
  journal={Proc. Amer. Math. Soc.},
  volume={144},
  number={10},
  pages={4471--4485},
  year={2016}
}

@article{koszmider2011complementation,
  title={Complementation and decompositions in some weakly {L}indel{\"o}f {B}anach spaces},
  author={Koszmider, Piotr and Zieli{\'n}ski, Przemys{\l}aw},
  fjournal={Journal of Mathematical Analysis and Applications},
  journal={J. Math. Anal. Appl.},
  volume={376},
  number={1},
  pages={329--341},
  year={2011},
  publisher={Elsevier}
}

@article{koszmider2021banach,
  title={A {B}anach space induced by an almost disjoint family, admitting only few operators and decompositions},
  author={Koszmider, Piotr and Laustsen, Niels Jakob},
  journal={Adv. Math.},
  fjournal={Advances in Mathematics},
  volume={381},
  pages={39, paper no. 107613},
  year={2021},
  publisher={Elsevier}
}

@article{koszmider2005decompositions,
  title={On decompositions of {B}anach spaces of continuous functions on {M}r{\'o}wka’s spaces},
  author={Koszmider, Piotr},
  fjournal={Proceedings of the American Mathematical Society},
  journal={Proc. Amer. Math. Soc.},
  volume={133},
  number={7},
  pages={2137--2146},
  year={2005}
}

@book{schlumprecht_phd_thesis,
  title={Limitierte Mengen in Banachr{\"a}umen},
  author={Schlumprecht, Thomas},
  year={1988},
  series = {PhD Thesis},
  publisher={Ludwig-Maximilians-Universit\"{a}t},
  address = {Munich}
}

@book{cembranos2006banach,
  title={Banach {S}paces of {V}ector-{V}alued {F}unctions},
  author={Cembranos, Pilar and Mendoza, Jos{\'e}},
  year={2006},
  SERIES = {Lecture Notes in Mathematics},
  VOLUME = {1676},
  publisher={Springer Berlin},
  address = {Heidelberg}
}

@article{Sobczyk_c0,
  title={Projection of the space ($m$) on its subspace ($c_0$)},
  author={Andrew Sobczyk},
  fjournal={Bulletin of the American Mathematical Society},
  journal={Bull. Amer. Math. Soc.},
  volume={47},
  pages={938--947},
  year={1941}
}

@article{Zippin1977_separable,
  title={The separable extension problem},
  author={M. Zippin},
  journal={Israel J. Math.},
  year={1977},
  volume={26},
  pages={372-387},
  url={https://api.semanticscholar.org/CorpusID:123150206}
}

@article{CORREA_compact_lines_sobczyk,
title = {Compact lines and the {S}obczyk property},
journal = {J. Funct. Anal.},
volume = {266},
number = {9},
pages = {5765-5778},
year = {2014},
issn = {0022-1236},
doi = {https://doi.org/10.1016/j.jfa.2014.02.007},
url = {https://www.sciencedirect.com/science/article/pii/S0022123614000615},
author = {Claudia Correa and Daniel V. Tausk}
}

@book{hajek2008biorthogonal,
  title     = {Biorthogonal Systems in {B}anach Spaces},
  author    = {H\'{a}jek, Petr and Montesinos Santaluc\'{i}a, Vicente and Vanderwerff, Jon and Zizler, V\'{a}clav},
  series    = {CMS Books in Mathematics},
  publisher = {Springer},
  address   = {New York, NY},
  year      = {2008},
  isbn      = {978-0-387-68914-2},
  doi       = {10.1007/978-0-387-68915-9}
}

@article{FriedmanRusso1982,
  author  = {Friedman, Yaakov and Russo, Bernard},
  title   = {Contractive projections on {$C_0(K)$}},
  journal = {Trans. Amer. Math. Soc.},
  year    = {1982},
  volume  = {273},
  number  = {1},
  pages   = {57--73},
  doi     = {10.1090/S0002-9947-1982-0664030-4},
  url     = {www.ams.org}
}

@article{peck1976lattice,
  title={Lattice projections on continuous function spaces},
  author={Peck, Emily Mann},
  journal = {Pacific J. Math.},
  fjournal={Pacific Journal of Mathematics},
  volume={66},
  number={2},
  pages={477--489},
  year={1976},
  publisher={Pacific Journal of Mathematics, A Non-profit Corporation}
}
\end{document}